\documentclass[11pt]{amsart}

\usepackage{amsmath,amscd,amssymb,amsthm,mathtools,enumerate,wasysym}
\usepackage{graphicx,subfig,float}
\usepackage{xcolor}
\usepackage[pdfencoding=auto]{hyperref} 

\usepackage{verbatim}

\hypersetup{
     colorlinks,
     linkcolor= {blue!80!black},
     citecolor={blue!80!black},
     urlcolor={blue!80!black}}
     
\usepackage[nameinlink,capitalise]{cleveref}

\usepackage[a4paper,left=3cm,right=3cm,top=4.5cm,bottom=4.5cm]{geometry}

\usepackage{parskip}
\usepackage[most]{tcolorbox} %colored boxes

\usepackage{tikz-cd} %diagrams

\newcommand{\nwc}{\newcommand}

\nwc{\mf}{\mathbf} %Latex (as in \bf not tilted math letters)
\nwc{\blds}{\boldsymbol} %Latex 
\nwc{\ml}{\mathcal} %Latex

\newcommand{\mB}{\mathrm{B}} % Born B
\newcommand{\C}{\mathbb{C}} 
\newcommand{\F}{\mathcal{F}} % fourier
\newcommand{\gH}{\mathfrak{H}} % spherical harmonics
\newcommand{\N}{\mathbb{N}}
\newcommand{\R}{\mathbb{R}}
\newcommand{\Z}{\mathbb{Z}}  
\newcommand{\gm}{\mathfrak{m}}

\renewcommand{\d}[1]{\mathrm{d}#1}
\newcommand{\abs}[1]{\left\lvert #1 \right\rvert}
\newcommand{\norm}[1]{\left\lVert #1 \right\rVert}

\newcommand{\cc}[1]{\overline{#1}}
\newcommand{\hp}[2]{\left\langle #1 ,#2\right\rangle}

\newcommand{\br}[1]{\left \langle #1 \right \rangle} 
 
\newcommand{\ol}[1]{\overline{#1}}

\newcommand{\p}{\partial}
\newcommand{\cp}{\overline{\partial}}

\newcommand{\To}{\longrightarrow} 

\newcommand{\rad}{\mathrm{rad}}
\newcommand{\loc}{\mathrm{loc}}

\newcommand{\qe}{q^{\mathrm{B}}} 
 
\newcommand{\gb}{\gamma^{\mathrm{B}}}

\DeclareMathOperator{\Id}{\mathrm{Id}}

\DeclareMathOperator{\Div}{\mathrm{div}}

\DeclareMathOperator{\Spec}{\mathrm{Spec}}

\newcommand{\ka}{\kappa}
\newcommand{\la}{\lambda} 
\newcommand{\La}{\Lambda} 
\newcommand{\ga}{\gamma} 

\newcommand{\Om}{\Omega}
\newcommand{\si}{\sigma}
\newcommand{\eps}{\varepsilon}

\nwc{\IA}{\mathbb{A}} %algebraic
\nwc{\IB}{\mathbb{B}} %ball
\nwc{\IC}{\mathbb{C}} %complex
\nwc{\ID}{\mathbb{D}} %Dedekind
\nwc{\IE}{\mathbb{E}} %Euklides
\nwc{\IF}{\mathbb{F}} %finite field
\nwc{\IG}{\mathbb{G}} %Gauss
\nwc{\IH}{\mathbb{H}} %Hilbert\N-subgroup
\nwc{\IN}{\mathbb{N}} %natural
\nwc{\IP}{\mathbb{P}} %prime
\nwc{\IQ}{\mathbb{Q}} %rational
\nwc{\IR}{\mathbb{R}} %real
\nwc{\IS}{\mathbb{S}} %sphere
\nwc{\IT}{\mathbb{T}} %torus
\nwc{\IZ}{\mathbb{Z}} %integers

\nwc{\cA}{\ml{A}}
\nwc{\cB}{\ml{B}}
\nwc{\cC}{\ml{C}}
\nwc{\cD}{\ml{D}}
\nwc{\cE}{\ml{E}}
\nwc{\cF}{\ml{F}}
\nwc{\cG}{\ml{G}}
\nwc{\cH}{\ml{H}}
\nwc{\cI}{\ml{I}}
\nwc{\cJ}{\ml{J}}
\nwc{\cK}{\ml{K}}
\nwc{\cL}{\ml{L}}
\nwc{\cM}{\ml{M}}
\nwc{\cN}{\ml{N}}
\nwc{\cO}{\ml{O}}
\nwc{\cP}{\ml{P}}
\nwc{\cQ}{\ml{Q}}
\nwc{\cR}{\ml{R}}
\nwc{\cS}{\ml{S}}
\nwc{\cT}{\ml{T}}
\nwc{\cU}{\ml{U}}
\nwc{\cV}{\ml{V}}
\nwc{\cW}{\ml{W}}
\nwc{\cX}{\ml{X}}
\nwc{\cY}{\ml{Y}}
\nwc{\cZ}{\ml{Z}}

\newtheorem{main-theorem}{Theorem}
\newtheorem{proposition}{Proposition}[section]

\newtheorem{lemma}[proposition]{Lemma}
\newtheorem{theorem}[proposition]{Theorem}

\theoremstyle{remark}
\newtheorem{remark}[proposition]{Remark}

\numberwithin{equation}{section}

\begin{document}
\title[Characterization of DtN maps via the Born approximation]{Characterization of Dirichlet-to-Neumann maps via the Born approximation}

\author[C. Castro]{Carlos Castro}
\address[CC]{M$^2$ASAI. Universidad Politécnica de Madrid, ETSI Caminos, c. Profesor Aranguren s/n, 28040, Madrid, Spain.}
\email{carlos.castro@upm.es}

\author[F. Macià]{Fabricio Macià}
\address[FM]{M$^2$ASAI. Universidad Politécnica de Madrid, ETSI Navales, Avda. de la Memoria, 4, 28040, Madrid, Spain.}
\email{fabricio.macia@upm.es}

\author[C. Meroño]{Cristóbal Meroño}
\address[CM]{M$^2$ASAI. Universidad Politécnica de Madrid, ETSI Navales, Avda. de la Memoria, 4, 28040, Madrid, Spain.}
\email{cj.merono@upm.es}

\author[D. Sánchez-Mendoza]{Daniel Sánchez-Mendoza}
\address[DSM]{M$^2$ASAI. Universidad Politécnica de Madrid, ETSI Navales, Avda. de la Memoria, 4, 28040, Madrid, Spain.}
\email{daniel.sanchezmen@upm.es}

%%%%%%%%%%%%%%%%%%%%%%%%%%%%%%%%%%%%%%%%%%%%%%%%%%%%%%%%%
%%%%%%%%%%%%%%%%%%%%%%%%%%%%%%%%%%%%%%%%%%%%%%%%%%%%%%%%%

\begin{abstract}
    The problem of identifying the set of Dirichlet-to-Neumann (DtN) maps arising from conductivities on a smooth domain, among operators acting on functions on the boundary, is a challenging issue in the mathematical analysis of the Calderón inverse problem. This question is also relevant in specific applications since, as the inverse problem is ill-posed, numerically reconstructing a conductivity from the knowledge of its DtN map is particularly delicate. In this article, we address this issue by proving that any DtN map arising from a radial conductivity in the unit ball of  $\R^d$ admits an exact representation as a linearized DtN map for a uniquely determined integrable function, that we call the Born approximation. This gives a strong necessary condition for an operator to be a DtN map arising from a radial conductivity. 
    In particular, our results are a starting point towards developing a rigorous foundation for numerous linearization-based methods that are commonly used in the numerical solution of the Calderón inverse problem.
    We also characterize the Born approximation as a solution to a generalized moment problem that is formally well-defined even for non-radial conductivities. We investigate the uniqueness and structure of general non-radial solutions to this moment problem on the unit disk and provide an algorithm to numerically reconstruct the Born approximation in this setting. 
    We provide numerical experiments to test the resolution and robustness of the Born approximation in different situations. 
\end{abstract}

\maketitle

\section{Introduction}

\subsection{Setting and results}

Calderón's inverse problem asks whether it is possible to determine the electrical conductivity of a medium by making voltage and current measurements on its boundary. Its mathematical formulation involves the conductivity equation, which is the elliptic PDE on a domain $\Omega$ in Euclidean space with conductivity $\ga$ (a positive function defined on $\Omega$) that allows recovering the electrostatic potential $u^\ga_f$ on $\Omega$ from the voltage $f$ applied to its boundary $\p\Omega$. This equation is
\begin{equation}\label{e:conduct}
    \left\{\begin{array}{rl}
        -\nabla\cdot(\gamma\nabla u_f^\gamma)=0,  &  \text{ in }\Omega,\\
        u_f^\gamma= f, &   \text{ on }\partial\Omega.
    \end{array}\right.
\end{equation}
The Dirichlet-to-Neumann map (DtN) returns the current through $\p\Omega$ corresponding to a certain voltage $f$. This is the linear operator defined on functions on $\p\Omega$ by
\begin{equation}\label{e:DtN}
    \Lambda_\gamma f := (\gamma \partial_\nu u_f^\gamma) |_{\partial\Omega},
\end{equation}   
whenever $u_f^\gamma$ solves \eqref{e:conduct}. 
The Calderón problem can be stated precisely as the problem of reconstructing the conductivity $\ga$ from the DtN map $\Lambda_\ga$ or, equivalently,
as the inversion of the nonlinear map $\Phi$ defined as
\begin{equation}\label{e:defphi}
    \Phi(\ga):= \Lambda_\ga,
\end{equation}
where $\ga$ varies in some class of admissible conductivities. The mathematical analysis of this inverse problem has a long and rich history; results showing that $\gamma$ can be uniquely reconstructed from $\Lambda_\ga$ (\textit{i.e.}, that $\Phi$ is injective) date back to the seminal works by Kohn and Vogelius \cite{KoVo84}, Sylvester and Uhlmann \cite{SU86, SU87}, and Nachman \cite{Na88}. Stability, another important aspect of the inverse problem, has also been widely studied. $\Phi^{-1}$ is known to be discontinuous for any natural topology on the set of DtN maps (the inverse problem is \textit{ill-posed}) although numerous \textit{conditional stability} results are known to hold. We refer the reader to the recent monograph \cite{FSUBook} for a detailed description of these results as well as a thorough guide to the literature. 

Much less is known about the problem of characterizing the range of $\Phi$, that is, the set of all DtN maps arising from conductivities (some partial results for very specific classes of two-dimensional conductivities can be found in \cite{Ingerman, Sha11}); because of ill-posedness, this question is especially relevant to the reconstruction aspect of the Calderón problem: namely, that of finding effective procedures to construct $\gamma$ from $\Lambda_\gamma$. Reconstruction has been thoroughly investigated, partly due to its connection to a number of imaging techniques, such as Electrical Impedance Tomography (EIT).

The results in this article aim to contribute to the understanding of the range of $\Phi$, and as a byproduct, to provide a rigorous background for some widely used reconstruction methods in EIT based on linearization. A short review on these methods, and a description of their connection with the results in this article is presented in \Cref{s:survnumerico}. 

Linearization methods are based on the inversion of $\d \Phi_\sigma$, the Fréchet differential of $\Phi$ at some reference conductivity $\sigma$. This is known as the linearized Calderón problem. Since the work of Calderón \cite{calderon}, understanding this linearized version of the inverse problem has been a cornerstone in the strategy for solving the full nonlinear problem from both theoretical and numerical viewpoints. One tries to determine a function $\gb_\sigma$ that solves
\begin{equation}\label{e:defbor0}
    \d \Phi_\sigma(\gb_\sigma - \sigma) =\Lambda_\ga - \Lambda_\sigma,
\end{equation}
where $\ga$ is the conductivity one desires to approximate with respect to a background conductivity $\sigma$ (or a sufficiently close prior); see \cite{HaSe2010} for instance. 
This is merely formal, since it is not clear that such a function exists.
Motivated by the linear approximation in inverse scattering theory developed by Born in the 1920s, following \cite{Barcelo2022, BCMM23_n}, we call $\gb_\si$ the Born approximation of $\ga$ with respect to $\si$.

Though widely used in EIT, there are very few rigorous results on the existence of $\gb_\si$. Unlike in inverse scattering problems, the question of existence of the Born approximation in the Calderón problem is a subtle issue. If one were able to establish it, this would imply that every DtN map could be written as a linearized DtN operator, and therefore that the range of $\Phi$ is contained in that of $\d \Phi_\sigma$. To the best of our knowledge, the existence of $\gb_\si$ has been rigorously established only when $\si = 1$, $\Om = \IB^d\subset\IR^d$ is the unit ball, and $\ga \in W^{2,p}(\IB^d)$ is radial; see \cite[Theorem~1]{radial_cond}\footnote{\cite[Theorem~1]{radial_cond} proves the existence of the Born approximation as the solution of a certain moment problem involving the spectrum of the DtN map. The connection with the Fréchet differential of $\Phi$ is presented in identity \eqref{e:moments_ka0} in the proof of \Cref{t:existencegb}.} (see also \cite{Radial_Born,MMS-M25} for rigorous results on the existence of the Born approximation for DtN maps associated with Schrödinger operators).

Furthermore, if  $\gb_\sigma$ exists, then the strong injectivity properties of $\d\Phi_\si$ imply that one has a two-step factorization of the inverse problem: a linear step that asks for reconstructing the Born approximation $\gb_\si$ from the DtN map, and a non-linear one that consists of the inversion of the map $\Phi_\si^\mB$ that sends the function $\ga$ to the function $\gb_\si$.  More precisely, using that $\La_\si = \d\Phi_\si(\si)$ (see \Cref{prop:Frech}), \eqref{e:defbor0} becomes
\begin{equation}\label{e:defborn_2}
    \d \Phi_\sigma(\gb_\sigma) =\Lambda_\ga ,
\end{equation}
therefore, inducing the factorization of $\Phi$ shown in the diagram:
\begin{equation}  \label{e:diagram_cond}
\begin{tikzcd}
\gamma \; \ar[dr,|->,"\Phi_\si^\mB" ] \ar[rr,|->,"\Phi"] && \La_\gamma \\
 &  \gb_\si  \ar[ur,|->,"\d\Phi_\si"]
\end{tikzcd}
\end{equation}
This type of factorization can be formulated in a general context and has been applied under different names to solve a number of ill-posed inverse problems (see the survey \cite{MaMe24}). 
In \cite{radial_cond} it is shown that when $\ga \in W^{2,p}(\IB^d)$ is a radial conductivity,  $\gb_1$ exists and uniquely determines $\ga$.
Moreover, in this case, the map $(\Phi_1^\mB)^{-1}: \gb_1 \longmapsto \ga$ is Hölder continuous under suitable \textit{a priori} assumptions on the conductivities (that unlike most conditional stability results, do not imply any compactness on the admissible set of conductivities). 
In other words, the aforementioned fact that $\Phi^{-1}:\La_\ga \longmapsto \ga$ is discontinuous is due to the linear part of the factorization: the inversion of $\d\Phi_\si$.
We say that the inverse problem admits a \textit{stable factorization} if $\Phi$ can be factorized into a linear ill-posed map and a non-linear stable map. In addition to the results in  \cite{radial_cond} for the conductivity case, this property has also been proven in \cite{Radial_Born} for the radial Schrödinger case and an associated inverse spectral problem, and locally in \cite{MMS-M25} for the radial fixed energy Schrödinger case.

In this article, we present a systematic approach to establishing the existence of the Born approximation based on the well-known reduction to a problem involving Schrödinger operators; as a consequence of our analysis, we broaden the class of conductivities $\si$ and $\ga$ for which \eqref{e:defbor0} can be rigorously solved. More specifically, we show that $\gb_\sigma$ exists as an integrable function as soon as $\Omega = \IB^d$ for $d\ge 2$, $\ga$ is a $W^{2,\infty}$ radial conductivity, and $\sigma$ is
\begin{equation} \label{e:si_ga}
    \sigma_{\ka,d} (x):= \left(c_d\frac{J_{\nu_d}(\sqrt{\ka} |x|)}{(\sqrt{\ka}|x|)^{\nu_d}}\right)^2,\qquad \nu_d:=\frac{d-2}{2},\quad c_d:= \Gamma(\nu_d+1)2^{\nu_d},
\end{equation} 
for $\ka\in (-\infty, \lambda_{\nu_d,1}^2)$, where $\lambda_{\nu_d,1}$ stands for the first positive zero of the Bessel function $J_{\nu_d}$.\footnote{When $\ka<0$, note that $\sigma_{\ka,d}(x)=\left(c_d\frac{I_{\nu_d}(\sqrt{|\ka|} |x|)}{(\sqrt{|\ka|}|x|)^{\nu_d}}\right)^2$. When $\ka = 0$, the formula should be understood in the limit sense, so that $\sigma_{0,d}=1$.}

We denote by $\cE'(\Om)$ the space of distributions with compact support contained in $\Om$. It is shown in \Cref{r:extdiff} that, for $\si \in \cC^\infty(\ol{\Om})$,  one can extend  $\d \Phi_{\sigma}$ from its natural domain $L^\infty(\Omega,\IR)$ to $L^\infty(\Om)+\cE'(\Om)$.
\begin{main-theorem}\label{t:existencegb}
    Suppose that $\Omega:=\IB^d$ is the unit ball in $\IR^d$, that $\gamma\in W^{2,\infty}(\IB^d)$ is a radial conductivity, and that $\ka\in(-\infty,\lambda_{\nu_d,1}^2)$. 
    Then there exists a distribution $\gb_{\sigma_{\ka,d}}\in W^{2,\infty}(\IB^d)+\cE'(\IB^d) $ such that:
    \begin{equation} \label{e:mainthm}
        \Lambda_\ga = \d \Phi_{\sigma_{\ka,d}}(\gb_{\sigma_{\ka,d}}).
    \end{equation}
    In addition, the following statements hold:
    \begin{enumerate}[i)]
        \item $\gb_{\sigma_{\ka,d}}$ is uniquely determined in $W^{2,\infty}(\IB^d)+\cE'(\IB^d) $ by \eqref{e:mainthm}.
        \item  $\gb_{\sigma_{\ka,d}}$ is radial and real,  $\gb_{\sigma_{\ka,d}}\in W^{1,1}(\IB^d)$ and $\gb_{\sigma_{\ka,d}} \in \cC^1(\overline{\IB^d}\setminus \{0\})$.
        \item $\gb_{\sigma_{\ka,d}}  |_{\partial\IB^d} = \ga|_{\partial\IB^d}$, and $\partial_\nu\gb_{\sigma_{\ka,d}}  |_{\partial\IB^d}=\partial_\nu\ga|_{\partial\IB^d}$.
    \end{enumerate}
\end{main-theorem}
Identity \eqref{e:mainthm} in the theorem implies the following representation formula for $\Lambda_\ga$:
    \begin{equation*}
         \hp{f}{\Lambda_\ga g}_{H^{1/2}\times H^{-1/2}} =  \int_{\IB^d} \gb_{\sigma_{\ka,d}}(x)  \nabla u^{\sigma_{\ka,d}}_f(x) \cdot \nabla u^{\sigma_{\ka,d}}_g(x)\, \d x,  \qquad \forall f,g \in H^{1/2}(\IS^{d-1}). 
    \end{equation*}

 \Cref{t:existencegb} generalizes to $\ka \neq 0$ some of the results in \cite{radial_cond}. The uniqueness statement actually holds in a slightly more general setting, see Theorems \ref{thm:Born_d=2_uniq_0_Om} and \ref{thm:Born_d=2_uniq_Om}. % \ref{prop:Born_d=2_uniq_0}

A second set of results aims to understand the validity of \Cref{t:existencegb} for non-radial conductivities. We first observe that the Born approximation can be characterized as a solution to a certain \textit{complex moment problem} that is described as follows. Let  $(e_m)_{m\in \N}$ be an orthonormal basis of $L^2(\p \Om)$ such that $e_m\in H^{1/2}(\p \Om)$ for all $m\in \N$.  By direct computation of the Fréchet differential $\d\Phi_\si$ (see \Cref{prop:Frech}), one can prove  that for all $F\in L^\infty(\Om)$, and $\ell,m \in \N$
\begin{equation*}
    \gm_{\ell,m}^\si[F] := \br{\cc{e_{\ell}}, \d\Phi_\si(F) e_m}_{H^{1/2}\times H^{-1/2}} = \int_{\Om} F(x) \ol{\nabla u_\ell^\si} \cdot \nabla u_m^\si \, \d x, 
\end{equation*}
% \begin{equation*}
%     \gm_{\ell,m}^{\si}[F] = \int_{\Om} F(x) \ol{\nabla u_\ell^\si} \cdot \nabla u_m^\si \, \d x,
% \end{equation*}
where $u_m^\si$ is the solution of \eqref{e:conduct} with $f= e_m$.
% $\gm_{\ell,m}^{\si}[\gb] =  \hp{e_{\ell}}{\Lambda_{\ga} e_m}_{H^{1/2}\times H^{-1/2}}$ for all $\ell,m\in\N,$,
% \begin{equation*}
%     =  \hp{\cc{e_{\ell}}}{\Lambda_{\ga} e_m}_{H^{1/2}\times H^{-1/2}}, \qquad \ell,m\in\N,
% \end{equation*}
By \eqref{e:defborn_2}, it follows that $\gb_\si$ solves the generalized moment problem
\begin{equation}\label{e:gen_moment_prob}
    \gm_{\ell,m}^{\si}[\gb_\si] =  \int_{\Om} \gb_\si \, \ol{\nabla u_\ell^\si} \cdot \nabla u_m^\si \, \d x =  \hp{\cc{e_{\ell}}}{\Lambda_{\ga} e_m}_{H^{1/2}\times H^{-1/2}}, \qquad \ell,m\in\N.
\end{equation}
These remarks show that when $\Omega= \ID$ is the unit disk in $\IR^2$ and $\sigma=1$, the conclusion of \Cref{t:existencegb} may fail for non-radial $\gamma$. In fact, letting $e_m(\theta):=(2\pi)^{-1/2}e^{i m\theta}$ be the standard basis of $L^2(\IS^1)$ of spherical harmonics, we show (see \Cref{r:freemoments}) that for $F\in L^1(\ID)$,
\begin{equation*}
    \gm_{\ell,m}^{1}[F]=0,\qquad \forall \ell,m\in\IZ,\; \ell m < 0. 
\end{equation*}
As a consequence, $\d\Phi_1(F)$ maps $H^{1}(\IS^1)\cap \cH_+$ into $\cH_+$, where
\[
\cH_+:=\{u\in L^2(\IS^1)\,:\,\hp{e_{\ell}}{u}_{L^2(\IS^1)}=0,\;\forall \ell<0\}.
\]
If the Born approximation $\gb_1$ were to exist, this would imply that $\Lambda_\ga$ also maps $H^{1}(\IS^1)\cap \cH_+$ into $\cH_+$, which is not necessarily the case when $\ga$ is non-radial. This can be checked in explicit examples: if $\eps>0$ and $\ga_\eps(x,y):=1+\eps(1-x^2-y^2)(x^2-y^2)$, with $(x,y)\in\ID$, direct computation shows that:
\begin{equation*}
    \hp{\cc{e_{-3}}}{\Lambda_{\ga_\eps} e_1}_{H^{1/2}\times H^{-1/2}}=-\frac{\eps^2}{320}+\cO(\eps^3).
\end{equation*}
Therefore, in this setting, it makes sense \textit{a priori} to define the Born approximation $\gb_1$ as a function/distribution satisfying:
\begin{equation*}
    \d\Phi_1(\gb_1)=\Pi_+\Lambda_\gamma \Pi_+|_{H^{1}(\IS^1)},
\end{equation*}
where $\Pi_+:L^2(\IS^1)\To\cH_+$ denotes the orthogonal projection. A characterization for $\gb_1$ based on this \textit{Ansatz}, which turns out to be equivalent to solving the moment problem \eqref{e:gen_moment_prob} with respect to the basis of spherical harmonics, is given in \Cref{s:sec_disk}. More generally, the structure of $\gb_{\si}$ for $\si=\sigma_{\ka,2}$ on $\ID$ with $\ka<\lambda_{0,1}^2$ is also investigated through the moment problem \eqref{e:gen_moment_prob} (again with respect to the basis $(e_m)_{m\in \N}$ of spherical harmonics in $\IS^1$). In particular, \eqref{e:gen_moment_prob} is shown to possess at most one solution in this case.    

We also numerically explore the structure and properties of the Born approximation as defined by \eqref{e:gen_moment_prob} through several examples of non-radial $\ga$, for $\si = \si_{\ka,2}$ on the unit disk in $\R^2$. 
These experiments are presented in \Cref{s:numerical}. This choice of $\si$ has the advantage of providing good \textit{a priori} knowledge of the corresponding solutions $u_m^\si$ and encompasses the framework in which we have rigorously proved the existence of the Born approximation. 

First, we study the ability of the Born approximation $\gb_1$ to recover different bump perturbations of the conductivity $\si = 1$ (which corresponds to $\ka=0$). 
Second, we analyze how the resolution deteriorates with depth. Third, we consider conductivities that are perturbations of $\si_{\ka,2}$ and analyze the difference between the Born approximations $\gb_{\si_{\ka,2}}$ and $\gb_{\si_{0,2}}$. 
This experiment illustrates the significant advantage of using a good prior for $\sigma$. In fact, we show that the correct perturbation can be recovered from $\gb_{\si_{\ka,2}}$ even in the presence of screening effects (low conductivity at the boundary), while $\gb_{\si_{0,2}}$ is completely useless for this task (see \cref{Fig:kappa+}). 
Finally, we study how the method deals with discontinuous conductivities and noisy data in experiments 4 and 5, respectively.

\subsection{The role of \texorpdfstring{$\gb_\sigma$}{the Born approximation} in linearization methods for EIT}\label{s:survnumerico}

Linearization methods have been among the main approaches to EIT since its early developments in the 1980s and 1990s. A common way to formulate the inverse problem is to characterize the conductivity as the solution of a least-squares minimization problem, for which gradient-based or even Newton-type methods can be applied. 
The main difficulty is to numerically solve  the linearized problem \eqref{e:defbor0}. 
Explicit expressions can be readily derived when linearizing around the constant conductivity $\si = 1$. This strategy has led to successful reconstruction methods, such as the NOSER algorithm, which effectively corresponds to performing a single Newton step starting at $\si=1$; see, for example, \cite{10.1002/ima.1850020203,Stephenson2009,HaSe2010}. 
The importance of understanding how the solution of \eqref{e:defbor0} differs from the conductivity  has been identified as a crucial question for EIT in \cite{HaSe2010}, where the authors show that assuming $\gb_\si$ exists,  then $\gb_\si-\si$ has the same outer support as $\ga-\si$. The fact that this is done for the Neumann to Dirichlet map instead of the DtN map is not a fundamental difference; see \Cref{r:neumann}.

The Born approximation also appears in the context of the D-bar method in EIT (see \cite{MuellerSiltanenBook, Mueller_Siltanen_20_D_bar_demistified}), under different forms. 
In particular,  \cite{BKM11,KnudsenMueller11} introduces, among others, the formal objects $q_{exp}$, $\ga_{exp}$, and $\ga_{ap}$ obtained from  $t_{exp}$:
a linearization of the scattering transform used in the D-bar method.  
In \cite{Barcelo2022,MaMe24}, it is shown that $q_{exp}$ is just a solution of the Schrödinger version of \eqref{e:defbor0} (given in \eqref{e:born_pot_def} below when $q_0 = 0$), and hence, it coincides with the Born approximation of the potential $q$, denoted in this work by $q^{\mB}_0$. 
From  \Cref{s:2_2}, it follows that the approximation $\gamma_{ap}$, obtained in \cite{BKM11,KnudsenMueller11}  from $q_{exp}$ by linearizing the map  $q = \Delta \sqrt{\ga}/\sqrt{\ga}$, is a formal solution of  \eqref{e:defbor0}, and hence it coincides with the Born approximation $\gb_1$ defined by \eqref{e:defbor0} (this is especially simple in the radial case; see \cite{BCMM23_n}). 
    
Under the name of \textit{Calderón method}, the works \cite{Bik_Mueller_08, KnudsenMueller11, BKM11} consider the approximation of $\gamma$ using $\gb_1$. 
This method is also  used in the anisotropic case in \cite{MR4186178}, with second order corrections in \cite{Mueller_second_order},
and when $d=3$ in \cite{DHK12, DelKn14}, together with full reconstruction algorithms.  The effectiveness of the approximation $\gb_1$ has been   compared with other methods  in \cite{HIKMTB21}  using synthetic data to simulate real discrete data from electrodes.
    
The idea that a suitable linearization of the Calderón problem can lead to a factorization into a linear unstable part and a non-linear stable part has also been implicitly used in other numerical approaches to EIT; however, as far as we know, the only rigorous results are in \cite{Radial_Born, MMS-M25, radial_cond}. 
This kind of factorization has been observed, for example, in \cite{2408.12992}, where the approach in \cite{GLSSU2018} is used to reconstruct $\ga$ when $d=2$.   
This is done in two steps; one linear, which involves inverting a certain Fréchet differential in the same spirit as \eqref{e:defbor0}, and one non-linear, which is observed to be  more numerically stable.

The Born approximation and the stable factorization property have also been successfully used in combination with Neural Network (NN) approaches \cite{Martin2017,Cen_Jin_etal_deeep_calderon_2023}. 
In both papers, reconstruction is divided into two steps: first,  EIT images are generated using the Born approximation (under the name of the Gauss-Newton method in \cite{Martin2017} and the Calderón method in \cite{Cen_Jin_etal_deeep_calderon_2023}), and then the non-linear step $(\Phi_\sigma^{\mB})^{-1}: \gb_\si \longmapsto \ga $ is approximated with a NN. 
In \cite{Martin2017}, the aim of the NN is to reduce the regularization error, and in \cite{Cen_Jin_etal_deeep_calderon_2023}, to enhance the resolution  of the first step (which leads to blurry images). After the previous discussion, one might suppose that the NN also benefits from the improved stability of  the non-linear step.  
See, for instance, \cite{MR4173583,Hamilton2019,8352045,Chow_2014,doi:10.1137/20M1367350,2511.20361} for other approaches to EIT and Calderón using NNs.

\subsection{Structure of the article} 
\Cref{s:born} analyzes the linearized Calderón problem and the problem of constructing the Born approximation. In particular, we prove a series of uniqueness results, Theorems \ref{thm:Born_d=2_uniq_0_Om} and \ref{thm:Born_d=2_uniq_Om},  and we prove  \Cref{t:existencegb}. In \Cref{s:sec_disk} we consider the particular case of the disk in two-dimensional Euclidean space and analyze the connection between the Born approximation and the solution to certain complex moment problems. In particular, we derive explicit expressions for the matrix elements of the Fréchet differential $\d \Phi_{\si_\ka}(\ga)$ in the spherical harmonics, which will be useful for the numerical reconstruction of the Born approximation in \Cref{s:numerical}. Also, this analysis  leads to a representation result of the Fourier transform of functions supported on the disk, \Cref{thm:Fourier d=2}, which is of independent interest. Finally, \Cref{s:numerical} describes a computational strategy for solving the moment problem presented in \Cref{s:sec_disk}, along with several numerical experiments in model situations.

\subsection*{Acknowledgments}

The authors have been supported by grants PID2021-124195NB-C31 and PID2024-158664NB-C21 from Agencia Estatal de Investigación (Spain).

\section{The Born approximation}\label{s:born}

We start in \Cref{s:subsec_lin_calderon} by studying the linearized Calderón Problem  for conductivities and for Schrödinger operators. We then analyze in \Cref{s:2_2} the relation between the Born approximations of these two problems. In \Cref{s:2_3_uniqueness} we address the question of uniqueness under the assumption of existence. Finally, we prove \Cref{t:existencegb} in \Cref{s:2_4}. 

\subsection{Linearized Calderón problem}\label{s:subsec_lin_calderon}\

We start by studying some properties of the Fréchet differential of the nonlinear map $\Phi$ defined in \eqref{e:defphi}. 

Let $\Omega$ be a smooth domain; a function $\ga \in L^\infty(\Omega)$ is said to be a \textit{conductivity} provided it is real valued and that there exists $c>0$ such that $\gamma \geq c$ almost everywhere. 
Under this regularity assumption on $\ga$, the operator $\Lambda_\ga$ is defined through the following integration by parts identity:
\begin{equation}\label{e:Aless}
    \hp{f}{\Lambda_\gamma g}_{H^{1/2}\times H^{-1/2}}   = \int_\Omega  \gamma(x)\nabla u_f(x)\cdot \nabla u_g^\gamma (x) \,\d x,
\end{equation}
where $u_f$ is any function such that $u_f\in H^1(\Om)$ and $u_f|_{\partial\Omega} = f$.
The Fréchet differential $\d \Phi_\sigma$ at a conductivity $\sigma\in L^\infty(\Omega)$ maps essentially bounded functions in $\Omega$ into operators characterized by the following result.
\begin{proposition}\label{prop:Frech} 
    Let $\sigma\in L^\infty(\Omega)$ be a conductivity and $\ga\in L^\infty(\Omega,\IR)$. The Fréchet differential $\d \Phi_\sigma(\ga)$ exists as an operator in $\cL(H^{1/2}(\partial\Omega), H^{-1/2}(\partial\Omega))$ that is characterized by the identity
    \begin{equation}\label{e:formalbornka}
        \hp{f}{\d \Phi_\sigma(\ga) g}_{H^{1/2}\times H^{-1/2}} =  \int_\Omega \ga(x)  \nabla u^{\sigma}_f(x) \cdot \nabla u^{\sigma}_g(x)\, \d x, \qquad \forall  f,g \in H^{1/2}(\partial\Omega).
    \end{equation}
    In particular, it always holds that $\d\Phi_\sigma(\sigma)=\Lambda_\sigma$.
\end{proposition}
\begin{proof}
    Given $\eps>0$ write $K_\eps:=\eps^{-1}(\Phi(\sigma+\eps\gamma) - \Phi(\sigma)) $. For every $f,g\in H^{1/2}(\partial\Omega)$, identity \eqref{e:Aless} implies that
    \begin{equation*}
        \hp{f}{K_\eps g}_{H^{1/2}\times H^{-1/2}}=\int_\Omega \ga(x)  \nabla u^{\sigma}_f(x) \cdot \nabla u^{\sigma+\eps \ga}_g(x)\, \d x ,   
    \end{equation*}
    since
    \begin{align*}
         \int_\Omega \sigma(x)  \nabla u^{\sigma}_f(x) \cdot \nabla u^{\sigma}_g(x)\, \d x=&\hp{f}{\Lambda_\sigma g}_{H^{1/2}\times H^{-1/2}} \\   
         =& \hp{g}{\Lambda_\sigma f}_{H^{1/2}\times H^{-1/2}}= \int_\Omega \sigma(x)  \nabla u^{\sigma}_f(x) \cdot \nabla u^{\sigma+\eps\gamma}_g(x)\, \d x.
    \end{align*}
    As $\lim_{\eps\to 0}\|u^{\sigma+\eps\gamma}_g-u^{\sigma}_g\|_{H^1(\Omega)}=0$, identity \eqref{e:formalbornka} follows. The last assertion follows from comparing \eqref{e:formalbornka} and \eqref{e:Aless}. 
\end{proof}
Once we have $\d\Phi_\sigma(\ga)$ defined for $\ga \in L^\infty(\Om,\R)$ we can extend it linearly to all complex valued $\gamma \in L^\infty(\Om)$. In the following remark we extend $\d\Phi_\sigma$ further to a space that includes $\cE'(\Omega)$, the space of compactly supported distributions on $\Omega$. 
% We will also use the notation $\cE'(\Omega,\IR)$ for the subspace of $\cE'(\Omega)$ formed by real valued distributions.  
\begin{remark}\label{r:extdiff}
    Formula \eqref{e:formalbornka} shows that $\d\Phi_\sigma(\gamma)$ can be defined for functions $\gamma$ that are less regular than $L^\infty(\Omega)$. For simplicity, assume that we linearize at a conductivity $\sigma\in\cC^\infty(\overline{\Omega})$.  Interior regularity elliptic estimates imply that, for any open set $U$ such that $\overline{U}\subset \Omega$ and for any $m\geq 1$, one has that 
    \begin{equation}\label{e:intelreg}
        \|u^\sigma_f\|_{H^m(U)}\leq C_m\|u^\sigma_f\|_{H^1(\Omega)}\leq C_m' \|f\|_{H^{1/2}(\partial\Omega)},
    \end{equation}
    uniformly in $f\in H^{1/2}(\partial\Omega)$.
    Estimate \eqref{e:intelreg} and identity \eqref{e:formalbornka} show that $\d\Phi_\sigma(\ga)$ can be defined as a bounded linear operator from $H^{1/2}(\p\Omega)$ to $H^{-1/2}(\p\Omega)$ for $\ga\in\cE'(\Omega)$ as:
    \begin{equation*}
        \hp{f}{\d \Phi_{\sigma}(\ga) g}_{H^{1/2}\times H^{-1/2}}:=\hp{\ga}{\chi\nabla u_f^{\sigma}\cdot \nabla u_g^{\sigma}}_{\cE'\times \cC^\infty},
    \end{equation*}
    where $\chi\in\cC^\infty_c(\Omega)$ satisfies $\chi \ga =\ga$. This definition is independent of the specific choice of $\chi$.
    More specifically, if $\ga$ is of order $k\in\IN$ and is supported on some open set $U$ properly contained in $\Omega$, one has by Sobolev's inequalities, taking $m$ in \eqref{e:intelreg} large enough,
    \begin{equation*}
        \left|\hp{f}{\d\Phi_\sigma(\ga) g}_{H^{1/2}\times H^{-1/2}}\right|\leq C_\ga\|\nabla u_f^\sigma\cdot\nabla u_g^\sigma\|_{\cC^k(U)}\leq C \|f\|_{H^{1/2}(\p\Omega)}\|g\|_{H^{1/2}(\p\Omega)}.
    \end{equation*}
    uniformly in $f,g\in H^{1/2}(\p\Omega)$. Therefore, when $\sigma\in \cC^\infty(\overline{\Omega})$ is a conductivity, $\d\Phi_\sigma$ can be extended to a linear map from $L^\infty(\Omega)+\cE'(\Omega)$ to $\cL(H^{1/2}(\p\Omega),H^{-1/2}(\p\Omega))$.
\end{remark}
In order to study the existence of the Born approximation, it will be useful to consider Dirichlet-to-Neumann maps associated to a Schrödinger operator.  
Given $q\in L^\infty(\Omega,\IR)$ such that $0$ is not a Dirichlet eigenvalue of $-\Delta+q$ and $g\in H^{1/2}(\partial\Omega)$ define $v_g^q$ to be the unique solution to 
\begin{equation}
    \label{e:schrod}
    \left\{\begin{array}{rl}
        (-\Delta +q) v_g^q =0,  &  \text{ in }\Omega,\\
        v_g^q= g, &   \text{ on }\partial\Omega.
    \end{array}\right.
\end{equation}
The corresponding Dirichlet-to-Neumann map is $\Lambda^{\rm S}_{q} g:=\partial_\nu v_g^q$. Analogously, one can consider the map $\Phi^{\rm S}(q_0):=\Lambda^{\rm S}_{q_0}$, defined for potentials $q_0$ in $L^\infty(\Omega,\IR)$ whose corresponding Schrödinger operator does not have zero in its Dirichlet spectrum. 
One can prove by integration by parts that for every $f,g\in H^{1/2}(\partial\Omega)$ and $q\in L^\infty(\Omega,\IR)$, one has:
\begin{equation}\label{e:Aless2}
    \hp{f}{(\La^{\rm S}_q-\La^{\rm S}_{q_0}) g}_{H^{1/2}\times H^{-1/2}}  = \int_\Omega  (q-q_0)(x) v_f^{q_0}(x) v_g^q (x) \,\d x,
\end{equation}
an integration by parts identity that goes back to \cite{Alessandrini88}. Therefore, a similar reasoning to that in the proof of \Cref{prop:Frech} gives
\begin{equation}\label{e:frechp}
    \hp{f}{\d \Phi_{q_0}^{\rm S}(q) g}_{H^{1/2}\times H^{-1/2}} =  \int_\Omega q(x)  v^{q_0}_f(x) v^{q_0}_g(x)\, \d x, \qquad \forall f,g\in H^{1/2}(\partial\Omega).
\end{equation}
A similar argument to that presented in \Cref{r:extdiff} (see also the proof of \Cref{l:comp} below) shows that as soon as $q_0\in \cC^\infty(\overline{\Omega},\IR)$, one can extend $\d \Phi_{q_0}^{\rm S}$ to $L^\infty(\Omega)+\cE'(\Omega)$. One sets, for $q\in\cE'(\Omega)$,
\begin{equation}\label{e:defdis}
    \hp{f}{\d \Phi_{q_0}^{\rm S}(q) g}_{H^{1/2}\times H^{-1/2}}:=\hp{q}{\chi v_f^{q_0} v_g^{q_0}}_{\cE'\times \cC^\infty},
\end{equation}
for some $\chi\in\cC^\infty_c(\Omega)$ with $\chi q =q$.

The two classes of DtN maps are related as follows. Suppose $\ga\in W^{2,\infty}(\Omega)$ is a conductivity and let 
\begin{equation} \label{e:q_funct}
    q(\gamma):=\frac{\Delta\sqrt{\ga}}{\sqrt{\ga}}.
\end{equation}
Then $q(\ga) \in L^\infty(\Om,\IR)$ and $u_f^\ga$ solves \eqref{e:conduct} if and only if $\sqrt{\ga}u_f^\ga$ solves \eqref{e:schrod} with $q=q(\gamma)$ and $g=\sqrt{\ga}f$. Moreover, one has:
\begin{equation}\label{e:reldtn}
    \Lambda_\ga f = \sqrt{\ga} \Lambda^{\rm S}_{q(\ga)} (\sqrt{\ga}f) - \frac{\partial_\nu\ga}{2}f.
\end{equation}
This identity allows us to prove a different representation formula for the Fréchet differential of $\Phi$ for smoother conductivities. 
\begin{proposition}\label{prop:Frech2} Let $\sigma\in \cC^\infty(\cc{\Omega})$ be a conductivity and let $\ga\in W^{2,\infty}(\Omega)+\cE'(\Omega)$.
The Fréchet differential $\d \Phi_\sigma(\ga)$ satisfies, for every $f\in H^{1/2}(\partial\Omega) $:
\begin{equation}\label{e:relfrech}
    \frac{1}{\sqrt{\sigma}}\d \Phi_\sigma(\ga) \left(\frac{1}{\sqrt{\sigma}}f\right) =\frac{1}{2}\left(\frac{\ga}{\sigma} \Lambda_{q(\sigma)}^{\rm S}(f)  + \Lambda_{q(\sigma)}^{\rm S} \left(\frac{\ga}{\sigma}f\right) - \frac{\partial_\nu\ga}{\sigma}f + \d \Phi_{q(\sigma)}^{\rm S}(\cL_\sigma \ga)(f)\right),  
\end{equation}
where $\ga|_{\partial\Omega},\partial_\nu\ga|_{\partial\Omega}$ should be interpreted as zero when $\ga\in\cE'(\Omega)$ and,
\begin{equation}\label{e:frechet2}
      \cL_\sigma \ga:=\frac{1}{\sigma}\Div\left(\sigma\nabla\left(\frac{\ga}{\sigma}\right)\right).      
\end{equation}
\end{proposition}
\begin{proof}
    Identity \eqref{e:reldtn} can be rewritten as:
    \begin{equation*}
        \Phi(\sigma) = \sqrt{\sigma} \Phi^{\rm S}(q(\sigma)) \sqrt{\sigma} - \frac{\partial_\nu\sigma}{2}\Id.
    \end{equation*}
    Differentiating this identity at $\sigma$ along the direction $\ga\in W^{2,\infty}(\Omega,\IR)$ and using the fact that the Fréchet differential of $\ga\longmapsto q(\ga)$ at $\sigma$ satisfies:
    \begin{equation*}
        \d q_\sigma(\ga)=\frac{1}{2\sqrt{\sigma}}\left[\Delta\left(\frac{\ga}{\sqrt{\sigma}}\right) - q(\sigma)\frac{\ga}{\sqrt{\sigma}}\right]=\frac{1}{2}\cL_\sigma \ga,
    \end{equation*}
    gives the result in that case. By linearity this can be extended to complex valued $\gamma \in W^{2,\infty}(\Om)$. To conclude, observe that when $\ga$ is compactly supported in $\Omega$, identity \eqref{e:relfrech} reduces to
    \begin{equation*}
        \frac{1}{\sqrt{\sigma}}\d \Phi_\sigma(\ga) \left(\frac{1}{\sqrt{\sigma}}f\right) =\frac{1}{2}\d \Phi_{q(\sigma)}^{\rm S}(\cL_\sigma \ga)(f).
    \end{equation*}
     Using this together with a density and regularization argument, one obtains the result for $\ga\in\cE'(\Omega)$.
\end{proof}

\begin{remark}  \label{r:neumann}
    The functional associating a conductivity to its Neumann-to-Dirichlet map $\ga\longmapsto\Lambda_\ga^{-1}$ can be linearized in a similar fashion. Its Fréchet differential at $\sigma$ is related to that of $\Phi$ by
    \begin{equation*}
        \ga\longmapsto-\Lambda_\sigma^{-1}\d\Phi_\sigma(\ga)\Lambda_\sigma^{-1}   . 
    \end{equation*}
   In fact, the analysis that follows can be performed, with minor changes, in that setting.
\end{remark}

We conclude this section by showing that under suitable assumptions, both $\La^{\rm S}_q-\La^{\rm S}_{q_0}$ and $\d \Phi_{q_0}^{\rm S}(q)$ are compact operators on $L^2(\partial\Omega)$. 
\begin{proposition}\label{l:comp}
    Let $q_0\in L^\infty(\Omega,\IR)$ such that zero is not a Dirichlet eigenvalue of $-\Delta+q_0$. The following are compact operators on $L^2(\partial\Omega)$:
    \begin{enumerate}[i)]
        \item $\La^{\rm S}_q-\La^{\rm S}_{q_0}$, provided that $q\in L^\infty(\Omega,\IR)$, and zero is not a Dirichlet eigenvalue of $-\Delta+q$.
        \item $\d \Phi_{q_0}^{\rm S}(q)$, provided that, in addition, $q_0\in \cC^\infty(\overline{\Omega})$ and $q\in L^\infty(\Omega)+\cE'(\Omega)$.
    \end{enumerate}
\end{proposition}
\begin{proof}
    For $q\in L^\infty(\Omega,\IR)$, define the operators:
    \begin{itemize}
        \item $M_q$ is the bounded linear operator on $L^2(\Omega)$ that acts by multiplication by $q$.
        \item $R_q$ is the Dirichlet resolvent $(-\Delta+q)^{-1}$. If zero is not a Dirichlet eigenvalue of $-\Delta+q$, then this is a well-defined, bounded linear operator
        \begin{equation*}
            R_q: L^2(\Omega)\To H^2(\Omega)\cap H^1_0(\Omega).
        \end{equation*}
        \item $P_q$ is the $q$-harmonic extension operator from $\partial\Omega$. As soon as zero is not a Dirichlet eigenvalue of $-\Delta+q$, it is well-defined as $P_{q}f:=u^q_f$, the unique solution to \eqref{e:schrod}  with boundary datum $f$. Since one has 
        \begin{equation}\label{e:poisson0}
            P_q = (\Id - R_q M_q)P_0,
        \end{equation}
        and $P_0$ is nothing but the Poisson integral, which maps continuously $L^2(\partial\Omega)$ in $H^{1/2}(\Omega)$ (see, for instance, \cite[Section~7.12]{Taylor2}), it follows that 
        \begin{equation}\label{e:poisson1}
            P_q\in \cL(L^2(\partial\Omega),H^{1/2}(\Omega)), \qquad P_qf\in H^2_{\rm loc}(\Omega), \quad \forall f\in L^2(\partial\Omega).
        \end{equation}
    \end{itemize}
    Identities \eqref{e:Aless2} and \eqref{e:frechp} imply, respectively, that for every $f\in L^2(\partial\Omega)$
    \begin{equation}\label{e:idendtn}
        (\La^{\rm S}_q - \La^{\rm S}_{q_0})f = -\partial_\nu(R_q M_q P_{q_0}f), \qquad \d \Phi_{q_0}^{\rm S}(q)f=-\partial_\nu(R_{q_0}M_q P_{q_0}f),
    \end{equation}
     Statement i) follows from the decomposition \eqref{e:idendtn} and the Rellich embedding theorem, which ensures, under the hypotheses of the lemma, that $R_q$ and $P_{q_0}$ are compact operators into $L^2(\Omega)$. This same argument also shows that ii) holds when $q\in L^\infty(\Omega)$.
    
     We now prove ii). Suppose now that $q_0\in \cC^\infty(\overline{\Omega},\IR)$. Standard elliptic theory gives, for every $\chi\in \cC^\infty_c(\Omega)$ and any $m\geq 1$, the existence of $C_m>0$ such that
    \begin{equation}\label{e:intelregq}
        \|\chi P_{q_0}f\|_{H^m(\Omega)}\leq C_m\|P_{q_0}f\|_{L^2(\Omega)}\leq C_m \|P_{q_0}f\|_{H^{1/2}(\Omega)}\leq C_m' \|f\|_{L^2(\partial\Omega)},
    \end{equation}
    uniformly in $f\in L^2(\partial\Omega)$. This implies that for every $q\in \cE'(\Omega)$ (which is necessarily of finite order), $\d \Phi_{q_0}^{\rm S}(q)$ is a bounded operator on $L^2(\p \Omega)$. In fact, if $\chi\in \cC^\infty_c(\Omega)$ is such that $\chi^2 q=q$, taking $m>0$ large enough we obtain, by Sobolev's inequalities and the fact that $H^m(\Omega)$ is an algebra with respect to multiplication, 
    \begin{align*}
        \left|\hp{f}{\d\Phi^{\rm S}_{q_0}(q) g}_{H^{1/2}\times H^{-1/2}}\right| =&
        \left|\hp{q}{\chi^2 P_{q_0}f P_{q_0}g}_{\cE'\times \cC^\infty}\right|\\\leq& C_q \|\chi P_{q_0}f \chi P_{q_0}g\|_{\cC^k(\Omega)} \\ \leq& C_q' \|\chi P_{q_0}f\|_{H^m(\Omega)}\|\chi P_{q_0}g\|_{H^m(\Omega)} \\
        \leq& C_q'' \|P_{q_0}f\|_{L^2(\Omega)}\| P_{q_0}g\|_{L^2(\Omega)} \leq C \|f\|_{L^2(\partial\Omega)}\|g\|_{L^2(\partial\Omega)},
    \end{align*}
    for $f,g\in L^2(\partial\Omega)$. 
    In particular, letting $f=\d \Phi_{q_0}^{\rm S}(q)g$, we obtain using again \eqref{e:intelregq} the uniform bound:
    \begin{equation}\label{e:ultimatebound}
        \|\d \Phi_{q_0}^{\rm S}(q)g\|_{L^2(\partial\Omega)}^2\leq C\|P_{q_0}g\|_{L^2(\Omega)}\|\d \Phi_{q_0}^{\rm S}(q)g\|_{L^2(\partial\Omega)}\leq C'\|P_{q_0}g\|_{L^2(\Omega)}\|g\|_{L^2(\partial\Omega)}.
    \end{equation}
    Let $(g_n)_{n\in\IN}$ be any bounded sequence in $L^2(\partial\Omega)$ that converges weakly to zero. Then, since $P_{q_0}$ is compact, $(P_{q_0}g_n)_{n\in\IN}$ converges strongly to zero in $L^2(\Omega)$. Using \eqref{e:ultimatebound} we find that 
    \[
    \lim_{n\to\infty}\|\d \Phi_{q_0}^{\rm S}(q)g_n\|_{L^2(\partial\Omega)}=0,
    \]
    which shows that $\d \Phi_{q_0}^{\rm S}(q)$ is a compact operator.
\end{proof}

\subsection{The Born approximation in the Calderón problem}\ \label{s:2_2}

The Born approximation of a conductivity $\ga\in L^\infty(\Omega)$ with respect to a background conductivity $\sigma\in L^\infty(\Omega)$ is formally defined as a function (or distribution) $\gb_\sigma$ on $\Omega$ that satisfies \eqref{e:defbor0} or, equivalently:
\begin{equation}\label{e:defborn}
    \gb_\sigma =\sigma + \d \Phi_{\sigma}^{-1}(\Lambda_\gamma-\Lambda_\sigma) = \d \Phi_{\sigma}^{-1}(\Lambda_\gamma),
\end{equation}
the last equality being a consequence of \Cref{prop:Frech}.
Note that for this definition to make sense, one should have that $\Lambda_\gamma$ belongs to the range of some extension of $\d\Phi_\sigma$ to a larger space containing $L^\infty(\Omega)$. In general, the ansatz \eqref{e:defborn} must be modified (see \Cref{r:freemoments}); nonetheless, \eqref{e:defborn} provides a rigorous definition of the Born approximation when $\Om = \IB^d$, $\sigma=\sigma_{\kappa,d}$, and $\gamma$ has rotational symmetry, as we show in \Cref{t:existencegb}. 

Existence of the Born approximation is proved by reduction to the Schrödinger operator case. Given two potentials $q,q_0\in L^\infty(\Omega)$ one formally defines the Born approximation of $q$ at $q_0$ as
\begin{equation} \label{e:born_pot_def}
    q^{\rm B}_{q_0} := q_0 + (\d \Phi^{\rm S}_{q_0})^{-1}(\Lambda^{\rm S}_q - \Lambda^{\rm S}_{q_0}).
\end{equation}
Existence for the Born approximation has been rigorously established in the case of radial Schrödinger operators in \cite{Radial_Born, MMS-M25}, at zero and constant non-zero potentials, respectively, and in \cite{radial_cond} for radial conductivities, when taken with respect to the conductivity that is identically equal to one. In the radial setting, much less regularity has to be assumed on the potential/conductivity. For instance, for radial $q_0$ in $L^\infty(\Omega)$, it is possible to extend $\d \Phi_{q_0}^{\rm S}$ to the space $L^1(\Omega)$.

The next result connects these two notions of Born approximation.
\begin{proposition}\label{p:bqisgb}
    Let $\sigma\in\cC^\infty(\cc{\Omega})$ be a conductivity. For every conductivity $\ga\in W^{2,\infty}(\Omega)$ such that $\ga|_{\partial\Omega}=a_\ga\sigma|_{\partial\Omega}$ for some constant $a_\ga>0$, the following holds.
    \begin{enumerate}[i)]
        \item If $q(\ga)^{\rm B}_{q(\sigma)}\in L^\infty(\Omega) + L^p_c(\Omega)$ exists,\footnote{We use the notation $F_c(\Omega)$ to denote the set of functions on a function space $F(\Omega)$ that have compact support contained in $\Omega$.}  
        for some $1< p\leq \infty$, then $\gb_\sigma\in W^{2,\infty}(\Omega) + W^{2,p}_c(\Omega)$ exists. 
        If $p=1$, the conclusion holds with $\gb_\sigma\in W^{2,\infty}(\Omega) + W^{1,1}_c(\Omega)$.\footnote{It follows from Calderón-Zygmund theory that the Hessian of $\gb_\sigma$ belongs to the Lorenz space $L^{1,\infty}_c(\Omega,\IC^{d\times d})$.} In both cases, $\gb_\sigma$ can be chosen as the unique solution to
        \begin{equation}\label{e:gbfromqb}
        \cL_\sigma \gb_\sigma = 2a_\ga \left(q(\ga)^{\rm B}_{q(\sigma)}-q(\sigma)\right),\qquad \gb_\sigma  |_{\partial\Omega} = \ga|_{\partial\Omega}, 
        \end{equation}
         and  one has $\partial_\nu\gb_\sigma  |_{\partial\Omega} = \partial_\nu\ga|_{\partial\Omega}$.
        \item If $\gb_\sigma\in W^{2,\infty}(\Omega)+  \cE'(\Omega)$ exists, then one has 
        \begin{equation*}
            \gb_\sigma  |_{\partial\Omega} = \ga|_{\partial\Omega}, \qquad\partial_\nu\gb_\sigma  |_{\partial\Omega}=\partial_\nu\ga|_{\partial\Omega},    
        \end{equation*}
         and the Born approximation  $q(\ga)^{\rm B}_{q(\sigma)}\in L^\infty(\Omega) + \cE'(\Omega)$ exists and can be obtained as:
         \begin{equation} \label{e:born_pot}
             q(\ga)^{\rm B}_{q(\sigma)}=\frac{1}{2a_\gamma}\cL_\sigma \gb_\sigma + q(\sigma).
         \end{equation}
    \end{enumerate}    
\end{proposition}
\begin{proof}
    Assume that $q(\ga)^{\rm B}_{q(\sigma)}\in L^\infty(\Omega) + L^p_c(\Omega)$ exists:
    \begin{equation*}
    \d \Phi_{q(\sigma)}^{\rm S}\left(q(\ga)^{\rm B}_{q(\sigma)}-q(\sigma)\right)=\Lambda^{\rm S}_{q(\ga)}-\Lambda^{\rm S}_{q(\sigma)},
\end{equation*}
We define $\gb_\sigma$ as the unique weak solution to:
\begin{equation}\label{e:gbfromqb2  }
    \cL_\sigma \gb_\sigma = 2a_\ga \left(q(\ga)^{\rm B}_{q(\sigma)}-q(\sigma)\right),\qquad \gb_\sigma  |_{\partial\Omega} = \ga|_{\partial\Omega}.
\end{equation}
Calderón-Zygmund type estimates for elliptic operators ensure that $\gb_\sigma$ has the claimed regularity. 
Then, by \Cref{prop:Frech2}:
\begin{equation}\label{e:gbaux}
    \frac{1}{\sqrt{\sigma}}\d \Phi_\sigma(\gb_\sigma) \frac{1}{\sqrt{\sigma}} =\frac{1}{2}\left(2a_\ga \Lambda^{\rm S}_{q(\sigma)} - \frac{\partial_\nu\gb_\sigma}{\sigma} + \d \Phi_{q(\sigma)}^{\rm S}(\cL_\sigma \gb_\sigma)\right).
\end{equation}
This implies that
\begin{equation}\label{e:almost}
   \d \Phi_\sigma(\gb_\sigma) = \sqrt{\sigma}\left(a_\ga \Lambda^{\rm S}_{q(\ga)}-\frac{\partial_\nu\gb_\sigma}{2\sigma}\right)\sqrt{\sigma} = \Lambda_{\ga} + \frac{\partial_\nu\ga-\partial_\nu\gb_\sigma}{2}. 
\end{equation}
On the other hand, as $\Lambda_{\ga}(1)=\d \Phi_\sigma(\gb_\sigma) (1)=0$ one gets 
% \begin{equation*}
%     0=\frac{1}{\sqrt{\sigma}}\d \Phi_\sigma(\ga) (1) =\frac{1}{2}\left( - \frac{\partial_\nu\ga}{\sqrt{\sigma}} + \d \Phi_{q(\sigma)}^{\rm S}(\cL_\sigma \ga)(\sqrt{\sigma})\right).  
% \end{equation*}
\begin{equation*}
    0=\d \Phi_\sigma(\gb_\sigma) (1) = \Lambda_{\ga}(1) + \frac{\partial_\nu\ga-\partial_\nu\gb_\sigma}{2}=\frac{\partial_\nu\ga-\partial_\nu\gb_\sigma}{2}. 
\end{equation*}
Inserting this in \eqref{e:almost} shows that $\d \Phi_\sigma(\gb_\sigma)  =\Lambda_\ga$. 

To prove the converse, assume that a $\gb_\sigma\in W^{2,\infty}(\Omega)+ \cE'(\Omega)$ exists such that $\d\Phi_\sigma(\gb_\sigma)=\Lambda_\ga$.  We first show that $\gb_\sigma  |_{\partial\Omega} = \ga|_{\partial\Omega}$; to do this, we use a construction due to Brown \cite{Brown} (see also \cite[Theorem~A1]{Andoni}). Fix $x\in\partial\Omega$ and consider the bounded sequence $(f_N^x)_{N\in\IN}$ in $H^{1/2}(\partial\Omega)$ defined in \cite[Theorem~A1]{Andoni}. This result, combined with \eqref{e:Aless} and \eqref{e:formalbornka}, implies that:
\begin{equation*}
    \gb_\sigma(x)=\lim_{N\to\infty}\hp{\cc{f^x_N}}{\d \Phi_\sigma(\gb_\sigma) f_N^x}_{H^{1/2}\times H^{-1/2}}=\lim_{N\to\infty}\hp{\cc{f^x_N}}{\Lambda_\ga f_N^x}_{H^{1/2}\times H^{-1/2}}=\ga(x).
\end{equation*}
Using \eqref{e:gbaux}, \eqref{e:reldtn} and $\gb_\sigma  |_{\partial\Omega} = \ga|_{\partial\Omega}$ we obtain
\begin{equation}\label{e:casiborn}
    \frac{1}{2a_\ga}\d \Phi_{q(\sigma)}^{\rm S}\left(\cL_\sigma\gb_\sigma)\right)=\Lambda^{\rm S}_{q(\ga)} - \Lambda^{\rm S}_{q(\sigma)} -  \frac{\partial_\nu\ga-\partial_\nu\gb_\sigma}{2\ga}.   
\end{equation}
Fix now any $x\in\partial\Omega$ and consider a sequence $(g^x_N)_{N\in\IN}$ in $L^2(\partial\Omega)$ that converges weakly to zero as $N\to\infty$, and such that $(|g_N^x|^2)_{N\in\IN}$ converges to $\delta_x$ weak-$\ast$ in $\cC(\partial\Omega)'$, the space of signed Radon measures on $\partial\Omega$. Since $\d \Phi_{q(\sigma)}^{\rm S}\left(\cL_\sigma\gb_\sigma)\right)$ and $\Lambda^{\rm S}_{q(\ga)} - \Lambda^{\rm S}_{q(\sigma)}$ are compact operators on $L^2(\partial\Omega)$ (by \Cref{l:comp}) and the normal derivatives of $\gb_\sigma$ and $\ga$ are continuous functions on $\partial\Omega$, \eqref{e:casiborn} implies
\[
0=\lim_{N\to\infty}\int_{\partial\Omega}\frac{\partial_\nu\ga-\partial_\nu\gb_\sigma}{2\ga}|g_N^x|^2\d{x}=\frac{\partial_\nu\ga(x)-\partial_\nu\gb_\sigma(x)}{2\ga(x)}.
\]
Therefore, $\partial_\nu\ga|_{\partial\Omega}=\partial_\nu\gb_\sigma  |_{\partial\Omega}$ and \eqref{e:casiborn} allows us to conclude that
\begin{equation*}
    q(\ga)^{\rm B}_{q(\sigma)}:=  \frac{1}{2a_\ga} \cL_\sigma \gb_\sigma + q(\sigma)\in L^\infty(\Omega) + \cE'(\Om),  
\end{equation*}
is the Born approximation of $q(\ga)$ with respect to $q(\si)$, as claimed.
\end{proof}

\subsection{Uniqueness of the Born approximation} \ \label{s:2_3_uniqueness}

In this section we prove that the Born approximation $\gb_{\si}$, if it exists,  is uniquely determined by the DtN map $\La_\ga$ for a certain family of the conductivities  $\si$ that includes the conductivities $\si_{\ka,d}$ defined in \eqref{e:si_ga}. In the proof, we use a standard Complex Geometrical Optics construction.
For $\eta \in \C^d$ define
\begin{equation*}
    e_{\eta}(x) : = e^{\eta \cdot x} ,\qquad \forall x\in \R^d.
\end{equation*}

Let $f\in L^1(\Om)$, we define its Fourier transform  $\cF f$ by the formula 
\begin{equation*}
    \cF f (\xi) := \int_{\Om} e^{-ix\cdot \xi} f(x) \, dx, \qquad \xi \in \R^d, 
\end{equation*}
so that, if we extend $f$ by zero to $\R^d \setminus \Om$, it coincides with the usual Fourier transform on $L^1(\R^d)$. $\cF$ can be extended to $L^1(\Om)+ \cE'(\Om)$ by replacing the integral by the $\R^d$ distributional duality pairing.

We start with the simplest case: we assume $\si=1$ on $\Om$.
\begin{theorem}\label{thm:Born_d=2_uniq_0_Om}
    Let $\ga\in L^\infty(\Om)$ be a conductivity. For $\xi\in \R^d$, let $\eta_1,\eta_2\in \IC^d$ such that $\eta_1 \cdot \eta_1 = \eta_2 \cdot \eta_2  = 0$ and $\eta_1+\eta_2 = -i\xi$. If $\gb_{1}\in L^\infty(\Om) + \cE'(\Om)$ exists, that is, $\d\Phi_{1}(\gb_{1})=\Lambda_{\ga}$,
    % \begin{equation*}
    %     \d\Phi_{1}(\gb_{1})=\Lambda_{\ga},
    % \end{equation*}
    then we have that
    \begin{equation*}
        % \widehat{\gb_1}(\xi) 
        \cF \gb_1 (\xi)= -\frac{2}{|\xi|^2} \hp{e_{\eta_1}}{\Lambda_\gamma e_{\eta_2}}_{H^{1/2}\times H^{-1/2}}, \qquad \forall \xi \in \R^d\setminus \{0\}.
\end{equation*}
    As a consequence $\gb_{1}$ is uniquely determined by $\La_\ga$ in $L^\infty(\Om)+ \cE'(\Om)$.
\end{theorem}
\begin{proof}
    Assume, for simplicity, that $\gb_1 \in L^\infty(\Om)+ L^1_c(\Om)$.
    Since $\eta_j \cdot \eta_j =0$ we notice that 
    \[
       -|\xi|^2 = (\eta_1+\eta_2)^2 = 2\eta_1 \cdot \eta_2,
    \]
    and that $-\Delta e^{\eta_j \cdot x} = 0$ for all $x\in \R^d$ and $j=1,2$. Therefore, since $\eta_1 \cdot \eta_2 = -|\xi|^2/2 $ and $\eta_1 + \eta_2 = -i\xi$ we have that
    \begin{equation*}
        \hp{e_{\eta_1}}{\Lambda_\gamma e_{\eta_2}}_{H^{1/2}\times H^{-1/2}} = \int_{\Om} \gb_1(x) \nabla e^{\eta_1 \cdot x} \cdot \nabla e^{\eta_2 \cdot x} dx = - \frac{|\xi|^2}{2} \int_{\Om}  \gb_1(x) e^{-ix\cdot\xi} \, \d x. 
    \end{equation*}
    The proof for $\gb_1 \in L^\infty(\Om)+ \cE'(\Om)$ is identical: one just needs to replace the integrals by the corresponding duality pairing.
\end{proof}

As expected, the case $\si \neq 1$ is more complicated, so we will study a special case. Since in the proof we require the reduction to the potential case via the mapping $\ga\longmapsto q(\ga)$ defined by \eqref{e:q_funct}, we need to assume that $\ga$ has two derivatives close to the boundary.

\begin{theorem}\label{thm:Born_d=2_uniq_Om}
    For $\ka\in \R$, assume there exists a conductivity  $\si_\ka \in \cC^\infty(\ol{\Om})$ such that\footnote{For example, if $\Om = \IB^d$, the existence of such $\si_\ka$ is guaranteed for  $\ka\in(-\infty,\lambda_{\nu_d,1}^2)$.} $q(\si_\ka) = -\ka$. Let $\ga\in W^{2,\infty}(\Om)$ be a conductivity such that $\ga|_{\IS^1}=a_\ga \si_\ka|_{\partial \Om}$ for some $a_\ga>0$. If $\gb_{\si_\ka}\in W^{2,\infty}(\Om)+\cE'(\Om)$ exists, that is,
    \begin{equation*}
        \d\Phi_{\si_\ka}(\gb_{\si_\ka})=\Lambda_{\ga},
    \end{equation*}
    then we have that
    \begin{equation} \label{e:trace_n}
        \gb_{\si_\ka}  |_{\p \Om} = \ga|_{\p \Om}, \qquad \partial_\nu\gb_{\si_\ka}  |_{\p \Om}=\partial_\nu\ga|_{\p \Om}. 
    \end{equation}
    In addition,  for every  $\xi \in \R^d$ one has
    \begin{equation}\label{e:ftgae}
    \left[\F \cL_{\si_\ka}(\gb_{\si_\ka})\right](\xi)  = 2a_\ga \hp{e_{\eta_1}}{(\Lambda^{\rm S}_{q(\ga)}-\Lambda^{\rm S}_{-\ka}) e_{\eta_2}}_{H^{1/2}\times H^{-1/2}},
    \end{equation}
    where $\eta_1,\eta_2 \in \IC^d$ are vectors satisfying $\eta_1+\eta_2 =-i\xi$ and $\eta_1 \cdot \eta_1 = \eta_2 \cdot \eta_2 = -\ka $.\\
    As a consequence, $\gb_{\si_\ka}$ is uniquely determined by $\La_\ga$ in $ W^{2,\infty}(\Om) +  \cE'(\Om)$.
\end{theorem}
\begin{proof}
    Since $\d\Phi_{\si_\ka}(\gb_{\si_\ka})=\Lambda_{\ga}$ with $\gb_{\si_\ka}\in W^{2,\infty}(\Om) + \cE'(\Om)$ and $q(\si_\ka) = -\ka$, \Cref{p:bqisgb}~ii) implies that the identities \eqref{e:trace_n} hold and the Born approximation
 $q(\ga)^{\rm B}_{-\ka}\in \cE'(\Omega)+L^\infty(\Omega)$ exists and can be obtained by \eqref{e:born_pot}.
 
 Applying \eqref{e:born_pot_def} with $q_0 = -\ka$ we obtain
 \begin{equation*}
   \d \Phi^{\rm S}_{-\ka}\left( q^{\rm B}_{-\ka} + \ka  \right ) = \Lambda^{\rm S}_q - \Lambda^{\rm S}_{-\ka} ,
\end{equation*}
and using \eqref{e:frechp} and \eqref{e:born_pot} we get
\begin{equation*}
     2a_\ga  \hp{e_{\eta_1}}{(\Lambda^{\rm S}_{q(\ga)}-\Lambda^{\rm S}_{-\ka}) e_{\eta_2}}_{H^{1/2}\times H^{-1/2}} = \int_{\Om}  \cL_{\si_\ka}(\gb_{\si_\ka})(x) v_{e_{\eta_1}}^{-\ka}(x)   v_{e_{\eta_2}}^{-\ka}(x) \, \d x.
\end{equation*}
For any $\xi \in \R^d$, $d\ge 2$,  there always exist  $\eta_1,\eta_2 \in \IC^d$ such that $\eta_1+\eta_2=-i\xi$ and $\eta_1 \cdot \eta_1 = \eta_2 \cdot \eta_2 = -\ka $ hold. In particular, the last identity implies that $(-\Delta-\ka)e^{\eta_j \cdot x} =0$. 
Thus $v_{e_{\eta_j}}^{-\ka}(x) = e^{\eta_j \cdot x}$ on $\Om$ and
\begin{multline*}
     2a_\ga  \hp{e_{\eta_1}}{(\Lambda^{\rm S}_{q(\ga)}-\Lambda^{\rm S}_{-\ka}) e_{\eta_2}}_{H^{1/2}\times H^{-1/2}} = \int_{\Om}  \cL_{\si_\ka}(\gb_{\si_\ka})(x) e^{\eta_1 \cdot x}  e^{\eta_2 \cdot x} \, \d x
     \\ =  \left[\F \cL_{\si_\ka}(\gb_{\si_\ka})\right](\xi).
\end{multline*}
This proves that \eqref{e:ftgae} holds.

Now, suppose that $\ga$ has two Born approximations $w_1,w_2\in  W^{2,\infty}(\Om)+\cE'(\Om) $. Then $\Lambda_{\ga}= \d\Phi_{\si_\ka}(w_1)=\d\Phi_{\si_\ka}(w_2)$ and \eqref{e:trace_n}, \eqref{e:ftgae} show that that for $j=1,2$ $w_j|_{\p \Om}= \ga|_{\p \Om}$, $\p_\nu w_j|_{\p \Om}= \p_\nu\ga|_{\p \Om}$ and $\F \cL_{\si_\ka}(w_1)=\F \cL_{\si_\ka}(w_2)$. Injectivity of the Fourier transform on tempered distributions ensures that $\cL_{\si_\ka}(w_1) = \cL_{\si_\ka}(w_2)$, 
and therefore $\rho=w_1-w_2$ is a distribution in $W^{2,\infty}(\Om) + \cE'(\Om)$ that solves
\begin{equation*}
        \cL_{\si_\ka}\rho=0,\quad \text{on }\Om,\qquad \rho|_{\p \Om}=0,\quad \p_\nu \rho|_{\p \Om} =0.
\end{equation*} 
By elliptic regularity we know that $\rho \in \cC^\infty(\Om)$, so, in fact, $\rho \in W^{2,\infty}(\Om)$. Thus, since $\rho$ vanishes at the boundary,  we get that $\rho = 0$ identically in $\Om$. This finishes the proof of the theorem.
\end{proof}

\begin{remark} \label{r:uniq_frechet}\
\begin{enumerate}[i)]
    \item The previous theorem can be extended to show the unique determination of $\gb_\si$ for any conductivities $\si \in \cC^\infty(\ol{\Om})$ and $\ga \in W^{2,\infty}(\Om)$. It suffices to replace the plane waves $e_\eta$ by standard Complex Geometrical Optics solutions of the Schrödinger equation $(-\Delta+ q(\si))v = 0$.
    \item The proofs of \cref{thm:Born_d=2_uniq_0_Om} and \Cref{thm:Born_d=2_uniq_Om} can be easily adapted to show, respectively, that:
    \begin{itemize}
        \item If $ \d\Phi_{1}(\ga_1)  = \d\Phi_{1}(\ga_2) $ for $\ga_1,\ga_2 \in  L^\infty(\Om)+ \cE'(\Om)$, then $\ga_1 = \ga_2$ on $\Om$.
        \item If $ \d \Phi^{\rm S}_{-\ka}(q_1)  =  \d \Phi^{\rm S}_{-\ka}(q_2) $ for  $q_1,q_2 \in  L^\infty(\Om)+ \cE'(\Om)$, then $q_1 = q_2$ on $\Om$.
    \end{itemize}
    \end{enumerate}
\end{remark}

\subsection{Existence of the Born approximation in the radial case} \ \label{s:2_4}

In this section, we prove \Cref{t:existencegb}.
Here we consider $\Om= \IB^d$, so that $\p \Om = \p \IB^d = \IS^{d-1}$. We will use the notation $L^p_\rad(\IB^d)$ to refer to the  subspace of radial functions in $L^p(\IB^d)$.

If $\ga$ is a radial conductivity, then $\La_\ga$ is a rotation invariant operator on $L^2(\IS^{d-1})$. This is also the case for $\d \Phi_\sigma(\ga)$  when $\si$ and $\ga$ are radial,
% any conductivity $\si \in L^\infty_\rad(\IB^d)$ and for any   $\ga \in 
% L^\infty_\rad(\IB^d)$. 
and for $\La_{q}^{\rm S}$ and $\d \Phi^{\rm S}_{q_0}(q)$ with $q,q_0$ radial. If $T$ is any of the aforementioned operators, rotation invariance implies that $T$ is diagonal in the basis of spherical harmonics. In particular, $\forall m \in \N_0$ there exists a $\lambda_m[T] \in \IC$ such that
\begin{equation*}
    T |_{\mathfrak H_{m,d}} = \lambda_m[T] \Id|_{\mathfrak H_{m,d}} , 
\end{equation*}
where $\mathfrak H_{m,d} \subset \cC^\infty(\IS^{d-1})$ stands for the subspace of spherical harmonics of order $m$. One can verify by direct computation that $\la_0[\Lambda_\ga] =0$ for all conductivities $\ga \in L^\infty_\rad(\IB^d)$ and that $\la_m[\La_1] = \la_m[\La_0^{\mathrm{S}}] = m$ for all $m\in \N_0$.

Recall the notation.
\begin{equation*}
    \sigma_{\ka,d} (x)= \left(c_d\frac{J_{\nu_d}(\sqrt{\ka} |x|)}{(\sqrt{\ka}|x|)^{\nu_d}}\right)^2,\qquad \nu_d=\frac{d-2}{2},\quad c_d=\Gamma(\nu_d+1)2^{\nu_d}.
\end{equation*}
The function $\sigma_{\ka,d}\in\cC^\infty(\cc{\Omega})$ is an admissible (strictly positive) conductivity as soon as $\ka\in(-\infty,\lambda_{\nu_d,1}^2)$, where $\lambda_{\nu_d,1}$ stands for the first positive zero of the Bessel function $J_{\nu_d}$; the normalization constant $c_d$ ensures that $\sigma_{0,d}=1$.

\begin{proposition}\label{p:Born_exist_aux}
   Let $\ka\in(-\infty,\lambda_{\nu_d,1}^2)$ and $\ga \in W^{2,\infty}(\IB^d)$ be a conductivity, and let $q: = q(\ga)$ given by \eqref{e:q_funct} be radial.  Then there exists a radial function $q^{\rm B}_{-\ka} \in L^1(\IB^d,\R)$ such that
    \begin{equation*}
        \d\Phi^{\rm S}_{-\ka}(q^{\rm B}_{-\ka}+\ka)=\Lambda^{\rm S}_{q}-\Lambda^{\rm S}_{-\ka}.
    \end{equation*}
    In addition $q-q^{\rm B}_{-\ka} \in \cC(\ol{\IB^d}\setminus\{0\})$.
\end{proposition}
This proposition is based on results from  \cite{MMS-M25}. We postpone momentarily its proof in order to prove \Cref{t:existencegb} first.

\begin{proof}[Proof of \Cref{t:existencegb}]
    Direct computation shows that $q(\sigma_{\ka,d})=-\ka$ and that $q(\ga)\in L^\infty(\IB^d)$ is real and radial. 
    Then \Cref{p:Born_exist_aux} proves that the born approximation $q(\ga)^{\mB}_{-\ka}$ exists and belongs to the space $L^1_\rad(\IB^d)$.  
    In fact, since $q-q(\ga)^{\rm B}_{-\ka}  \in \cC(\IB^d\setminus\{0\})$ and $q(\ga)$ is bounded, we have that $q(\ga)^{\mB}_{-\ka} \in L^\infty(\IB^d) + L^1_c(\IB^d)$. 
    This last condition allows us to apply \Cref{p:bqisgb} i), which implies that  $\gb_{\si_{\ka,d}}\in W^{2,\infty}(\Omega)+ W^{1,1}_c(\Omega)$ exists.
    
    Moreover, $\gb_{\si_{\ka,d}}$ must be real and radial since it is obtained by solving \eqref{e:gbfromqb}, an elliptic, radially symmetric boundary value problem with real coefficients in $\IB^d$ with a constant real Dirichlet datum $ \gb_{\si_{\ka,d}}  |_{\partial\IB^d} = \ga|_{\partial\IB^d}$ (recall that  $\ga$ is radial). 
    As a consequence of \Cref{p:bqisgb}~(i) it also follows that  $\p_\nu \gb_{\si_{\ka,d}}  |_{\partial\IB^d} = \p_\nu \ga|_{\partial\IB^d}$. 
    The fact that $\gb_{\si_{\ka,d}}$ is unique in  $ W^{2,\infty}(\IB^d) + \cE'(\IB^d)$ is a consequence of \Cref{thm:Born_d=2_uniq_Om}.

    \textit{Case $\ka =0$.} There is alternative proof  that does not require the use of \Cref{p:Born_exist_aux} for the case $\ka=0$. 
    Let $\ell,m\in\IN$. If $f \in \gH_{\ell,d}$ and $g \in \gH_{m,d}$ are spherical harmonics  then for every 
    $\ga \in W^{2,1}_\rad(\IB^d)$, one has
    \begin{equation} \label{e:moments_ka0}
        \hp{\cc{f}}{\d \Phi_{1}(\ga) g}_{H^{1/2}\times H^{-1/2}} =   \ell(2\ell +d-2) \frac{\hp{f}{g}_{L^2(\IS^{d-1})}}{|\IS^{d-1}|}\int_{\IB^d} \gamma(x)|x|^{2\ell-2}\,\d x.
    \end{equation}
    This follows  from \eqref{e:formalbornka}, using that
        $\nabla u_f^1(x)= |x|^{\ell-1} \left(\ell f(\hat{x})\hat{x} + \nabla_{\IS^{d-1}} f(\hat{x})\right)$, where $\hat{x}:=x/|x|$, integrating by parts $\nabla_{\IS^{d-1}}$ and using that $-\Delta_{\IS^{d-1}} f = \ell(\ell +d-2)f$. 
    Taking above $g:=f$ with $\|f\|_{L^2(\IS^{d-1})}=1$ gives
    \begin{equation*}
        \la_\ell\left[\d \Phi_{1}(\ga) \right] =   \frac{ \ell(2\ell +d-2)}{|\IS^{d-1}|}\int_{\IB^d} \gamma(x)|x|^{2\ell-2}\,\d x.
    \end{equation*}
    Hence, \eqref{e:mainthm} holds iff, for all $\ell \in \N$
    \begin{equation*}
        \la_\ell\left[\La_\ga \right] =   \frac{ \ell(2\ell +d-2)}{|\IS^{d-1}|}\int_{\IB^d} \gb_1(x)|x|^{2\ell-2}\,\d x.
    \end{equation*}
    In \cite[Theorem 1]{radial_cond} it is proved that there exist a real and radial function $\gb_1 \in W^{2,1}(\IB^d)$ such that the identity above holds.
    This proves that \eqref{e:mainthm} holds for $\ka = 0$. In fact, is the unique solution in $L^\infty(\IB^d) +\cE'(\IB^d)$ by \Cref{thm:Born_d=2_uniq_0_Om}. The identity of the traces at the boundary and that  $\gb_1 \in \cC^1(\ol{\IB^d}\setminus\{0\})$  also follow from \cite[Theorem 1]{radial_cond}. This proves the theorem for $\ka=0$.
\end{proof}

\begin{proof}[Proof of \Cref{p:Born_exist_aux}]
    Since $q$ is radial, $\Lambda^{\rm S}_{q(\ga)}$ and $\Lambda^{\rm S}_{-\ka}$ are invariant by rotations and its eigenspaces are $\gH_{m,d}$ for all $m\in\IN_0$. 
    
    In \cite[Theorems 1 and 3.1]{MMS-M25}, it is shown that $\d\Phi^{\rm S}_{-\ka}$ extends to the space $\cB_d:=\cE'_\rad(\IB^d)+L^1_\rad(\IB^d)$ 
    % (here,  $\cE'(\IB^d)$ stands for the space of distributions with support contained in $\IB^d$) 
    and that for every $q\in L^\infty(\IB^d)$ one can find a radial distribution $q^{\rm B}_{-\ka}\in\cB_d$ such that $\d\Phi^{\rm S}_{-\ka}(q^{\rm B}_{-\ka}+\ka)=\Lambda^{\rm S}_{q}-\Lambda^{\rm S}_{-\ka}$. This is stated, specifically in \cite[Equation (1.15)]{MMS-M25}, which in our 
    notation\footnote{The potential in  \cite{MMS-M25} corresponds in our notation to $q+\ka$, and the DtN map $\La_{q,\ka}$ in \cite{MMS-M25} corresponds here to $\La_{q+\ka}^\mathrm{S}$.}
    reads
    \begin{equation} \label{e:mom_fix_en}
        \lambda_\ell\left[\d\Phi^{\rm S}_{-\ka}(q^{\rm B}_{-\ka}+\ka)\right] = \lambda_\ell\left[\Lambda^{\rm S}_{q}-\Lambda^{\rm S}_{-\ka}\right] \qquad \forall \ell \in \N_0.
    \end{equation} 
    The fact that the distribution $q^{\rm B}_{-\ka}$ is a function outside the origin and that $q-q^{\rm B}_{-\ka} \in \cC(\IB^d\setminus\{0\})$ it is proved in  \cite[Theorem 2]{MMS-M25}.
    
    It remains to prove that  $q = q(\ga)$ implies that $q^{\rm B}_{-\ka} \in L^1(\IB^d)$. 
    For $\ka=0$, this is proved in \cite[Proposition 2.1]{radial_cond}. The case $\ka\neq 0$ is essentially contained in the works \cite{MMS-M25,radial_cond}, though it is not stated explicitly in them. Here we outline the chain of arguments that leads to proving that $q^{\rm B}_{-\ka} \in L^1(\IB^d)$:
        \begin{itemize}
        \item Let $Q \in L^1_\loc(\R_+)$ such that $Q(-\log|x|)= q(x)|x|^2$. Since $q$ arises from a conductivity  \cite[Theorem 3.2 and Lemma 3.9]{radial_cond} imply that  \cite[Identity (3.13)]{radial_cond} holds for all $z\ge \nu_d$.
        \item The identity \cite[Identity (3.13)]{radial_cond} is the same as \cite[Equation (4.11)]{MMS-M25}, so, again, it holds for all $z\ge z_{Q} : =  \nu_d$ and $d\ge 2$ (in particular, in $d=2$ this implies that the assumption $z_Q>0$ in \cite[Equation (4.11)]{MMS-M25} can be relaxed  to $z_Q = 0$.)
        \item On the other hand, the conclusion of \cite[Lemma 4.3]{MMS-M25} holds with $z_\ka=0$ since $\ka< \la_{0,1}^2$.
        % \item As a consequence of the previous two steps, the conclusion of \cite[Theorem 4.4]{MMS-M25}  is valid for $ s: = \max(z_Q,z_\ka) =  \nu_d$. Therefore $A_{Q,\ka}e^{-\nu_d(\cdot)} \in L^1(\R_+)$. 
        \item As a consequence of the two previous steps, the conclusion of \cite[Proposition 4.6]{MMS-M25} is valid with $\ell_q := \max(z_Q,z_\ka) -  \nu_d=0$ (the key here is that one can take $z_Q=\nu_d$ in \cite[Proposition 4.6]{MMS-M25} instead of the value of $z_Q$ given by \cite[Identity (4.12)]{MMS-M25}.). In particular, \cite[Identity (4.19)]{MMS-M25} holds for a function $q^s_\ka \in L^1(\IB^d)$. This identity is equivalent to \eqref{e:mom_fix_en} taking $q^{\rm B}_{-\ka} = q^s_\ka-\ka$, which proves that $q^{\rm B}_{-\ka} \in L^1(\IB^d)$.
    \end{itemize}
    This finishes the proof of the proposition.
\end{proof}

\section{The Born approximation of a conductivity on the unit Disk}\label{s:sec_disk}

 In this section, for $d=2$  and $\Om = \ID:=   \IB^2$, the unit disk, we explore the connection between the Born approximation and the solution to certain complex moment problems. In particular we derive explicit expressions for the matrix elements of the Fréchet differential $\d \Phi_{\si_\ka}(\ga)$ in the spherical harmonics that will be useful for the numerical reconstruction of the Born approximation in \Cref{s:numerical}.

We first need some notation and definitions. For $\ka\in (-\infty,\lambda_{0,1}^2)$, define
\begin{equation}
     \sigma_\ka(x) := \sigma_{\ka,2}(x)=J_0(\sqrt{\ka} |x|)^2,
\end{equation}
 for shortness.
Note that $\sigma_\ka\in \cC^\infty(\ol{\ID})$, $\sigma_0=1$ and, for $\ka<0$, one has $\sigma_\ka(x)=I_0(\sqrt{|\ka|} |x|)^2$.

We consider the standard $L^2$-orthonormal basis of spherical harmonics $\{e_\ell\}_{\ell\in\IZ}$ on $\IS^1 = \partial \ID$, defined by
\begin{equation*}%\label{e:exps}
    e_\ell(e^{i\theta}):=\frac{e^{i\ell\theta}}{\sqrt{2\pi}}, \qquad \forall \ell\in\IZ.
\end{equation*}
With a slight abuse of notation, in Euclidean coordinates we will  write $e_{\ell}(\widehat{x})$, where $\widehat{x} = x/|x|$, for the $0$-homogeneous extension of $e_\ell$ from $\IS^1$ to $\ID\setminus\{0\}$.

 For $\ell\in\Z$,
we write,
\begin{equation} \label{e:phi_def}
        \phi^\ka_\ell(r):=\left\{ \begin{array}{ll}\dfrac{J_{|\ell|}(\sqrt{\ka}r)}{J_{|\ell|}(\sqrt{\ka})}, &\ka\neq 0,\\
        r^{|\ell|}, &\ka= 0.
        \end{array}\right.
\end{equation}

For every $\ell,m\in\IZ$ and $q\in L^1(\ID)+\cE'(\Om)$ we define the moments
\begin{equation} \label{e:mom_ka}
    \mu_{\ell,m}^\ka[q] := \int_{\ID} q(x)\phi^\ka_{\ell}(|x|)\phi^\ka_{m}(|x|)\overline{e_{\ell}(\widehat{x})}e_{m}(\widehat{x}) \, \d{x}.
\end{equation}

\subsection{The Born approximation and the complex moment problem when \texorpdfstring{$\kappa = 0$}{k=0}} \

 Here we fix $\ka=0$. When convenient, we will use the notation $z=x_1+i x_2$ and denote by $\d m(z)$ the Lebesgue area element on $\IC$.
One can verify that in complex notation $\mu_{\ell,m}^0[q]$ can be written as
\begin{equation} \label{e:cm}
    \mu_{\ell,m}^0[q]= \frac{1}{2\pi}\int_{\ID} q(z) \ol{z}^\ell z^m \, \d m(z), \qquad \forall \ell,m\in \N_0,
\end{equation}
since $\sqrt{2\pi}\phi^0_{\ell}(|x|)e_{\ell}(\widehat{x}) = z^\ell$ for $\ell \ge 0$.

Given a sequence $(M_{\ell,m})_{\ell,m \in \N_0}$ in $\IC$, the \textit{complex moment problem} in $\ID$ is the inverse problem that consists in finding a function $f$ or distribution such that
\begin{equation} \label{e:cmp} 
    \mu_{\ell,m}^0[f] = M_{\ell,m},\qquad \forall\ell,m \in \N_0.
\end{equation}

The following lemma allows us to interpret the matrix elements of $ \d \Phi_{1}(\gamma)$ and $\d \Phi_{0}^{\rm S}(q)$ as complex moments of $\ga$ and $q$, respectively.
\begin{lemma}\label{r:freemoments} 
     For every $\gamma\in L^\infty(\ID)+\cE'(\Om)$, it holds that
    \begin{equation*}
        \ell m\leq 0 \;\implies\;  \hp{\cc{e_{\ell}}}{\d \Phi_{1}(\gamma) e_m}_{H^{1/2}\times H^{-1/2}}=0,
    \end{equation*}
    and for every $\ell, m\in\N$
    \begin{align*}
            \hp{\cc{e_{\ell}}}{\d \Phi_{1}(\gamma) e_m}_{H^{1/2}\times H^{-1/2}}&= 2\ell m\mu_{\ell-1,m-1}^0[\ga],\\
        \hp{\cc{e_{\ell}}}{\d \Phi_{0}^{\rm S}(q) e_m}_{H^{1/2}\times H^{-1/2}} &= \mu_{\ell,m}^0[q].
    \end{align*}
\end{lemma}    
\begin{proof} 
    Writing  \eqref{e:formalbornka} in complex notation yields
    \begin{equation*}
    \hp{\cc{e_{\ell}}}{\d \Phi_{1}(\gamma) e_m}_{H^{1/2}\times H^{-1/2}} =  \int_\ID \ga(z)  \ol{\nabla u^1_{e_\ell}(z)} \cdot \nabla u^1_{e_m}(z)\, \d m(z). 
    \end{equation*}
    We have that $u_{e_\ell}^1(z) = (2\pi)^{-\frac12} z^\ell$ if $\ell \ge 0$, and $u_{e_\ell}^1(z) = (2\pi)^{-\frac12} \ol{z}^{|\ell|}$ if $\ell < 0$. By direct computation, $\nabla z^\ell = \ell z^{\ell-1} (1,i)$ if $\ell\ge 0$, and $\nabla \ol{z}^{|\ell|} = |\ell| \ol{z}^{|\ell|-1} (1,-i)$ if $\ell <0$. This implies that $\nabla \ol{u_{e_\ell}^1} \cdot \nabla u_{e_m}^1 = 0$ pointwise if $\ell m \le 0$, which proves the first statement. The second identity follows now directly. The proof of the last identity is also immediate combining \eqref{e:cm} and \eqref{e:Aless2}, and using that $v_{e_\ell}^0(z) = u_{e_\ell}^1(z)$.  
\end{proof}

As a consequence of the previous result, we conclude that the Born approximations $\gb_1$ and $\qe_{0}$, if they exist, are solutions of the complex moment problem. In fact, we have the following result. 
\begin{proposition}\label{prop:Born_d=2_uniq_0}
    Let $\ga\in L^\infty(\ID)$ be a conductivity. If $\gb_{1}\in L^\infty(\ID)+\cE'(\Om)$ exists, that is, $\d\Phi_{1}(\gb_{1})=\Lambda_{\ga}$, then $\gb_1$ is a solution of the moment problem
    \begin{equation}\label{e:cmpBorn}
    \mu_{\ell,m}^0[\gb_1] =  \frac{\hp{\cc{e_{\ell+1}}}{\Lambda_\gamma e_{m+1}}_{H^{1/2}\times H^{-1/2}}}{2(\ell+1)(m+1)}, \qquad \forall\ell,m \in \N_0.
    \end{equation}
    Let $q \in L^\infty(\ID,\IR) $ such that $0\notin \Spec_{H^1_0(\ID)}(-\Delta +q)$. If $\qe_0 \in L^\infty(\ID) +\cE'(\ID)$ exists, that is,
    $\d \Phi^{\rm S}_{0}(\qe_{0})= \Lambda^{\rm S}_q - \Lambda^{\rm S}_{0}$,
 then $\qe_0$ is a solution of the moment problem
    \begin{equation*}
    \mu_{\ell,m}^0[\qe_{0}] =  \hp{\cc{e_{\ell}}}{(\Lambda^{\rm S}_{q}-\Lambda^{\rm S}_{0}) e_m}_{H^{1/2}\times H^{-1/2}}, \qquad \forall\ell,m \in \N_0.
    \end{equation*}
\end{proposition}
\begin{proof}
   The first statement is immediate from \Cref{r:freemoments}. For the potential case, notice that $0 \notin \Spec_{H^1_0(\ID)}(-\Delta +q)$ is required so that $\La_q$ is well defined. Then it also follows from \Cref{r:freemoments}.
\end{proof}

\subsection{Connection with moment problems in the general case} \

As expected, the connection with other moment problems is not so simple in the case $\ka \neq 0$. The main difficulty is that the moment problems associated to the potential and to the conductivity cases are not the same any more. The moments associated to the potential case are the ones defined in \eqref{e:mom_ka}, while the moments for the conductivity case have a more complicated expression. Nonetheless, it is worth obtaining an explicit expression for them since it will be useful for the numerical reconstruction of the Born approximation. 

First, for $\ga \in L^\infty(\ID)+ \cE'(\ID)$ we define the moments
\begin{equation*}
    \gm_{\ell,m}^\ka[\ga] : = \hp{\cc{e_{\ell}}}{\d \Phi_{\sigma_\ka}(\gamma) e_m}_{H^{1/2}\times H^{-1/2}}.
\end{equation*}

\begin{lemma}\label{l:moment_formula}
    The following assertions hold.
    \begin{enumerate}[i)]
    \item 
        Recall the definition of $\mu_{\ell,m}^\ka$ in \eqref{e:mom_ka}. For every $q\in L^\infty(\ID)  + \cE'(\Om)$,  one has
    \begin{equation*} 
    \hp{\cc{e_{\ell}}}{\d \Phi_{-\ka}^{\rm S}(q) e_m}_{H^{1/2}\times H^{-1/2}} = \mu_{\ell,m}^\ka[q].  
    \end{equation*}
    \item
    For every $\gamma\in L^\infty(\ID) + \cE'(\Om)$, $\ka\in (-\infty,\lambda_{0,1}^2)$, and $\ell,m\in\IZ$ one has
    \begin{multline} \label{e:mom_ga_def}
        \gm_{\ell,m}^\ka[\ga]  \\= J_0(\sqrt{\ka})^2 \! \!\int_{\ID}\frac{\gamma(x)}{\sigma_\ka(x)}\frac{\ell m+(|\ell| + \psi^\ka_\ell(|x|))(|m|+\psi^\ka_m(|x|))}{|x|^2}\phi^\ka_\ell(|x|)\phi^\ka_m(|x|)\overline{e_{\ell}(\widehat{x})}e_{m}(\widehat{x}) \,\d{x},   
    \end{multline}
    where $\widehat{x} = x/|x|$ and
    \begin{equation*}
        \psi_\ell^\ka(r):=\sqrt{\ka}r\left(\frac{J_{1}(\sqrt{\ka}r)}{J_0(\sqrt{\ka}r)}-\frac{J_{|\ell| + 1}(\sqrt{\ka}r)}{J_{|\ell|}(\sqrt{\ka}r)}\right),   \quad r>0,\qquad  \psi_\ell^\ka(0):=\lim_{r\to 0^+}\psi_\ell^\ka(r)=0. 
    \end{equation*} 
\end{enumerate}
\end{lemma}
As an immediate consequence of this lemma we obtain the following proposition.
\begin{proposition}\label{prop:Born_d=2_uniq}
 Let $\ka\in (-\infty,\lambda_{0,1}^2)$ and $\ga\in W^{2,\infty}(\ID)$ be a conductivity such that $\ga|_{\IS^1}=a_\ga \sigma_\ka|_{\IS^1}$ for $a_\ga>0$. If $\gb_{\sigma_\ka}\in W^{2,\infty}(\ID)+\cE'(\Om)$ exists, that is,
        $\d\Phi_{\sigma_\ka}(\gb_{\sigma_\ka})=\Lambda_{\ga}$,
    then $\gb_{\si_\ka}$ is a solution of the moment problem
    \begin{equation*} 
       \gm_{\ell,m}^\ka[\gb_{\si_\ka}] =  \hp{\cc{e_{\ell}}}{\La_\ga e_m}_{H^{1/2}\times H^{-1/2}} .
    \end{equation*}
    In addition, let $q \in L^\infty(\ID,\IR)$ such that $0\notin \Spec_{H^1_0(\ID)}(-\Delta +q)$. \\
    If $\qe_{-\ka} \in L^\infty(\ID) +\cE'(\Om)$ exists, that is,
    $\d \Phi^{\rm S}_{-\ka}(\qe_{-\ka}+\ka)= \Lambda^{\rm S}_q - \Lambda^{\rm S}_{-\ka}$,
 then $\qe_{-\ka}$ is a solution of the moment problem
    \begin{equation*}
    \mu_{\ell,m}^\ka[\qe_{-\ka}+\ka] =  \hp{\cc{e_{\ell}}}{(\Lambda^{\rm S}_{q}-\Lambda^{\rm S}_{-\ka}) e_m}_{H^{1/2}\times H^{-1/2}} .
    \end{equation*}
\end{proposition}
For the proof of \Cref{l:moment_formula} we need some preliminary results.

For shortness, we use $u^{\gamma}_\ell:=u^{\ga}_{e_\ell}$ and $v^q_\ell := v^q_{e_\ell}$, respectively, to denote the solution of \eqref{e:conduct} and \eqref{e:schrod} on $\Omega=\ID$ with boundary datum $e_\ell\in\cC^\infty(\partial\ID)$.
\begin{lemma}
    For every $\ell\in \IZ$ and $\ka\in (-\infty, \lambda_{0,1}^2)$ one has:
    \begin{equation}\label{e:ssol}
        u_\ell^{\ka}(re^{i\theta}):=u_\ell^{\sigma_\ka}(re^{i\theta})= \frac{\phi^\ka_{\ell}(r)}{ \phi^\ka_0(r)} e_\ell(e^{i\theta}).
    \end{equation}
\end{lemma}
\begin{proof}
    Let $w_\ell^\ka:=\sqrt{\sigma_\ka}u_\ell^{\sigma_\ka}$; then, noting that $\Delta\sqrt{\sigma_\ka}=-\ka\sqrt{\sigma_\ka}$ one gets 
    \begin{equation*}
    0= - \nabla \cdot ( \sigma_\ka \nabla ( \sigma_\ka^{-1/2} w_\ell^\ka)) = \sqrt{\sigma_\ka}(-\Delta - \ka) w_\ell^\ka ,\quad \text{ on }\ID,
    \end{equation*}
    and
    \begin{equation*}
        w_\ell^\ka(e^{i\theta}) = J_0(\sqrt{\ka})e_\ell(e^{i\theta})
    \end{equation*}
    This forces $w_\ell^\ka(x)=J_0(\sqrt{\ka})v^{-\ka}_{\ell}(x)$.
    Recall the definition of $\phi^\ka_\ell$ in \eqref{e:phi_def}. Note right away that, for every $\ka\in (-\infty, \lambda_{0,1}^2)$ and $\ell\in\IZ$, 
    \begin{equation} \label{e:v_ka}
        v_\ell^{-\ka}(re^{i\theta}) = \phi^\ka_\ell(r)e_\ell(e^{i\theta}),
    \end{equation}
(see \cite[Lemma~2.1]{MMS-M25} for a proof). This finishes the proof of the lemma.
\end{proof}
\begin{proof}[Proof of \Cref{l:moment_formula}]
     $i)$ Follows from \eqref{e:frechp} using that $v^{-\ka}_{e_\ell}(x) = v^{-\ka}_\ell(x) =  \phi^\ka_\ell(|x|) e_\ell(\widehat{x})$ by \eqref{e:v_ka}.
    We now prove  $ii)$.
    Denote the complex derivatives as $\p_+= \cp$ and $\p_-=\p$. Using 
    \begin{equation*}
        \p_\pm=\frac{1}{2}e^{\pm i\theta}\left(\partial_r \pm \frac{i}{r}\,\partial_\theta\right),
    \end{equation*}
    one gets by direct computation on \eqref{e:ssol} using elementary properties of Bessel functions,
    \begin{equation*}
        \p_\pm u^\ka_\ell(re^{i\theta}) = e^{\pm i\theta}\frac{\psi^\ka_\ell(r)+|\ell|\mp \ell}{2r}u^\ka_\ell(re^{i\theta}).
    \end{equation*}
    The result follows from \Cref{prop:Frech} and the identity 
    \begin{equation*}
        \ol{\nabla u_\ell^\ka}\cdot\nabla u_m^\ka = 2(\ol{\p_+ u_\ell^\ka}\p_+ u_m^\ka + \ol{\p_- u_\ell^\ka}\p_- u_m^\ka). \qedhere
    \end{equation*}
\end{proof}

\subsection{Uniqueness of the complex moment problems}\

The fact that we have uniqueness determination results for the Born approximation and the Fréchet derivatives (see \Cref{r:uniq_frechet}~(ii)), suggests a similar uniqueness result should hold for the moment problem associated to the moments $\mu_{\ell,m}^\ka[q]$. We start with the case $\ka= 0$.
For $\xi=(\xi_1,\xi_2)\in\R^2$ we set
\[
\zeta := -\frac{\xi_1-i\xi_2}{2}\in\C.
\]

\begin{lemma} \label{l:rep_moment_prob}
    Let $f\in L^1(\ID)$. Then, for every $\xi \in \R^2$, 
    \begin{equation*}
        \cF f(\xi) =  
         2\pi\sum_{\ell,m=0}^\infty    \frac{i^{\ell+m}}{\ell!m!}  \mu_{\ell,m}^0[f]\ol{\zeta}^\ell \zeta^m.
    \end{equation*}
    In particular, two functions in $L^1(\ID)$ having the same complex moments \eqref{e:cm} must necessarily coincide.
\end{lemma}
\begin{proof}
    Use that $e^{-ix\cdot \xi} = e^{i(\zeta z +\ol{\zeta z})}$ with $z= x+iy$, and a Taylor series of the exponential functions $e^{i\zeta z }$ and $e^{i\ol{\zeta z}}$. The uniqueness part follows from the formula, since it implies that the Fourier transforms of any two functions in $L^1(\ID)$ with the same complex moments must coincide.
\end{proof}
\begin{remark} \label{r:prop_subspace}
    As a consequence of this lemma and \eqref{e:cmpBorn} one gets
    \begin{equation*}
        % \widehat{\gb_1}(\xi) 
        \cF \gb_1 (\xi)=  \pi\sum_{m,\ell=0}^\infty    \frac{i^{\ell+m}}{(\ell+1)!(m+1)!} \hp{\cc{e_{\ell+1}}}{\Lambda_\gamma e_{m+1}}_{H^{1/2}\times H^{-1/2}} \ol{\zeta}^\ell \zeta^m.
    \end{equation*}
    This formula shows that only the matrix elements $\hp{\cc{e_{\ell+1}}}{\Lambda_\gamma e_{m+1}}_{H^{1/2}\times H^{-1/2}}$ with $m,\ell>0$ are required to determine $\gb_1$, not the complete operator $\La_\ga$.
\end{remark}

We now prove an analogous result to \Cref{l:rep_moment_prob} for $\ka \neq 0$. As expected, the situation is a bit more complicated when $\ka\neq0$. 
\begin{theorem}\label{thm:Fourier d=2}
    Let $f\in L^1(\ID)$ and $\kappa\in\R\setminus\{0\}$, then for $\xi\in \R^2\setminus\{0\}$ we have
\begin{equation*}
    \F f(\xi)=
       2\pi \sum_{\ell,m\in\Z} i^{\ell+m}\left(\frac{|\zeta|+\sqrt{|\zeta|^2-\ka}}{|\zeta|\sqrt{\ka}}\right)^{\ell+m}J_{\ell}(\sqrt{\kappa})J_m(\sqrt{\kappa})\mu_{\ell,m}^\ka[f]\ol{\zeta}^{\ell}\zeta^m,
\end{equation*}
and the series is absolutely convergent. In particular, two functions in $L^1(\ID)$ having the same $\mu_{\ell,m}^\ka$-moments are necessarily equal.
\end{theorem}

The proof of \Cref{thm:Fourier d=2} relies on an explicit computation. Given $\eta\in\IC^2$, let 
\begin{equation*}
    g_\eta(x):= e^{\eta\cdot x}.
\end{equation*}
\begin{lemma}\label{prop:Proyection d=2}
Let $\eta=(\eta_1,\eta_2)\in\C^2$ such that $\eta\cdot\eta=\eta_1^2+\eta_2^2\neq 0$. For every $\ell\in\IZ$ the following holds.
\begin{equation*}
\hp{e_\ell}{g_{\eta}}_{L^2(\IS^1)}=
\displaystyle \sqrt{2\pi}\left(\frac{\eta_1-i\eta_2}{\sqrt{\eta\cdot\eta}}\right)^{\ell}I_{\ell}\left(\sqrt{\eta\cdot\eta}\right).
\end{equation*}
\end{lemma}

\begin{proof}
From the generating function for the Bessel coefficients \cite[Section 2.1]{Watson1944}, and the identity $I_n(z)=(-i)^n J_n(iz)$ for $n\in\IZ$, we obtain
\begin{equation*}
    e^{\frac{z}{2}(w+w^{-1})}=\sum_{n\in\IZ}I_n(z)w^n.
\end{equation*}
With the choice $a(\eta):=\eta_1-i\eta_2$,
\begin{equation*}
    w= \frac{a(\eta)}{\sqrt{\eta\cdot\eta}} e^{i\theta}, \qquad z=\sqrt{\eta\cdot\eta},
\end{equation*}
the Fourier coefficients of the resulting function of $\theta$ are
\begin{equation*}
    \int_{0}^{2\pi} \exp\left(\frac{\sqrt{\eta\cdot\eta}}{2}\left(\frac{a(\eta)}{\sqrt{\eta\cdot\eta}} e^{i\theta}+\frac{\sqrt{\eta\cdot\eta}}{a(\eta)} e^{-i\theta}\right)\right)\ol{e_\ell(\theta)} \d{\theta} =  \sqrt{2\pi}\left(\frac{\eta_1-i\eta_2}{\sqrt{\eta\cdot\eta}}\right)^{\ell}I_{\ell}\left(\sqrt{\eta\cdot\eta}\right).
\end{equation*}
This suffices to conclude since
\begin{equation*}
    \frac{1}{2}\left(a(\eta) e^{i\theta}+\frac{\eta\cdot\eta}{a(\eta)} e^{-i\theta}\right) = \frac{1}{2}((\eta_1-i\eta_2)e^{i\theta} + (\eta_1 + i\eta_2)e^{-i\theta})=\eta_1\cos\theta+\eta_2\sin\theta.\qedhere    
\end{equation*}
\end{proof}

\begin{proof}[Proof of \Cref{thm:Fourier d=2}]

For any $\xi\in \R^2\setminus\{0\}$ and $\kappa\in\R$, define
\begin{equation}\label{eq:Explicit Z d=2}
    \eta_\pm(\xi):=-\frac{i}{2}\xi \pm \sqrt{-\kappa+\frac{\abs{\xi}^2}{4}}\frac{\xi^\bot}{|\xi|},
\end{equation}
where $\xi^\bot$ denotes the counter-clockwise rotation of angle $\pi/2$ of $\xi$ around the origin. 
Start by noting that $\{\eta_+(\xi),\eta_-(\xi)\}$ is the only choice of vectors of $\{\eta_+,\eta_-\}\subset\C^2$ that satisfies
\begin{equation*}
    \eta_+ +\eta_-=-i\xi\qquad \eta_+\cdot\eta_+ =\eta_-\cdot\eta_-=-\kappa.
\end{equation*}
With this in mind, one can compute the Fourier transform using that the plane  wave $e^{-i\xi\cdot x}$ can be factorized as
\[
e^{-i\xi\cdot x}=g_{\eta_+}(x)g_{\eta_-}(x).
\]
\begin{align*}
    \F f(\xi)&=\int_0^1 r\hp{\cc{g_{r\eta_-}}}{f(r\,\cdot)g_{r\eta_+}}_{L^2(\IS^{1})}\d{r}\\
    &=\int_0^1 r\left(\sum_{\ell,m\in\Z}\hp{e_{\ell}}{g_{r\eta_-}}_{L^2(\IS^1)}\hp{e_m}{g_{r\eta_+}}_{L^2(\IS^1)}\hp{\cc{e_{\ell}}}{f(r\,\cdot)e_m}_{L^2(\IS^{1})}\right)\d{r}\\
    &=\int_0^1 r\left(\sum_{\ell,m\in\Z}\hp{e_{-\ell}}{g_{r\eta_-}}_{L^2(\IS^1)}\hp{e_m}{g_{r\eta_+}}_{L^2(\IS^1)}\hp{e_{\ell}}{f(r\,\cdot)e_m}_{L^2(\IS^{1})}\right)\d{r}.
\end{align*}
Using \Cref{prop:Proyection d=2}, that 
\[
I_\ell(\sqrt{-\ka}\,r)=(\ka/|\ka|)^{\ell}i^{\ell}J_\ell(\sqrt{\ka}\,r),\quad J_{-\ell}=(-1)^{\ell}J_\ell,\quad \forall\ell\in\IZ,
\]
and the fact that we are taking the principal branch of the square root we obtain, writing again $a(\eta):=\eta_1-i\eta_2$ for $\eta=(\eta_1,\eta_2)\in\IC^2$, that the Fourier transform $\F f(\xi)$ equals
\begin{align*}
    2\pi\int_0^1  \left(\sum_{\ell,m\in\Z} (-1)^{\ell}\left(\frac{a(\eta_-)}{\sqrt{\kappa}}\right)^{-\ell}\left(\frac{a(\eta_+)}{\sqrt{\kappa}}\right)^{m}J_{\ell}\left(\sqrt{\kappa}\,r\right)J_{m}\left(\sqrt{\kappa}\,r\right)\hp{e_{\ell}}{f(r\,\cdot)e_m}_{L^2(\IS^{1})}\right)r\d{r}.
\end{align*}
This can also be written, after exchanging series and integral, as:
\begin{equation}\label{e:ftalmost}
    \F f(\xi)=2\pi \sum_{\ell,m\in\Z} (-1)^{\ell}\left(\frac{a(\eta_-)}{\sqrt{\kappa}}\right)^{-\ell}\left(\frac{a(\eta_+)}{\sqrt{\kappa}}\right)^{m}J_{\ell}\left(\sqrt{\kappa}\right)J_{m}\left(\sqrt{\kappa}\right)\mu_{\ell,m}^\ka[f].
\end{equation}
To check that this series converges absolutely, note that
\begin{align}\label{eq:Bound d=2}
  \sum_{\ell,m\in\Z} \abs{\frac{a(\eta_-)}{\sqrt{\kappa}}}^{\ell}\abs{\frac{a(\eta_+)}{\sqrt{\kappa}}}^{m}&\int_{\ID} \abs{f(x)J_{\ell}(\sqrt{\kappa}\abs{x})J_{m}(\sqrt{\kappa}\abs{x})}\d{x}\nonumber\\ &\leq\norm{f}_{L^1(\ID)}\left(2\sum_{\ell=0}^\infty M^{\ell}\norm{J_\ell(\sqrt{\kappa}\,\cdot)}_{L^\infty([0,1])}\right)^2,    
\end{align}
with 
\[M(\xi,\ka)=\max\left\{\abs{\frac{a(\eta_+)}{\sqrt{\kappa}}},\abs{\frac{a(\eta_-)}{\sqrt{\kappa}}},\abs{\frac{a(\eta_+)}{\sqrt{\kappa}}}^{-1},\abs{\frac{a(\eta_-)}{\sqrt{\kappa}}}^{-1}\right\}.\]
It follows from the power series expansion of $J_{\ell}$ that for $\ell$ large enough (depending on $\kappa>0$ or unconditional if $\ka<0$) we have 
\begin{equation*}
    \norm{J_{\ell}(\sqrt{\kappa}\,\cdot)}_{L^\infty([0,1])}=\abs{J_\ell(\sqrt{\kappa})}=\abs{I_{\ell}\left(\sqrt{-\kappa}\right)}.
\end{equation*}
The result follows from (see \cite[Section 8.1]{Watson1944})
\begin{equation*}
    \lim_{\nu\to\infty}\frac{\abs{I_{\nu}\left(z\right)}}{\sqrt{\frac{1}{2\pi \nu}}\abs{\frac{e z}{2\nu}}^\nu}=1    
\end{equation*}
which shows the series in \eqref{eq:Bound d=2} is finite. To conclude, observe that a direct computation shows that:
\begin{equation*}
    a(\eta_+)=i\frac{\zeta}{|\zeta|} (|\zeta|+ \sqrt{|\zeta|^2-\ka}),\qquad a(\eta_-)^{-1}=-i\frac{\ol{\zeta}}{|\zeta|}\frac{|\zeta|+ \sqrt{|\zeta|^2-\ka}}{\ka}.
\end{equation*}
Substituting this in \eqref{e:ftalmost} concludes the proof.
\end{proof}

\section{Numerical analysis of the Born approximation}\label{s:numerical}

In this section, we present how we numerically compute the matrix elements of $\La_\ga$ from $\gamma$ (direct problem) and how we construct $\gb_{\sigma_\kappa}$ from those matrix elements (inverse problem). Subsequently, we provide various numerical experiments that illustrate the advantages and limitations of $\gb_{\sigma_\kappa}$ as an approximation of $\ga$. The Julia code used can be found in the GitHub repository \cite{Git}.

\subsection{Direct problem method}\

We aim to compute the matrix elements $\hp{\cc{e_\ell}}{\Lambda_\ga e_m}_{H^{1/2}\times H^{-1/2}}$ from a given $\ga$. In order to do this, we first need to solve a discretized version of \eqref{e:conduct} with $f=e_m$. We embed in $\ID$ a spiderweb graph that has the positive, second kind Chebyshev nodes of order $2N_r-1$ in the radial direction and an even number $N_\theta$ of Fourier nodes in the angular direction. Furthermore, we assign to each direction its corresponding spectral differentiation matrix. Following \cite[Chapter 11]{Trefethen}, we use these nodes and differentiation matrices to transform \eqref{e:conduct} into a linear matrix equation, which we then solve. The projection over $e_\ell$ is then done with a trapezoidal integration rule. Since we use $N_\theta$ Fourier nodes, it only makes sense to compute the matrix elements for which $\abs{l},\abs{m}< N_\theta/2$.

We use spectral differentiation matrices, as they provide exponential convergence towards the actual matrix elements if $\ga$ is analytic (see \cite{Kopriva}). This ensures a high accuracy of the direct problem method, with a relatively low number of nodes $(N_r,N_\theta)$. However, the spectral differentiation matrices, and therefore our direct problem method, produce significant errors when $\ga$ is not regular enough.

Having accurate matrix elements allows us to isolate the errors made by our approximation to $\ga$ as originating primarily from the inverse problem method.

\subsection{Inverse problem method}\

From the matrix elements of $\Lambda_\gamma$ we want to construct $\gb_{\sigma_\ka}$, which we take as an approximation to $\ga$. To accomplish this, we exploit \Cref{prop:Born_d=2_uniq}: if $\gb_{\sigma_\kappa}$ exists, then it satisfies
\begin{equation}\label{e:num_moments}
    \gm_{\ell,m}^{\ka}[\gb_{\si_k}] =  \hp{\cc{e_{\ell}}}{\Lambda_{\ga} e_m}_{H^{1/2}\times H^{-1/2}}, \qquad \ell,m\in\Z,
\end{equation}
with $\gm_{\ell,m}^\ka[\cdot]$ given by \eqref{e:mom_ga_def}.
This equality allows us to numerically construct $\gb_{\sigma_\ka}$ as a solution to a least-squares problem as follows. 

Given the elements of the matrix $\hp{\cc{e_{\ell}}}{\Lambda_{\ga} e_m}_{H^{1/2}\times H^{-1/2}}$ with $\ell,m=1,\ldots,L$, we choose $I\in \IN$ and define the $L^2(\ID)$-orthonormal functions
\[f_{i,j}(r,\theta)\coloneqq \displaystyle \frac{I}{\sqrt{i-1/2}}\chi_{(\frac{i-1}{I},\frac{i}{I}]}(r)e_j(\theta),\quad i=1,\ldots I,\,j=-L,\ldots,L.\]
We assume that $\gb_{\sigma_\ka}$ can be expressed as
\begin{equation}\label{e:num_ansatz}
    \gb_{\sigma_\ka}=\sum_{i=1}^{I}\sum_{j=-L}^{L} x_{i,j}f_{i,j},
\end{equation}
for some coefficients $x_{i,j}\in\IC$. The linearity of the $\gm_{\ell,m}^\ka$-moments and the equalities \eqref{e:num_moments} and \eqref{e:num_ansatz} lead us to
\begin{equation}\label{e:num_leastsquares}
    \sum_{i=1}^{I}\sum_{j=-L}^{L} x_{i,j}\gm_{\ell,m}^\ka[f_{i,j}]=\hp{\cc{e_{\ell}}}{\Lambda_{\ga} e_m}_{H^{1/2}\times H^{-1/2}},\qquad\ell,m=1,\ldots,L.
\end{equation}
Since $\gm_{\ell,m}^\ka[f_{i,j}]$ can be computed explicitly for $\kappa=0$ and numerically otherwise, the relation \eqref{e:num_leastsquares} is a matrix linear equation for the coefficients $x_{i,j}$, which we write as $Ax=b$. This equation is ill-posed; hence we solve it via least squares with a Tikhonov regularization term that penalizes the $\ell^2$-norm of the vector $x$ (which is equal to the $L^2(\ID)$-norm of $\gb_{\sigma_\ka}$), that is, 
\[x_\lambda=(A^*A+\lambda\Id)^{-1}A^*b.\]
The regularization parameter $\lambda>0$ is chosen by observing the associated L-curve 
\[\lambda\longmapsto(\log\norm{Ax_\lambda-b}^2_{\ell^2},\log\norm{x_\lambda}^2_{\ell^2}).
\]
If the L-curve has a sharp corner, then we choose the $\lambda$ at which the curvature of the L-curve is maximal (we actually maximize a simplified version of the curvature introduced in \cite[Lemma 1.1]{Kindermann}), otherwise we choose the $\lambda$ at which the curvature vanishes.

Notice that this method allows great flexibility on the radial part of the functions $f_{i,j}$. Any family of orthonormal functions in $L^2([0,1],r\d{r})$ would be an acceptable choice. This could produce a more regular reconstruction of $\gb_{\sigma_\ka}$ following \eqref{e:num_ansatz}, at the cost of a more computationally expensive computation of the moments $\gm_{\ell,m}^\ka[f_{i,j}]$.

For the case $\kappa=0$, we can construct from $\gb_1$ a better approximation to $\ga$ by applying the iterative scheme (see \cite{BCMM23_n})
\[\ga_0=\gb_1,\qquad \gamma_{n}=\gamma_{n-1} + \gb_1 - (\gamma_{n-1})^\mathrm{B}_{1}.\]
For each iteration we first solve the direct problem of computing $\hp{\cc{e_{\ell}}}{\Lambda_{\ga_{n-1}} e_m}_{H^{1/2}\times H^{-1/2}}$, followed by the inverse problem of computing $(\gamma_{n-1})^\mathrm{B}_{1}$. Since the method we use for the direct problem requires regularity of its input and the output of our inverse problem method is of the form \eqref{e:num_ansatz}, which is clearly discontinuous, we can only apply the iterative scheme a few times before it diverges. However, we observe that one or two iterations are sufficient for $\norm{\gamma-\gamma_n}_{L^p(\ID)}$ to decrease.

\subsection{Experiments}\

\subsubsection*{Experiment 1}
The purpose of the experiment is to illustrate that when $\ga$ is a positive or negative bump on a constant background, $\gb_1$ accurately captures the location of the bump.

In \cref{Fig:Bump+} and \cref{Fig:Bump-} we use the conductivities
\[\gamma(r,\theta)=1+g(re^{i\theta},0.6\,e^{i\pi/4},0.25),\qquad \gamma(r,\theta)=1-0.5\,g(re^{i\theta},0.6\,e^{i\pi/4},0.25),
\]
where $g$ is the smooth bump function
\[
g(z,z_0,R)\coloneqq\begin{cases}
    \exp\left(\frac{-\abs{z-z_0}^2}{R^2-\abs{z-z_0}^2}\right),& \text{if } \abs{z-z_0} \leq R,\\
    0,& \text{else.}
\end{cases}
\]
We solve the direct problem with $(N_r,N_\theta)=(50,50)$ and the inverse problem with $(I,L)=(50,24)$. 
In both cases, we observe that $\gb_1$ is less than $\gamma$, but this discrepancy is significantly reduced after applying the iterative scheme a few times.

We can use the conductivity of \cref{Fig:Bump+} to demonstrate the precision of our direct problem method based on spectral differentiation. The maximum absolute error between $\hp{\cc{e_{\ell}}}{\Lambda_{\ga} e_m}_{H^{1/2}\times H^{-1/2}}$ as obtained with $(N_r,N_\theta)=(50,50)$ and $(N_r,N_\theta)=(100,100)$ is
\[
\max_{1\leq\ell,m\leq 24}\abs{\hp{\cc{e_{\ell}}}{\Lambda_{\ga} e_m}_{H^{1/2}\times H^{-1/2}}(50,50)-\hp{\cc{e_{\ell}}}{\Lambda_{\ga} e_m}_{H^{1/2}\times H^{-1/2}}(100,100)}=1.6\times10^{-4}.
\]
This error is two orders of magnitude smaller than any of the errors reported in \cref{Fig:Bump+}, from which we conclude that the inaccuracies in our approximation to $\gamma$ arise from the inverse problem method and not the direct one.

\subsubsection*{Experiment 2}
In this experiment, we show how the resolution (the ability to distinguish two nearby objects) of the Born approximation $\gb_1$ is highly dependent on the distance to the boundary.

In \cref{Fig:Resolution_Boundary} the conductivity 
\begin{align*}
\gamma(r,\theta)=1&+g(re^{i\theta},0.8\,e^{i\pi\left(\frac{3}{6}+\frac{1}{20}\right)},0.1)+g(re^{i\theta},0.8\,e^{i\pi\left(\frac{3}{6}-\frac{1}{20}\right)},0.1)\\
& +g(re^{i\theta},0.8\,e^{i\pi\left(\frac{7}{6}+\frac{1}{25}\right)},0.1)+g(re^{i\theta},0.8\,e^{i\pi\left(\frac{7}{6}-\frac{1}{25}\right)},0.1)\\
& +g(re^{i\theta},0.8\,e^{i\pi\left(\frac{11}{6}+\frac{1}{30}\right)},0.1)+g(re^{i\theta},0.8\,e^{i\pi\left(\frac{11}{6}-\frac{1}{30}\right)},0.1)
\end{align*}
consists of pairs of bumps near the boundary at different distances from each other. We observe that $\gb_1$ distinguishes the bumps that are farther apart from each other but fails with the closest pair. 

Notice that we use $(N_r,N_\theta)=(50,100)$ for the direct problem and $(I,L)=(50,49)$ for the inverse one. This is not because we need more precise matrix elements, but because we need more matrix elements to resolve the small bumps. That is, we need a larger $L$, but since $L< N_\theta/2$, we must first increase $N_\theta$ in order to increase $L$. The following table presents the errors made by different values of $L$ with $N_\theta=2(L+1)$, $N_r=I=50$.
\begin{table}[H]
\centering
\begin{tabular}[]{cccc}
\hline
$L$ \textbackslash\, $p$ & 1 & 2 & $\infty$ \\
\hline
24 & 0.26730 & 0.24074 & 0.91152 \\
29 & 0.19799 & 0.20441 & 0.89030 \\
34 & 0.12702 & 0.15496 & 0.69693 \\
39 & 0.10260 & 0.13376 & 0.67875 \\
44 & 0.09744 & 0.13329 & 0.62283 \\
49 & 0.08933 & 0.13038 & 0.61826 \\
\hline
\end{tabular}
\caption{$\norm{\gamma-\gb_1}_{L^p(\ID)}$ for the conductivity of \cref{Fig:Resolution_Boundary} using different values of $L$.}
\end{table}
Clearly, the errors decay as $L$ increases. We found that using $N_r,\,I>50$ produces negligible improvements.

In \cref{Fig:Resolution_Origin}, we repeat the same setting with the conductivity 
\begin{align*}
\gamma(r,\theta)=1&+g(re^{i\theta},0.4\,e^{i\pi\left(\frac{3}{6}+\frac{1}{10}\right)},0.1)+g(re^{i\theta},0.4\,e^{i\pi\left(\frac{3}{6}-\frac{1}{10}\right)},0.1)\\
& +g(re^{i\theta},0.4\,e^{i\pi\left(\frac{7}{6}+\frac{1}{13}\right)},0.1)+g(re^{i\theta},0.4\,e^{i\pi\left(\frac{7}{6}-\frac{1}{13}\right)},0.1)\\
& +g(re^{i\theta},0.4\,e^{i\pi\left(\frac{11}{6}+\frac{1}{15}\right)},0.1)+g(re^{i\theta},0.4\,e^{i\pi\left(\frac{11}{6}-\frac{1}{15}\right)},0.1),
\end{align*}
which has the bumps closer to the origin. In this case, $\gb_1$ does not distinguish the bumps in any of the pairs, and severely underestimates the height of the bumps.

\subsubsection*{Experiment 3}
This experiment shows that whenever $\gamma$ is a small perturbation of (a multiple of) $\sigma_\kappa$, then $\gb_{\sigma_\kappa}$ is a much better approximation to $\ga$ than $\gb_{\sigma_0}=\gb_1$.

In \cref{Fig:kappa+} we use the conductivity
$\displaystyle\gamma=\frac{\sigma_4}{J_0(2)^2}+\smiley{}$,
where $\smiley{}$ is given by
\[\smiley{}(r,\theta)=g(re^{i\theta},0.6\,e^{i\pi/4},0.25)+g(re^{i\theta},0.6\,e^{i3\pi/4},0.25)+ g(r, 0.6, 0.2) g(\theta, 3\pi/2, \pi / 3).\]
Due to the screening  effect of the very low conductivity near the boundary, $\gb_1$ fails to capture $\ga$ in the interior of the $\ID$; a task where $\gb_{\sigma_4}$ succeeds. This fact becomes even more apparent when we subtract $\frac{\sigma_4}{J_0(2)^2}$ from $\ga$, $\gb_{\sigma_0}$, and $\gb_{\sigma_4}$. For this experiment, we use $(N_r,N_\theta)=(50,50)$ for the direct problem and $(I,L)=(50,24)$ for the inverse one.

In \cref{Fig:kappa-} we repeat the same setting but now with $\ka<0$, namely, we use the conductivity $\ga=\sigma_{-9}+\smiley{}$. Again, we observe that $\gb_{\sigma_{-9}}$ performs much better than $\gb_{\sigma_0}$ in capturing the perturbation $\smiley{}$.

\subsubsection*{Experiment 4}
In this experiment, we consider the case of non-radial, discontinuous conductivities.
In \cref{Fig:Explicit1} and \cref{Fig:Explicit2} we use the functions
\[\gamma(r,\theta)=f\circ c(re^{i\theta}),
\qquad c(z)=\frac{4z-1}{z-4},\]
with $f$ being respectively
\[f(z)=\begin{cases}
3,&\abs{z}\leq\frac{1}{4},\\
2,&\frac{1}{4}<\abs{z}\leq\frac{1}{2},\\
1,&\frac{1}{2}<\abs{z},
\end{cases}\qquad f(z)=\begin{cases}
3,&\abs{z}\leq\frac{1}{4},\\
1/2,&\frac{1}{4}<\abs{z}\leq\frac{1}{2},\\
1,&\frac{1}{2}<\abs{z}.
\end{cases}\]
These $\ga$ are the compositions of the conformal transformation $c$ of $\ID$ with a radial piecewise constant function.

There are two points to note in these figures. First, due to the use of Tikhonov regularization, $\gb_1$ appears continuous even though $\ga$ is discontinuous. In \cref{Fig:Explicit2} the low conductivity annular region shields the high conductivity disk, preventing $\gb_1$ from correctly detecting its height, even after the application of the iterative scheme.

We use $(I,L)=(50,24)$ for the inverse problem method. We cannot use the direct problem method to compute the matrix elements of the DtN map, as it requires regular conductivities. Instead, the matrix elements $\hp{\cc{e_{\ell}}}{\Lambda_{\ga} e_m}_{H^{1/2}\times H^{-1/2}}$ can be analytically computed as a convergent series, using the fact that both conductivities are compositions of a conformal transformation with radial piecewise constant functions.

\subsubsection*{Experiment 5}
Our last experiment explores to what extent our inverse problem method is resistant to noise.

In \cref{Fig:Noise}, we use the conductivity
\[\ga(r,\theta)=1-0.5\,g(re^{i\theta},0.6\,e^{i3\pi/4},0.25)+ g(r, 0.6, 0.2) g(\theta, 3\pi/2, \pi / 3),\]
and compute its corresponding matrix elements using our direct problem method with $(N_r,N_\theta)=(50,50)$. Then, we add a complex Gaussian noise to each matrix element (respecting symmetry), namely,
\[
    \hp{\cc{e_{\ell}}}{\Lambda_{\ga} e_m}_{H^{1/2}\times H^{-1/2}}+\epsilon(\mathcal{N}(0,1) + i \mathcal{N}(0,1)).
\]
We now apply the inverse problem method with parameters $(I,L)=(50,24)$ to compute $\gb_1$ using the matrix elements perturbed by the noise.

From this we observe that $\gb_1$ is resistant to noise up to $\epsilon=10^{-3}$, but it starts to fail at $\epsilon=10^{-2}$. This is to be expected since most of the matrix elements $\hp{\cc{e_{\ell}}}{\Lambda_{\ga} e_m}_{H^{1/2}\times H^{-1/2}}$ are of order $10^{-2}$ for this particular $\ga$.

\newpage
\begin{figure}[t]
\centering
\subfloat[$\gamma$]{\includegraphics[width=0.3\textwidth, trim={1.7cm 0.8cm 1.7cm 0.2cm}, clip]{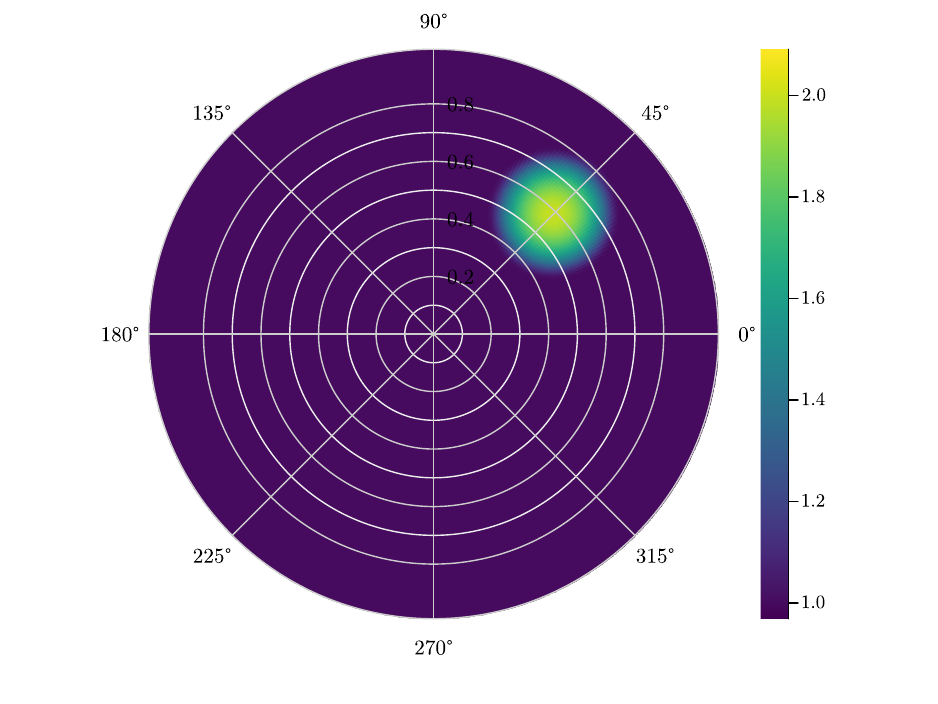}}\hspace{0.04\textwidth}
\subfloat[$\gamma_1^\mathrm{B}=\gamma_0$]{\includegraphics[width=0.3\textwidth, trim={1.7cm 0.8cm 1.7cm 0.2cm}, clip]{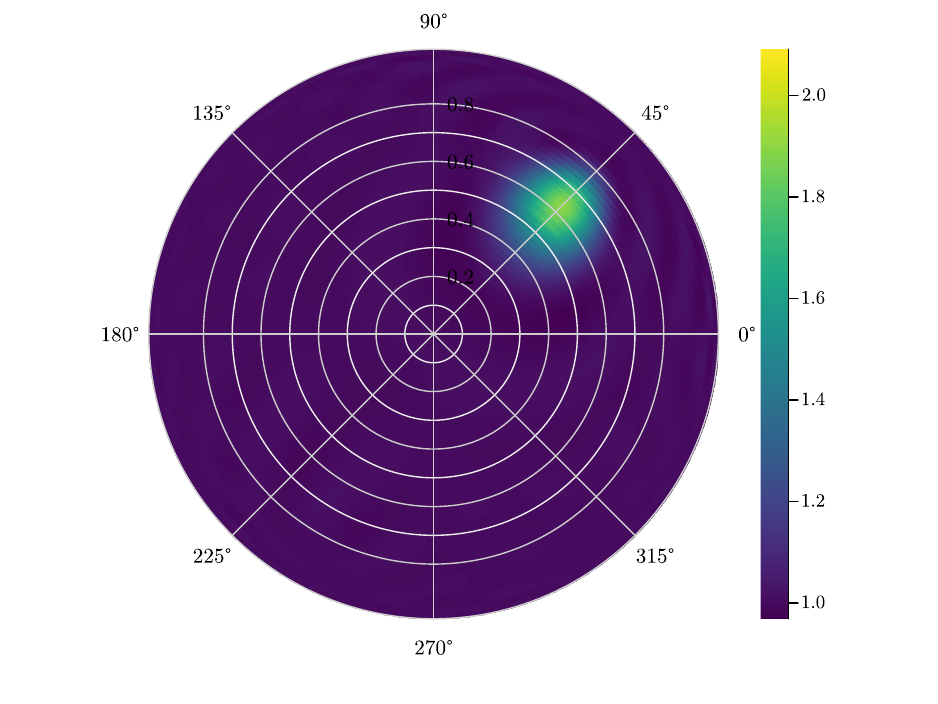}}\hspace{0.04\textwidth}
\subfloat[$\gamma_N$]{\includegraphics[width=0.3\textwidth, trim={1.7cm 0.8cm 1.7cm 0.2cm}, clip]{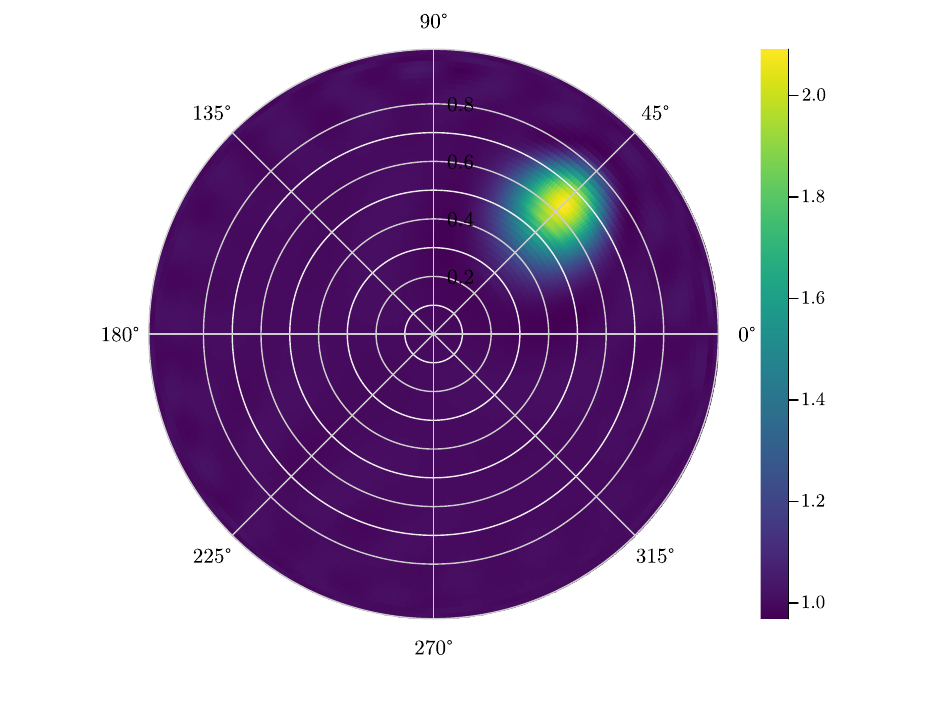}}
\vspace{-0.8\baselineskip}
\subfloat[Angular cross section]{\includegraphics[width=0.45\textwidth]{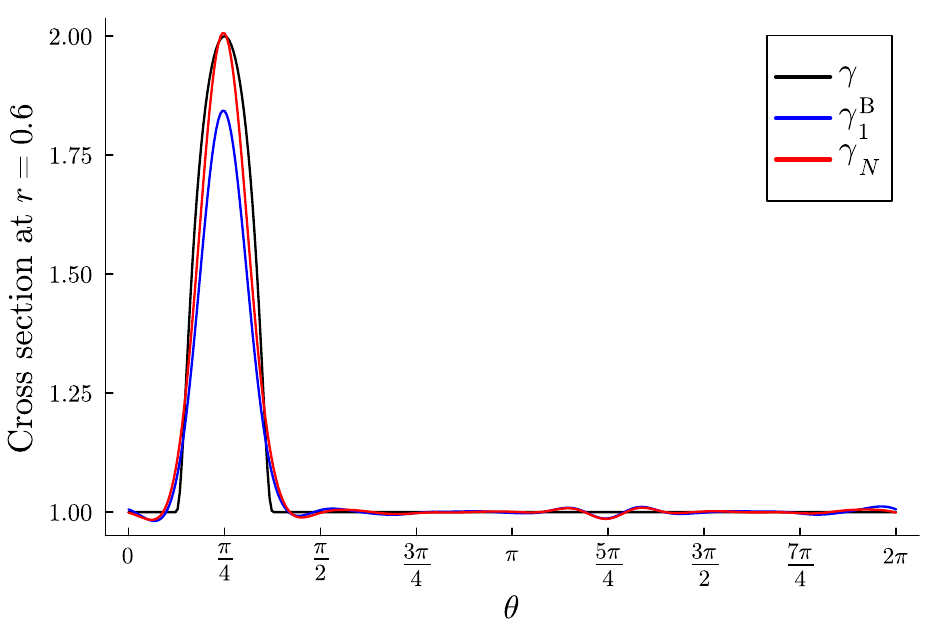}}\hspace{0.05\textwidth}
\subfloat[$\norm{\gamma-\gamma_n}_{L^p(\ID)}$]{
\begin{tabular}[b]{cccc}
\hline
$n$ \textbackslash\, $p$ & 1 & 2 & $\infty$ \\
\hline
0 & 0.03895 & 0.06904 & 0.33387 \\
1 & 0.03432 & 0.05210 & 0.26982 \\
2 & 0.03670 & 0.05064 & 0.25109 \\
\hline
\vspace{1cm}
\end{tabular}
}
\caption{Localization of positive circular bump. Iterative scheme run $N=2$ times. $(N_r,N_\theta)=(50,50)$, $(I,L)=(50,24)$.}
\label{Fig:Bump+}
\end{figure}

\begin{figure}[H]
\centering
\subfloat[$\gamma$]{\includegraphics[width=0.3\textwidth, trim={1.7cm 0.8cm 1.7cm 0.2cm}, clip]{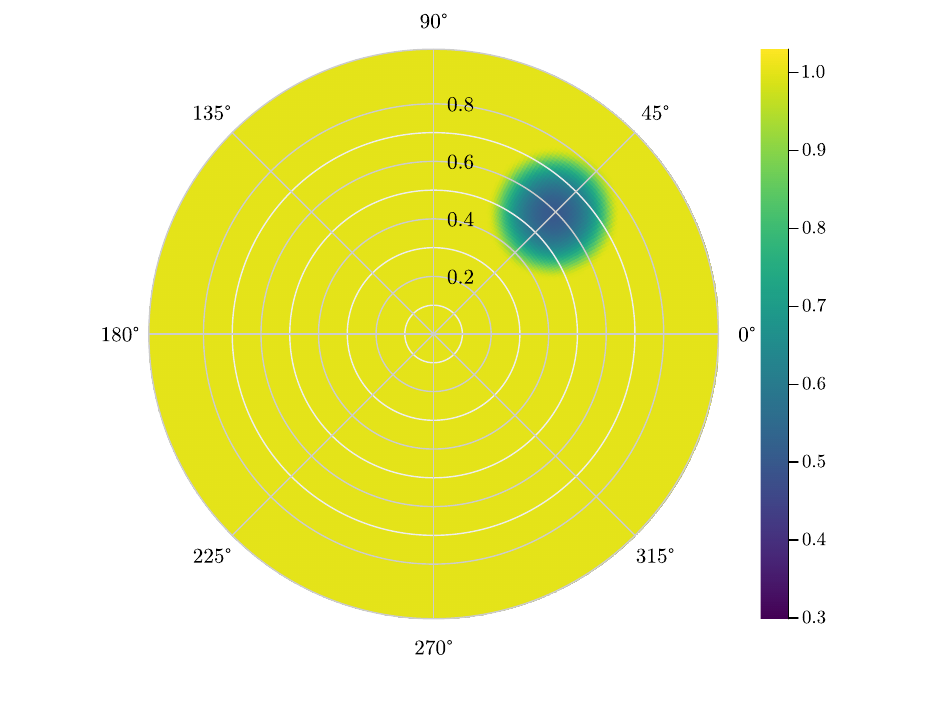}}\hspace{0.04\textwidth}
\subfloat[$\gamma_1^\mathrm{B}=\gamma_0$]{\includegraphics[width=0.3\textwidth, trim={1.7cm 0.8cm 1.7cm 0.2cm}, clip]{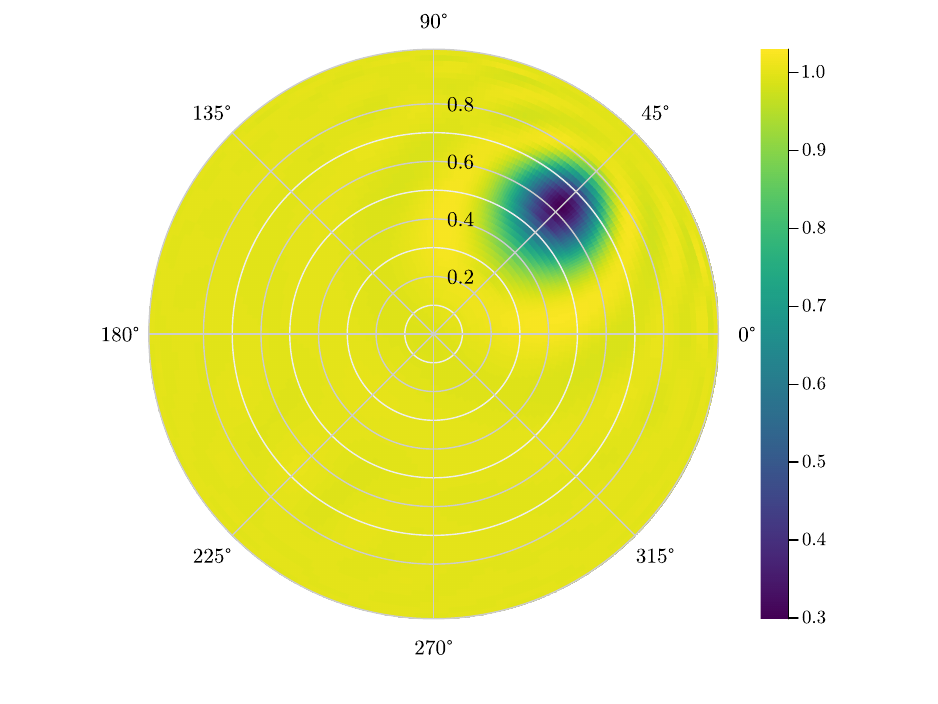}}\hspace{0.04\textwidth}
\subfloat[$\gamma_N$]{\includegraphics[width=0.3\textwidth, trim={1.7cm 0.8cm 1.7cm 0.2cm}, clip]{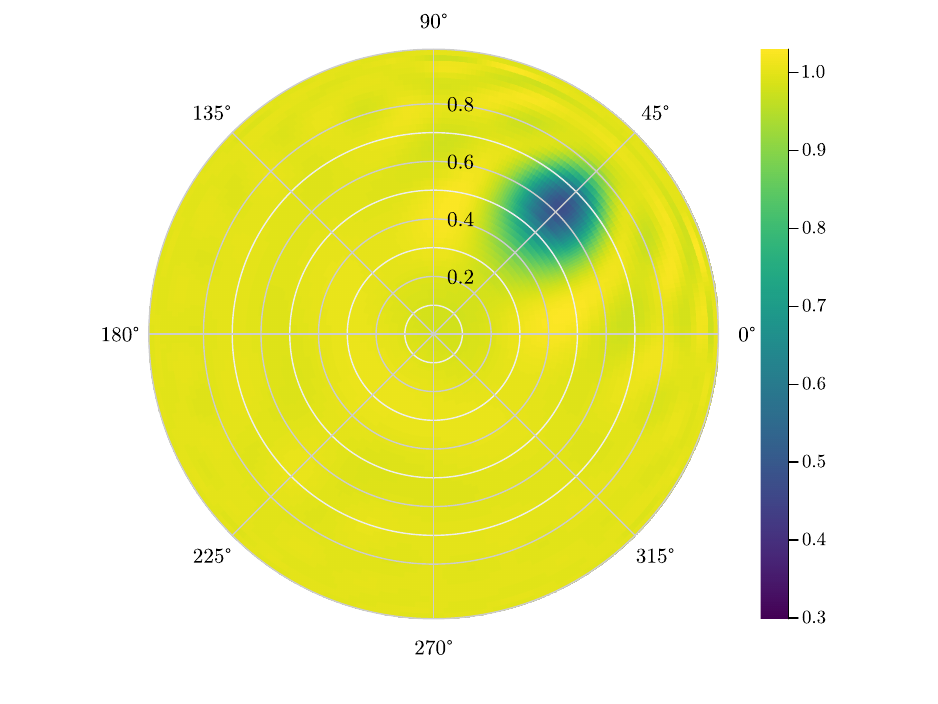}}
\vspace{-0.8\baselineskip}
\subfloat[Angular cross section]{\includegraphics[width=0.45\textwidth]{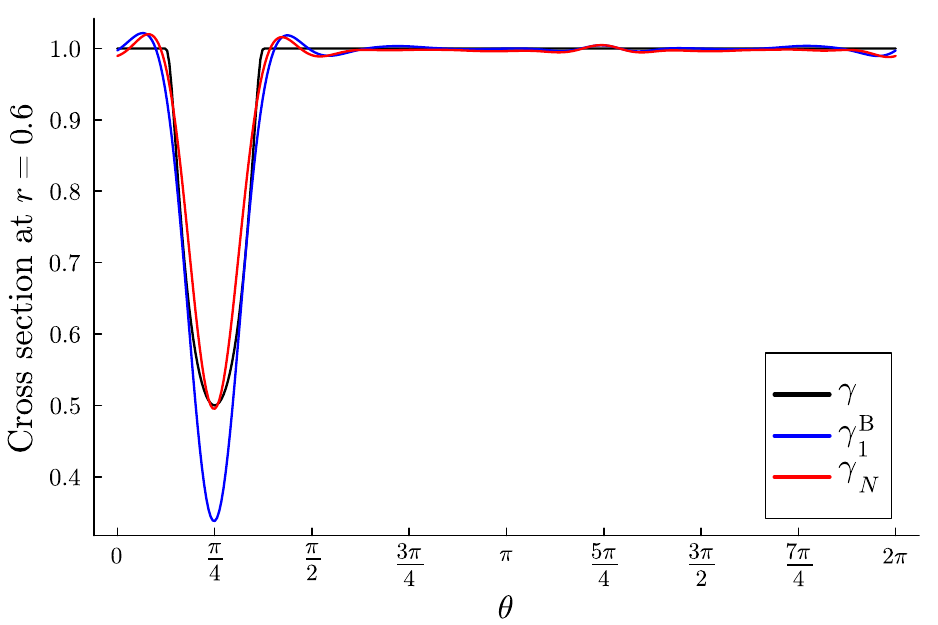}}\hspace{0.05\textwidth}
\subfloat[$\norm{\gamma-\gamma_n}_{L^p(\ID)}$]{
\begin{tabular}[b]{cccc}
\hline
$n$ \textbackslash\, $p$ & 1 & 2 & $\infty$ \\
\hline
0 & 0.02144 & 0.03374 & 0.22140 \\
1 & 0.02364 & 0.02665 & 0.11089 \\
\hline
\vspace{1cm}
\end{tabular}
}
\caption{Localization of negative circular bump. Iterative scheme run $N=1$ times. $(N_r,N_\theta)=(50,50)$, $(I,L)=(50,24)$.}
\label{Fig:Bump-}
\end{figure}

\newpage
\begin{figure}[t]
\centering
\subfloat[$\gamma$]{\includegraphics[width=0.3\textwidth, trim={1.7cm 0.8cm 1.7cm 0.2cm}, clip]{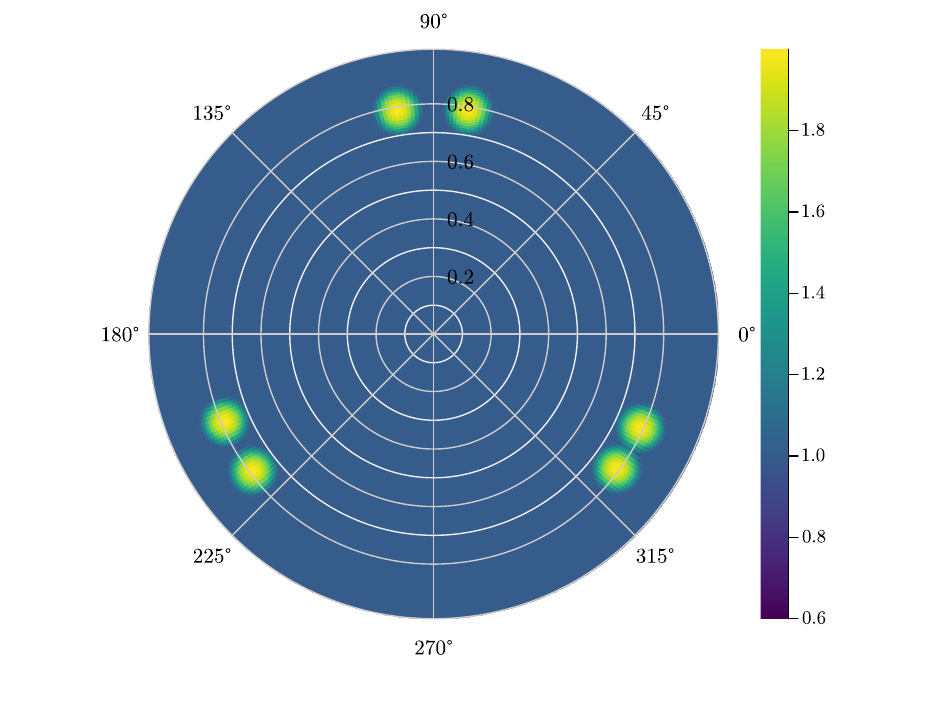}}\hspace{0.04\textwidth}
\subfloat[$\gamma_1^\mathrm{B}=\gamma_0$]{\includegraphics[width=0.3\textwidth, trim={1.7cm 0.8cm 1.7cm 0.2cm}, clip]{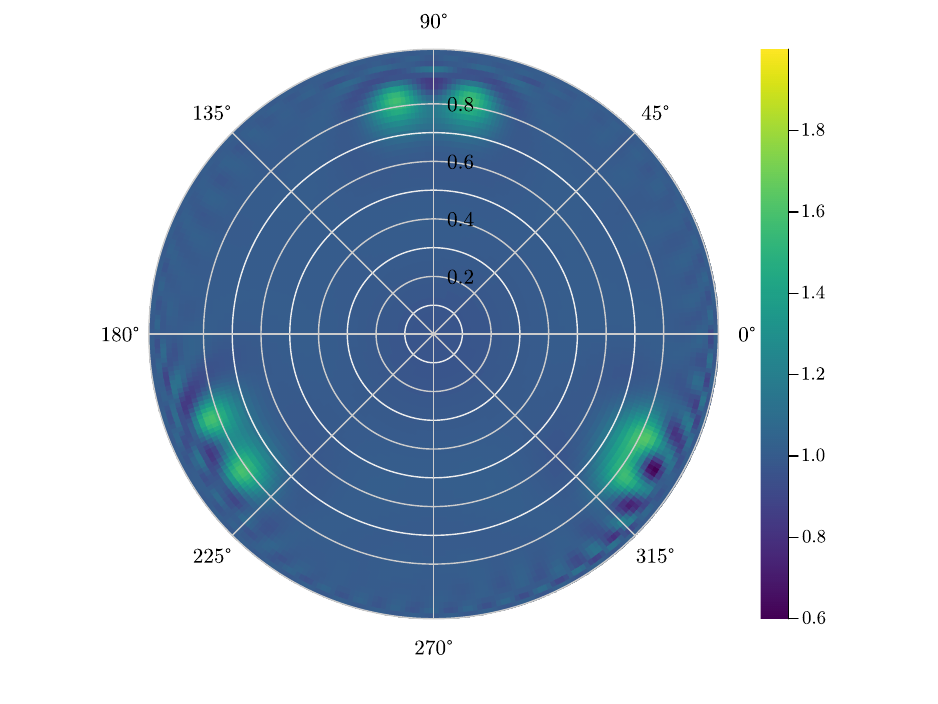}}\hspace{0.04\textwidth}
\subfloat[$\gamma_N$]{\includegraphics[width=0.3\textwidth, trim={1.7cm 0.8cm 1.7cm 0.2cm}, clip]{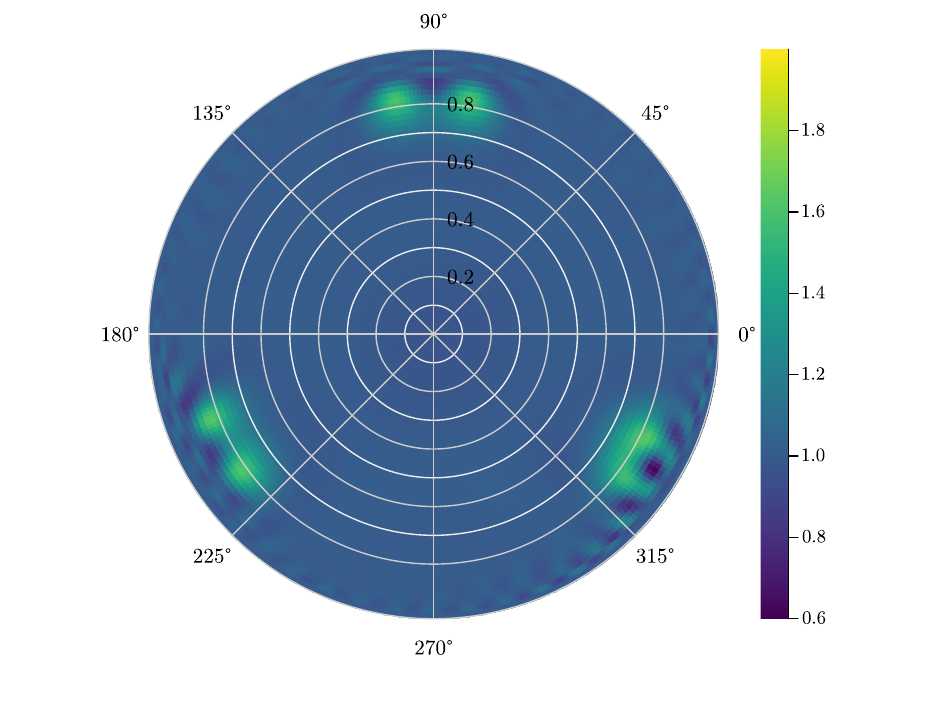}}
\vspace{-0.8\baselineskip}
\subfloat[Angular cross section]{\includegraphics[width=0.45\textwidth]{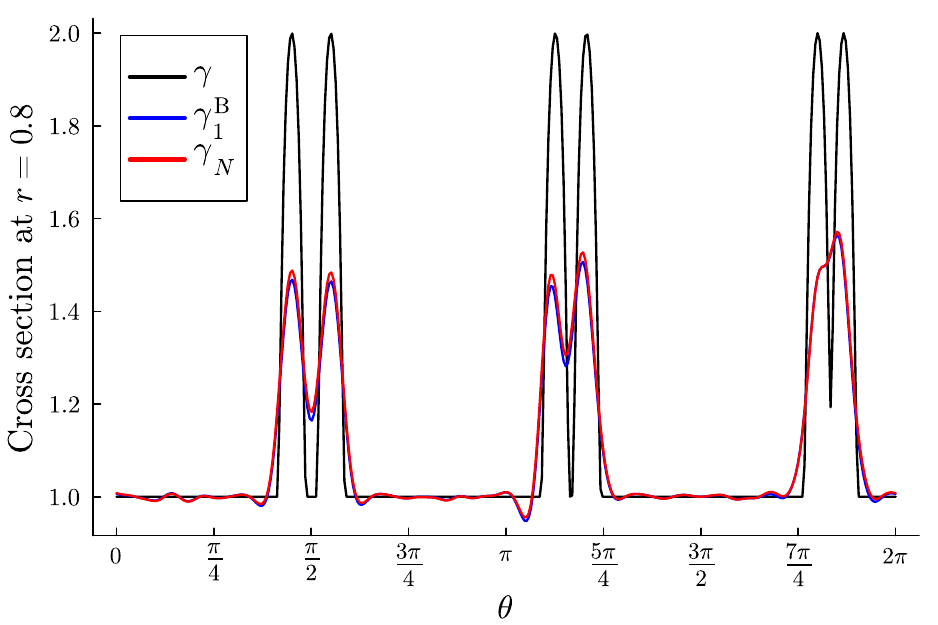}}\hspace{0.05\textwidth}
\subfloat[$\norm{\gamma-\gamma_n}_{L^p(\ID)}$]{
\begin{tabular}[b]{cccc}
\hline
$n$ \textbackslash\, $p$ & 1 & 2 & $\infty$ \\
\hline
0 & 0.08933 & 0.13038 & 0.61826 \\
1 & 0.08948 & 0.13020 & 0.59762 \\
\hline
\vspace{1cm}
\end{tabular}
}
\caption{Resolution near the boundary. Iterative scheme run $N=1$ times. $(N_r,N_\theta)=(50,100)$, $(I,L)=(50,49)$.}
\label{Fig:Resolution_Boundary}
\end{figure}

\begin{figure}[H]
\centering
\subfloat[$\gamma$]{\includegraphics[width=0.3\textwidth, trim={1.7cm 0.8cm 1.7cm 0.2cm}, clip]{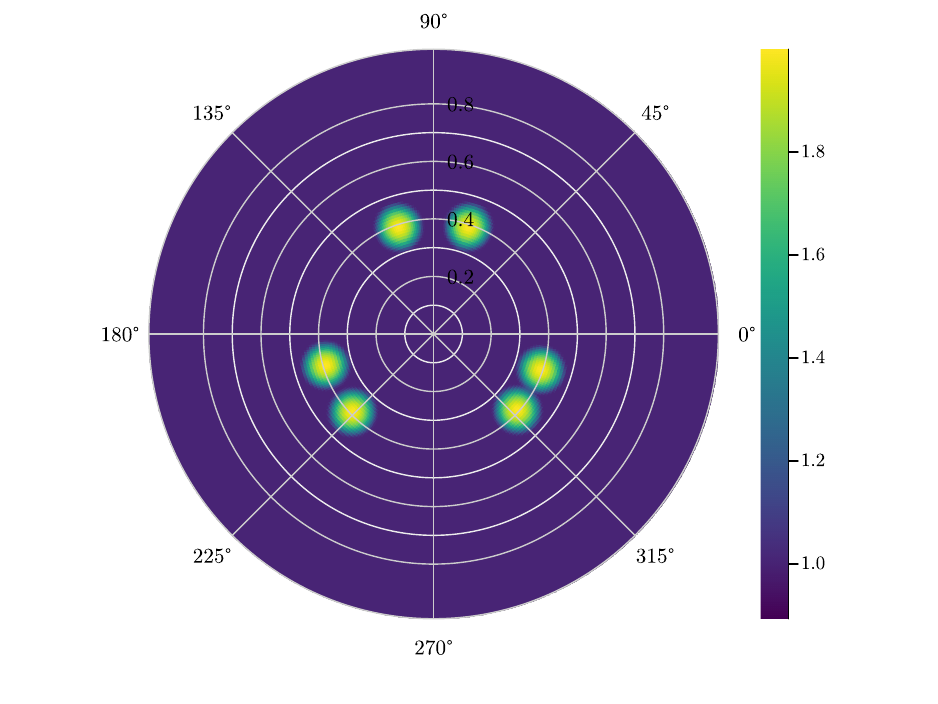}}\hspace{0.04\textwidth}
\subfloat[$\gamma_1^\mathrm{B}=\gamma_0$]{\includegraphics[width=0.3\textwidth, trim={1.7cm 0.8cm 1.7cm 0.2cm}, clip]{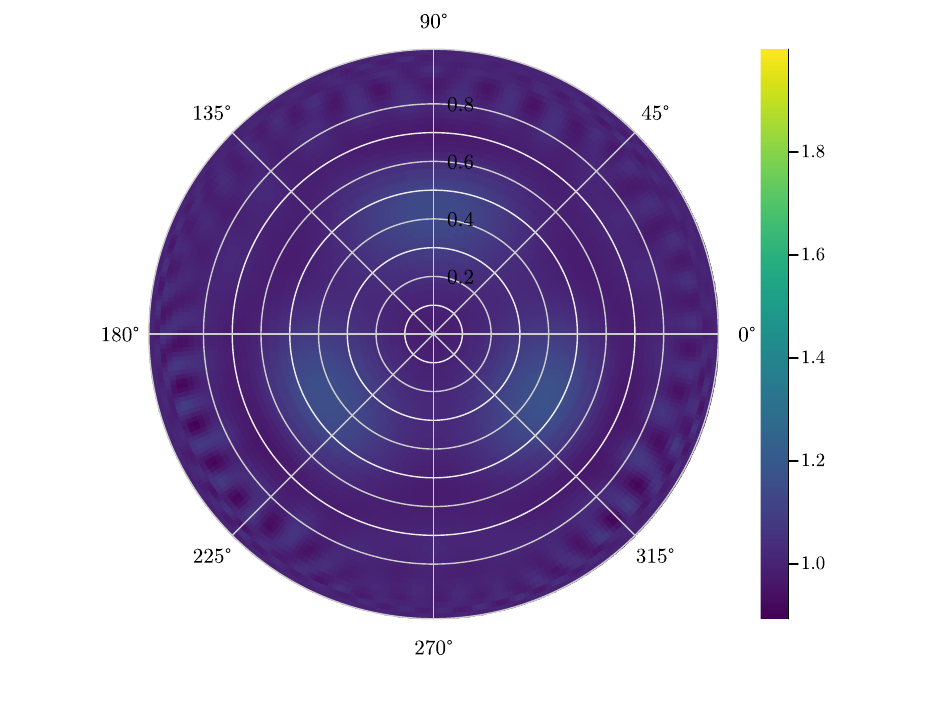}}\hspace{0.04\textwidth}
\subfloat[$\gamma_N$]{\includegraphics[width=0.3\textwidth, trim={1.7cm 0.8cm 1.7cm 0.2cm}, clip]{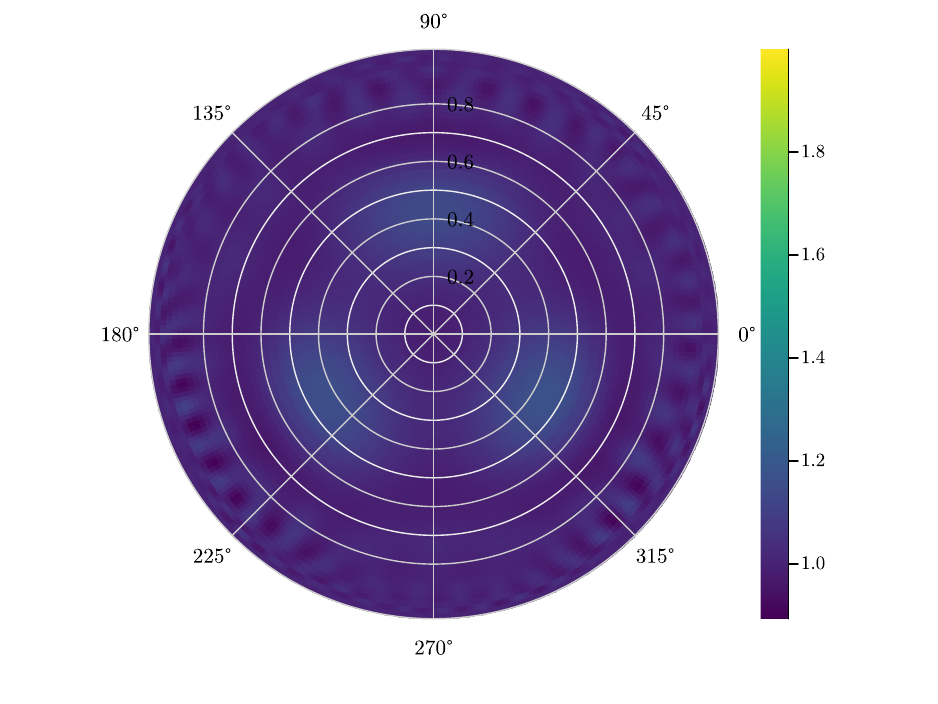}}
\vspace{-0.8\baselineskip}
\subfloat[Angular cross section]{\includegraphics[width=0.45\textwidth]{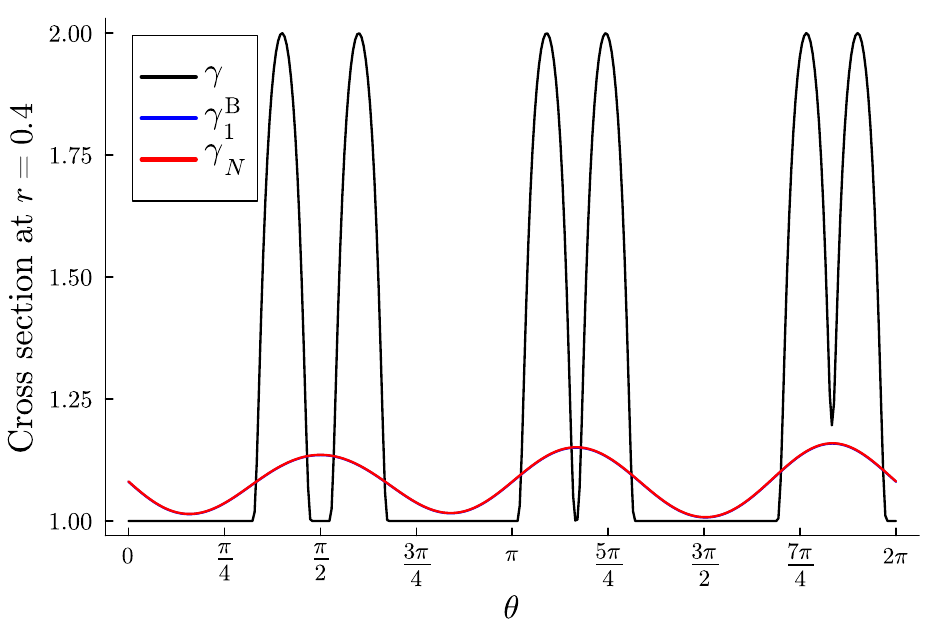}}\hspace{0.05\textwidth}
\subfloat[$\norm{\gamma-\gamma_n}_{L^p(\ID)}$]{
\begin{tabular}[b]{cccc}
\hline
$n$ \textbackslash\, $p$ & 1 & 2 & $\infty$ \\
\hline
0 & 0.13528 & 0.19910 & 0.88625 \\
1 & 0.13578 & 0.19917 & 0.88575 \\
\hline
\vspace{1cm}
\end{tabular}
}
\caption{Resolution near the origin. Iterative scheme run $N=1$ times. $(N_r,N_\theta)=(50,100)$, $(I,L)=(50,49)$.}
\label{Fig:Resolution_Origin}
\end{figure}

\newpage
\begin{figure}[t]
\centering
\subfloat[$\gamma$]{\includegraphics[width=0.3\textwidth, trim={1.7cm 0.8cm 1.7cm 0.2cm}, clip]{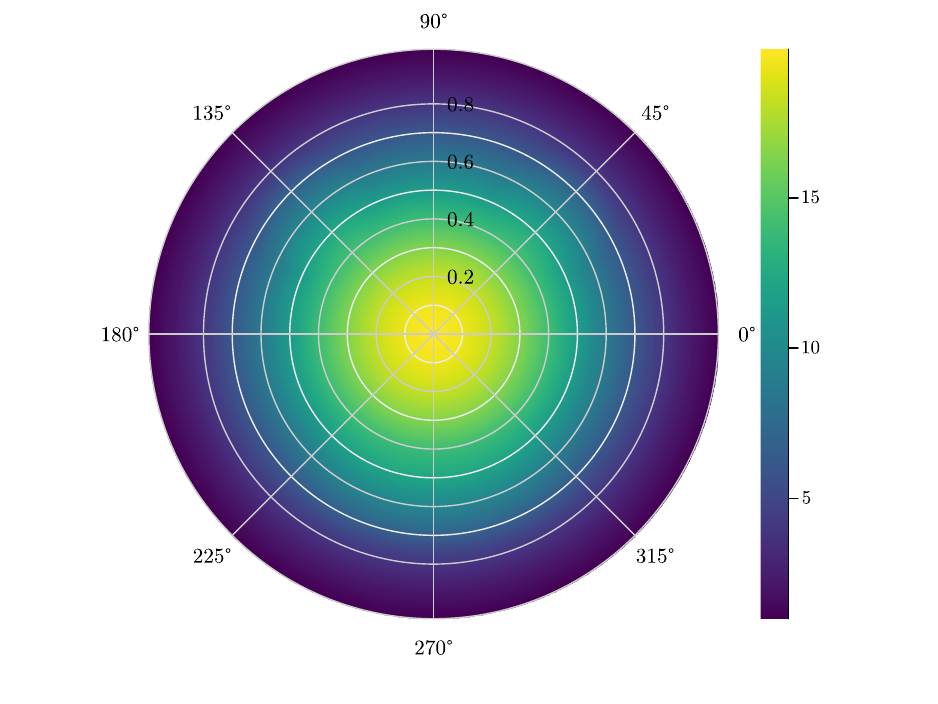}}\hspace{0.04\textwidth}
\subfloat[$\gamma_{\sigma_0}^\mathrm{B}$]{\includegraphics[width=0.3\textwidth, trim={1.7cm 0.8cm 1.7cm 0.2cm}, clip]{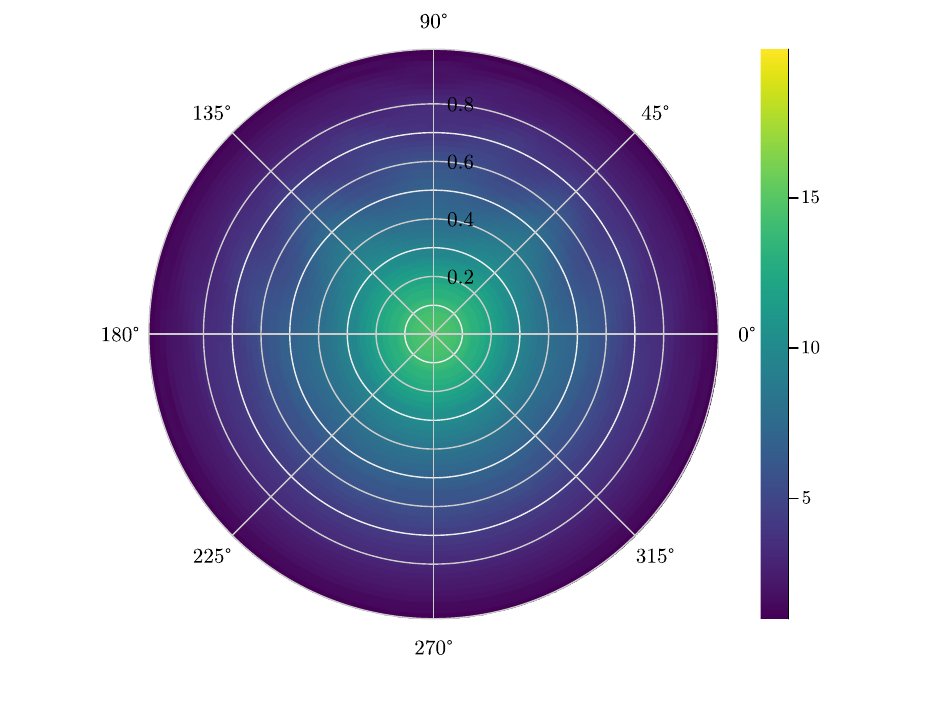}}\hspace{0.04\textwidth}
\subfloat[$\gamma_{\sigma_{4}}^\mathrm{B}$]{\includegraphics[width=0.3\textwidth, trim={1.7cm 0.8cm 1.7cm 0.2cm}, clip]{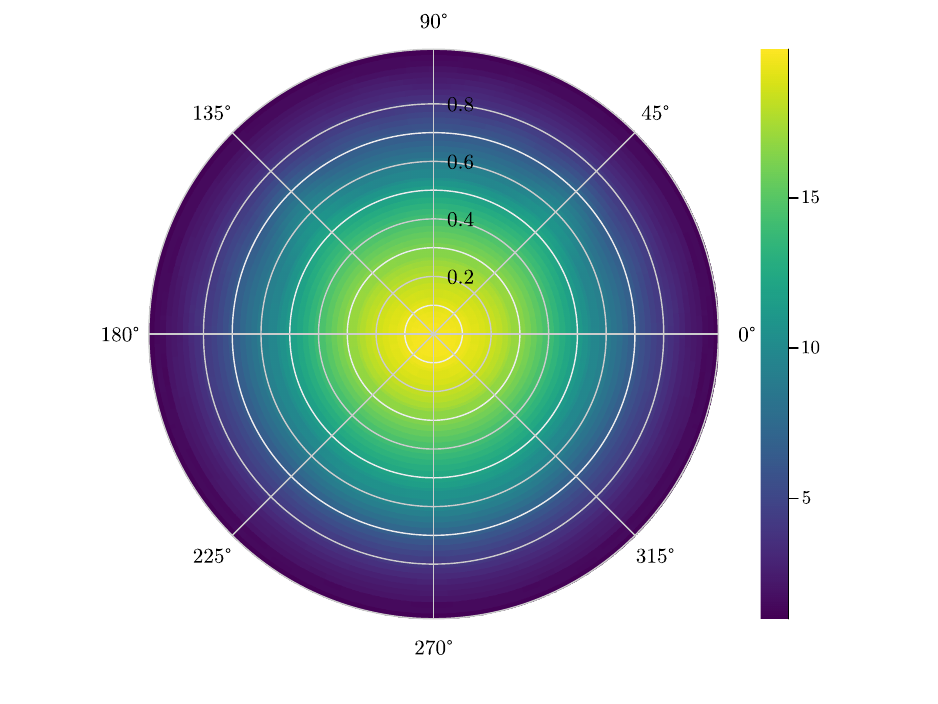}}
\vspace{-0.5\baselineskip}
\subfloat[$\gamma-\frac{\sigma_4}{J_0(2)^2}$]{\includegraphics[width=0.3\textwidth, trim={1.7cm 0.8cm 1.7cm 0.2cm}, clip]{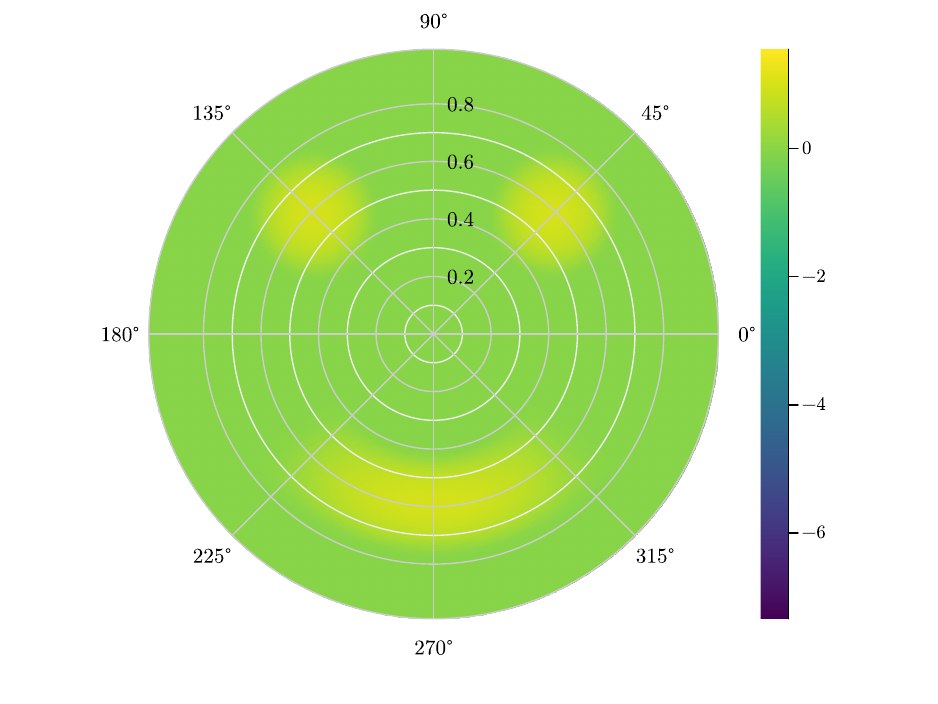}}\hspace{0.04\textwidth}
\subfloat[$\gamma_{\sigma_0}^\mathrm{B}-\frac{\sigma_4}{J_0(2)^2}$]{\includegraphics[width=0.3\textwidth, trim={1.7cm 0.8cm 1.7cm 0.2cm}, clip]{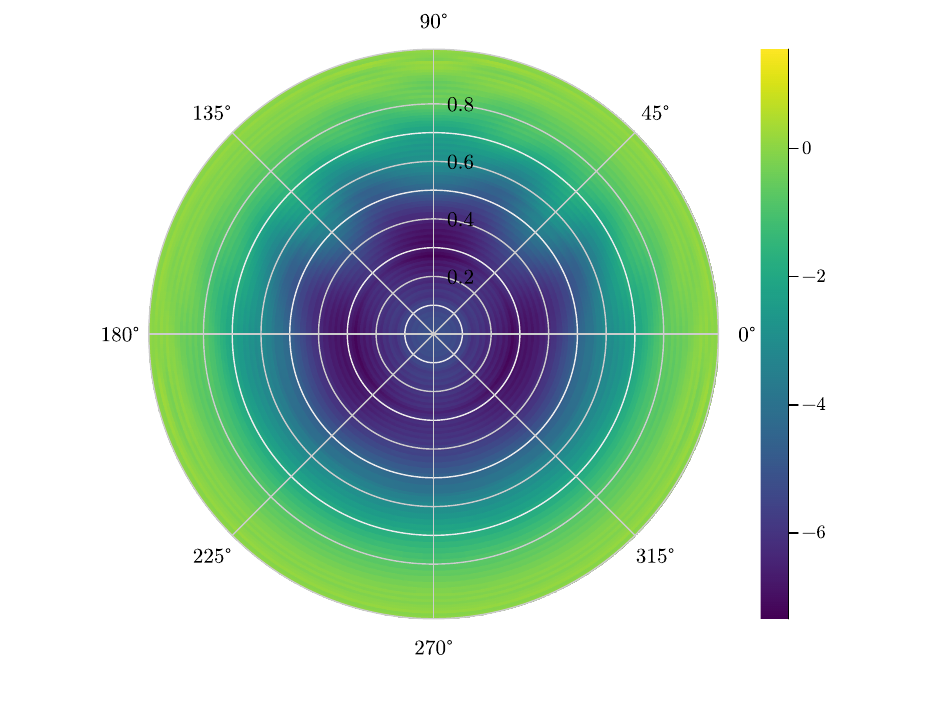}}\hspace{0.04\textwidth}
\subfloat[$\gamma_{\sigma_4}^\mathrm{B}-\frac{\sigma_4}{J_0(2)^2}$]{\includegraphics[width=0.3\textwidth, trim={1.7cm 0.8cm 1.7cm 0.2cm}, clip]{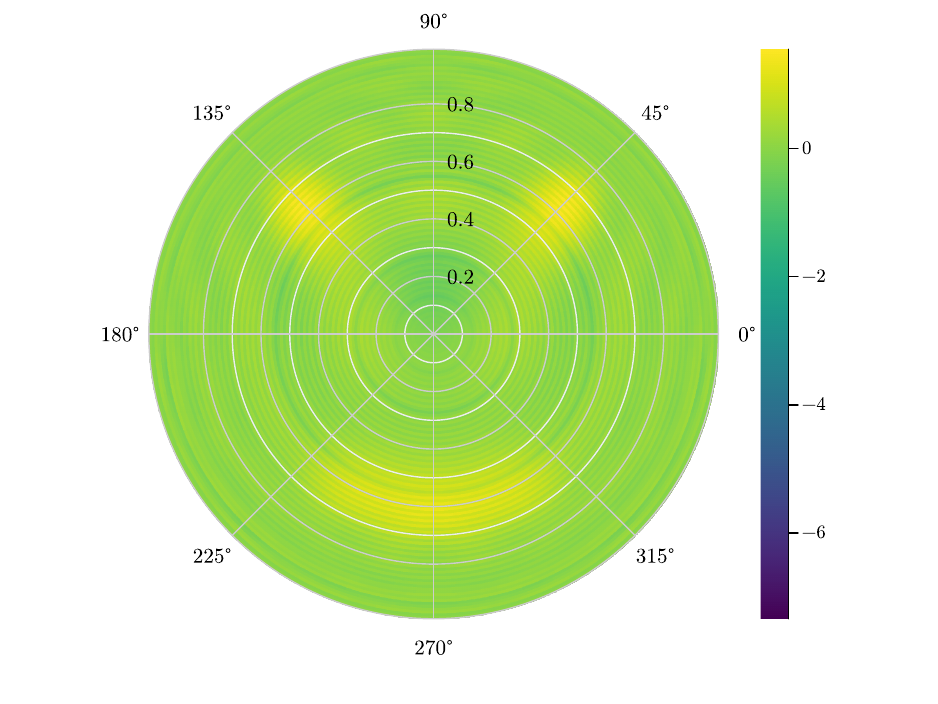}}
\vspace{-0.5\baselineskip}
\subfloat[Angular cross section]{\includegraphics[width=0.45\textwidth]{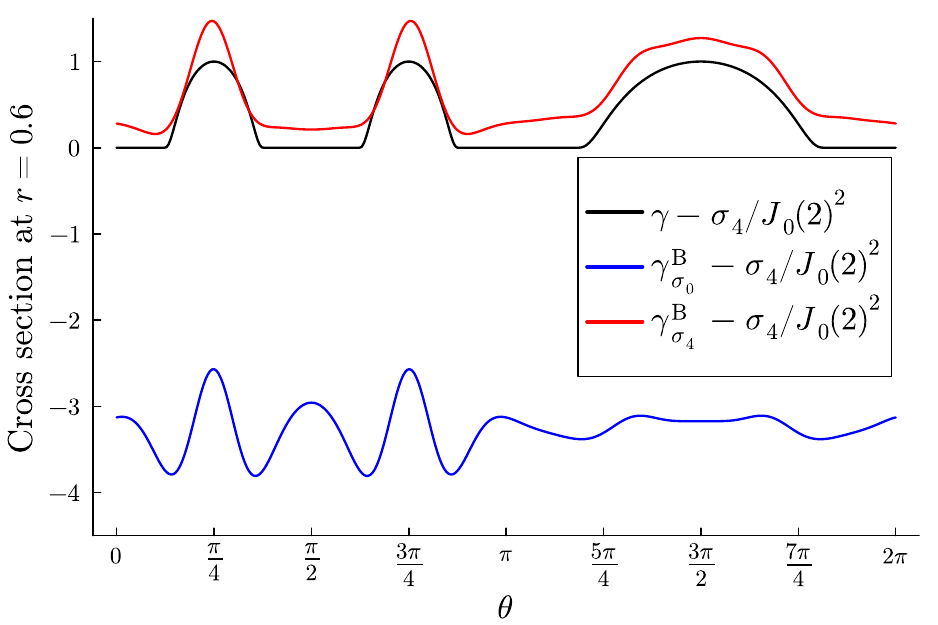}}\hspace{0.05\textwidth}
\subfloat[$\norm{\gamma-\gamma_{\sigma_\kappa}^\mathrm{B}}_{L^p(\ID)}$]{
\begin{tabular}[b]{cccc}
\hline
$\kappa$ \textbackslash\, $p$ & 1 & 2 & $\infty$ \\
\hline
0 & 8.56242 & 6.37936 & 7.34184 \\
4 & 0.48048 & 0.33540 & 0.59925 \\
\hline
\vspace{1cm}
\end{tabular}
}
\caption{Linearization at $\kappa=4$. $(N_r,N_\theta)=(50,50)$, $(I,L)=(50,24)$}
\label{Fig:kappa+}
\end{figure}

\clearpage
\begin{figure}[t]
\centering
\subfloat[$\gamma$]{\includegraphics[width=0.3\textwidth, trim={1.7cm 0.8cm 1.7cm 0.2cm}, clip]{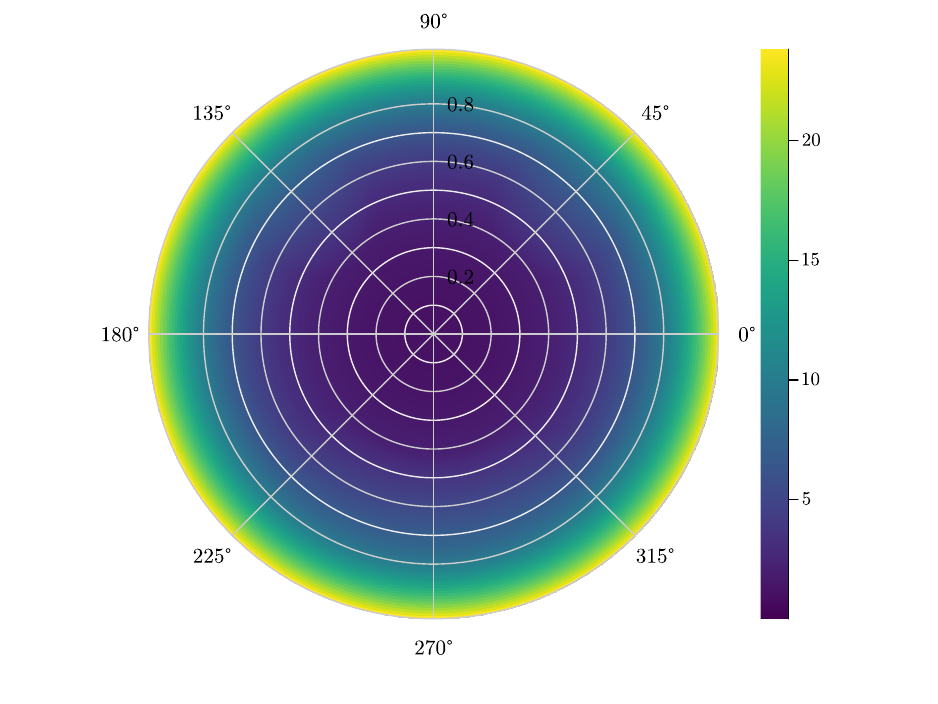}}\hspace{0.04\textwidth}
\subfloat[$\gamma_{\sigma_0}^\mathrm{B}$]{\includegraphics[width=0.3\textwidth, trim={1.7cm 0.8cm 1.7cm 0.2cm}, clip]{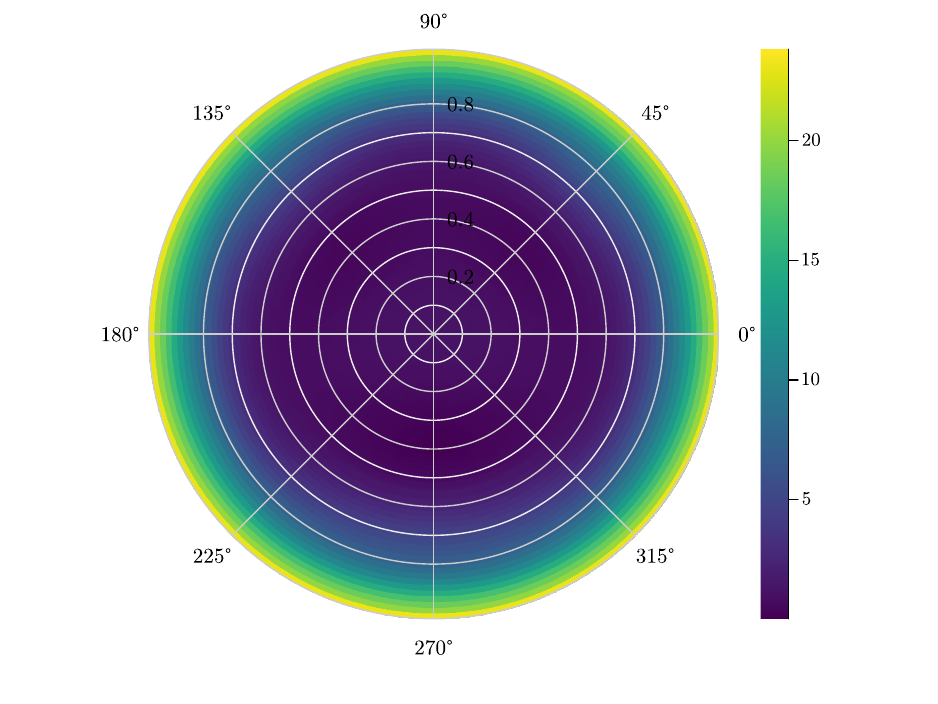}}\hspace{0.04\textwidth}
\subfloat[$\gamma_{\sigma_{-9}}^\mathrm{B}$]{\includegraphics[width=0.3\textwidth, trim={1.7cm 0.8cm 1.7cm 0.2cm}, clip]{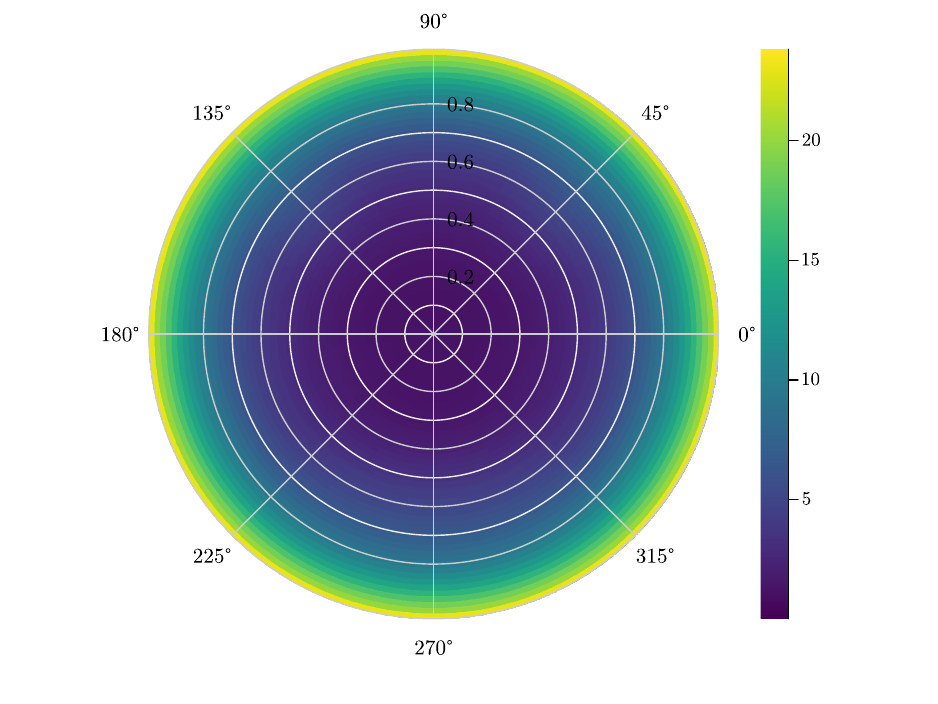}}
\vspace{-0.5\baselineskip}
\subfloat[$\gamma-\sigma_{-9}$]{\includegraphics[width=0.3\textwidth, trim={1.7cm 0.8cm 1.7cm 0.2cm}, clip]{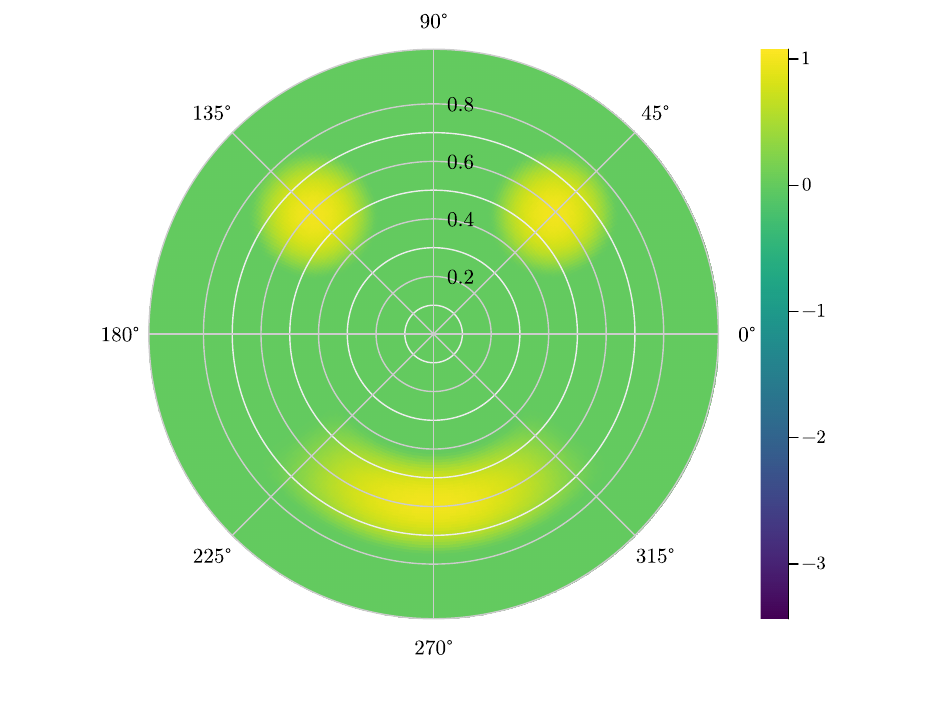}}\hspace{0.04\textwidth}
\subfloat[$\gamma_{\sigma_0}^\mathrm{B}-\sigma_{-9}$]{\includegraphics[width=0.3\textwidth, trim={1.7cm 0.8cm 1.7cm 0.2cm}, clip]{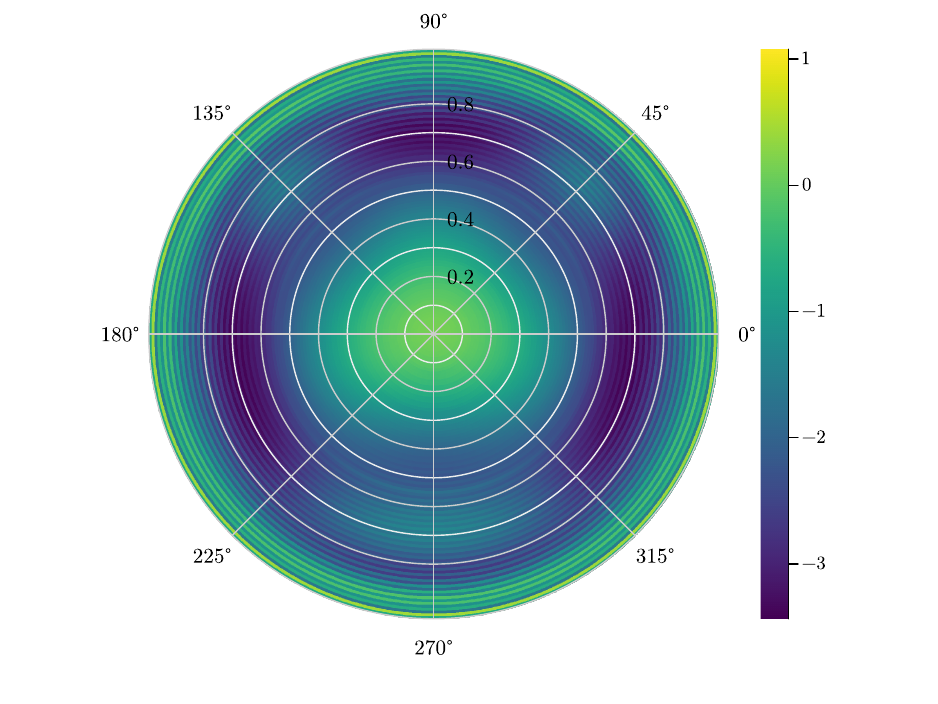}}\hspace{0.04\textwidth}
\subfloat[$\gamma_{\sigma_{-9}}^\mathrm{B}-\sigma_{-9}$]{\includegraphics[width=0.3\textwidth, trim={1.7cm 0.8cm 1.7cm 0.2cm}, clip]{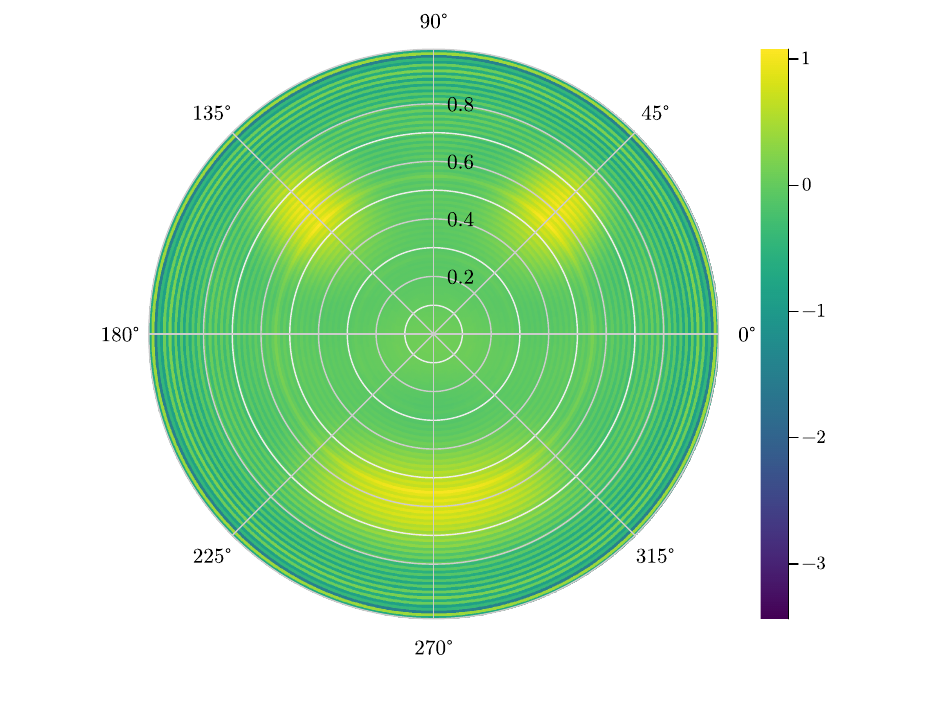}}
\vspace{-0.5\baselineskip}
\subfloat[Angular cross section]{\includegraphics[width=0.45\textwidth]{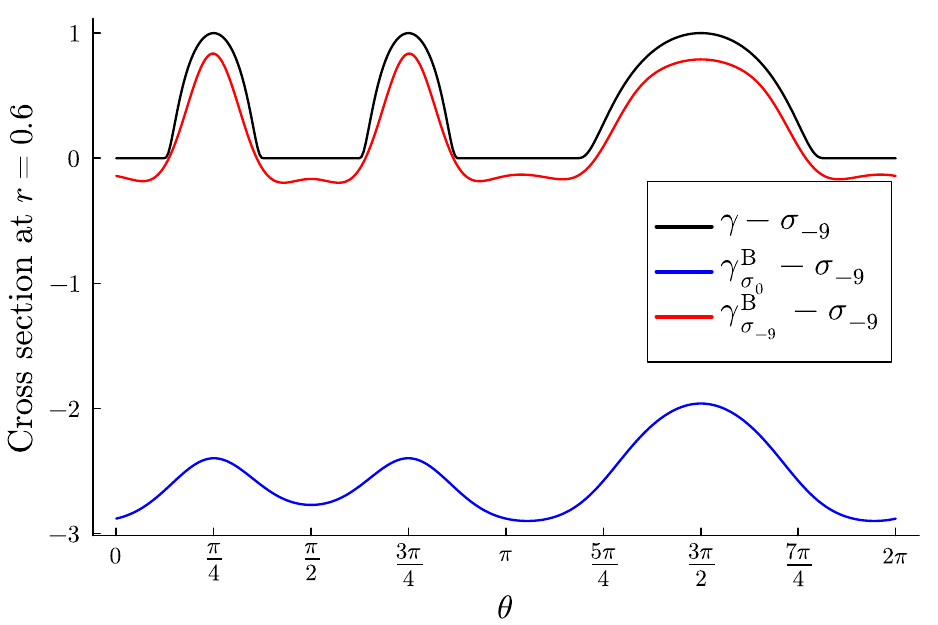}}\hspace{0.05\textwidth}
\subfloat[$\norm{\gamma-\gamma_{\sigma_\kappa}^\mathrm{B}}_{L^p(\ID)}$]{
\begin{tabular}[b]{cccc}
\hline
$\kappa$ \textbackslash\, $p$ & 1 & 2 & $\infty$ \\
\hline
0 & 5.88235 & 3.70735 & 3.43452 \\
-9 & 0.74047 & 0.65851 & 1.46998 \\
\hline
\vspace{1cm}
\end{tabular}
}
\caption{Linearization at $\kappa=-9$. $(N_r,N_\theta)=(50,50)$, $(I,L)=(50,24)$}
\label{Fig:kappa-}
\end{figure}

\newpage
\begin{figure}[t]
\centering
\subfloat[$\gamma$]{\includegraphics[width=0.3\textwidth, trim={1.7cm 0.8cm 1.7cm 0.2cm}, clip]{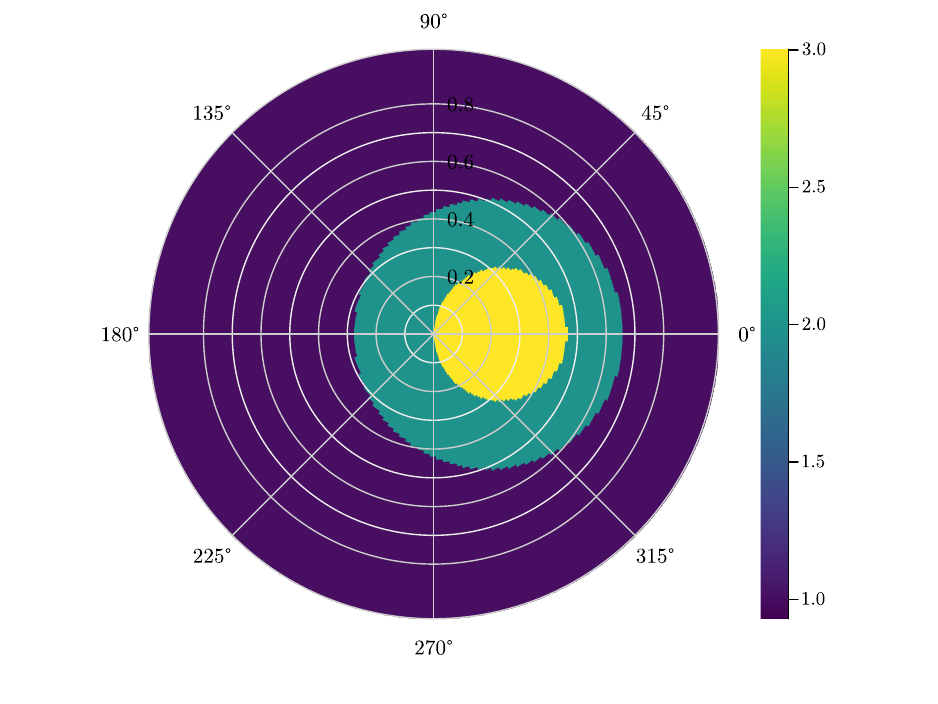}}\hspace{0.04\textwidth}
\subfloat[$\gamma_1^\mathrm{B}=\gamma_0$]{\includegraphics[width=0.3\textwidth, trim={1.7cm 0.8cm 1.7cm 0.2cm}, clip]{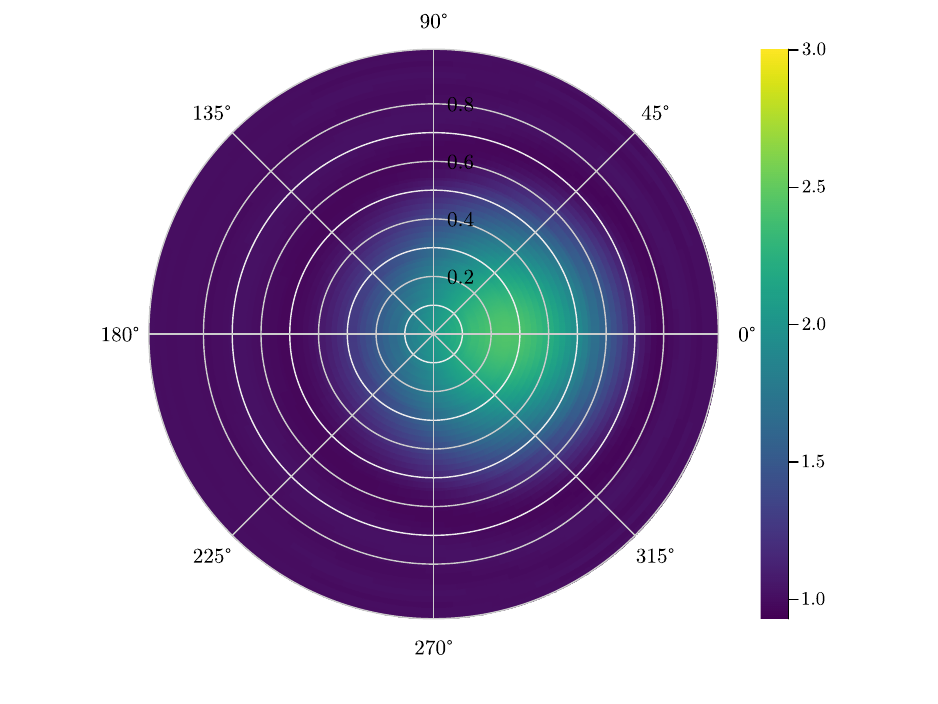}}\hspace{0.04\textwidth}
\subfloat[$\gamma_N$]{\includegraphics[width=0.3\textwidth, trim={1.7cm 0.8cm 1.7cm 0.2cm}, clip]{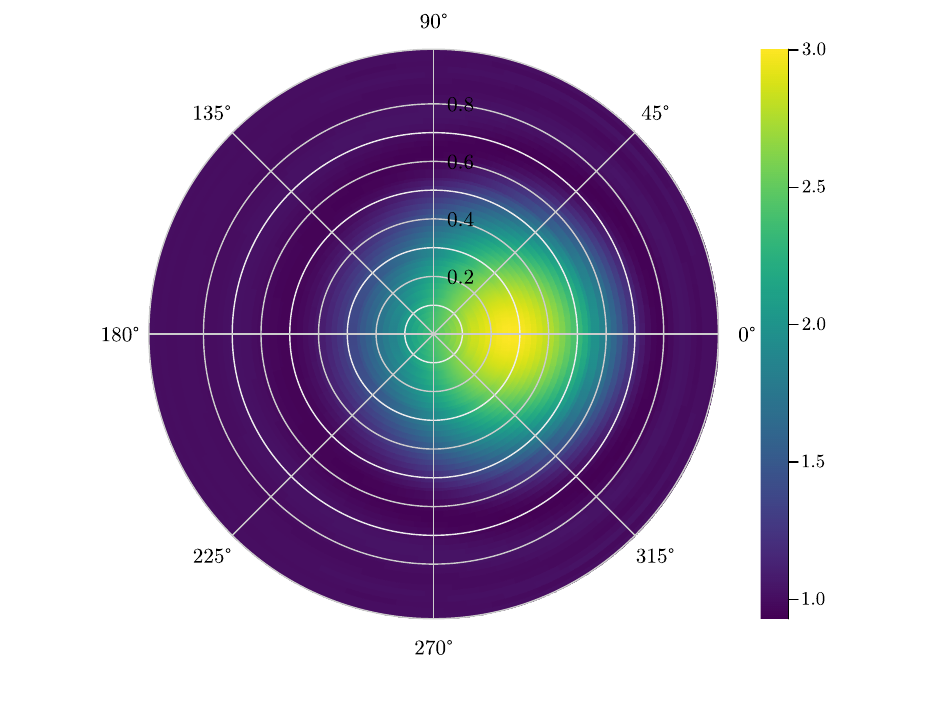}}
\vspace{-0.8\baselineskip}
\subfloat[Angular cross section]{\includegraphics[width=0.45\textwidth]{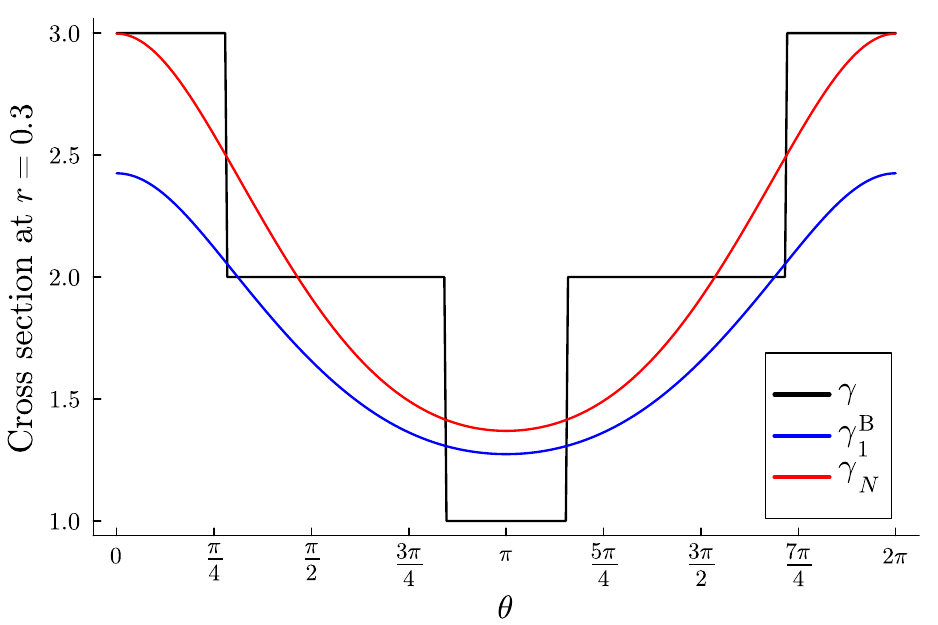}}\hspace{0.05\textwidth}
\subfloat[$\norm{\gamma-\gamma_n}_{L^p(\ID)}$]{
\begin{tabular}[b]{cccc}
\hline
$n$ \textbackslash\, $p$ & 1 & 2 & $\infty$ \\
\hline
0 & 0.41037 & 0.46046 & 0.96906 \\
1 & 0.33719 & 0.35114 & 0.74160 \\
2 & 0.31877 & 0.32043 & 0.65578 \\
3 & 0.31213 & 0.31109 & 0.64217 \\
4 & 0.30938 & 0.30796 & 0.63510 \\
\hline
\vspace{0.5cm}
\end{tabular}
}
\caption{Reconstruction from an explicit DtN map. Iterative scheme run $N=4$ times. $(I,L)=(50,24)$.}
\label{Fig:Explicit1}
\end{figure}

\begin{figure}[H]
\centering
\subfloat[$\gamma$]{\includegraphics[width=0.3\textwidth, trim={1.7cm 0.8cm 1.7cm 0.2cm}, clip]{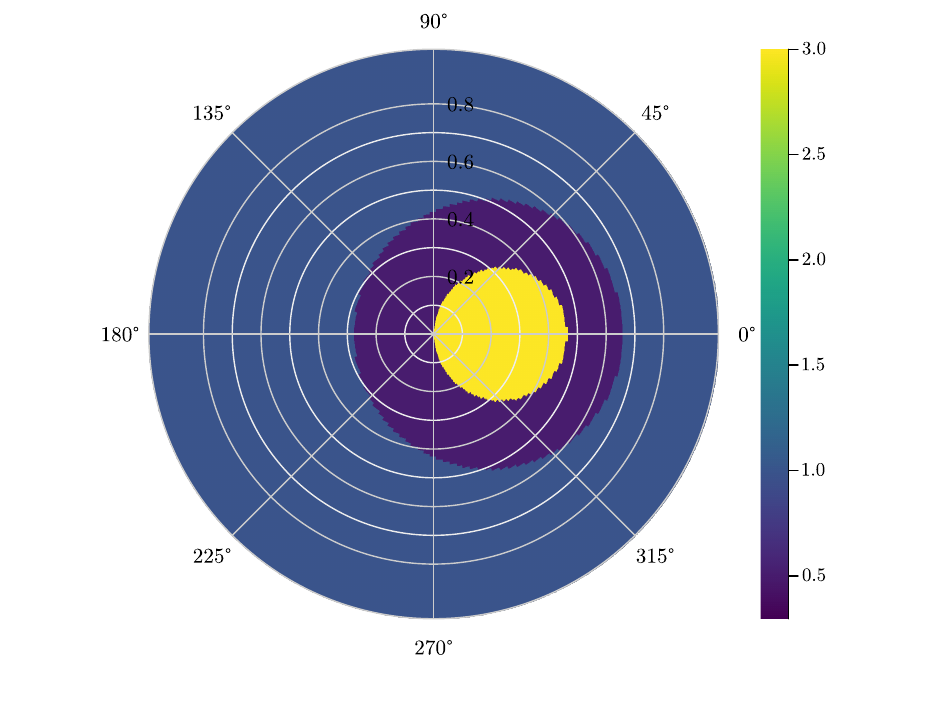}}\hspace{0.04\textwidth}
\subfloat[$\gamma_1^\mathrm{B}=\gamma_0$]{\includegraphics[width=0.3\textwidth, trim={1.7cm 0.8cm 1.7cm 0.2cm}, clip]{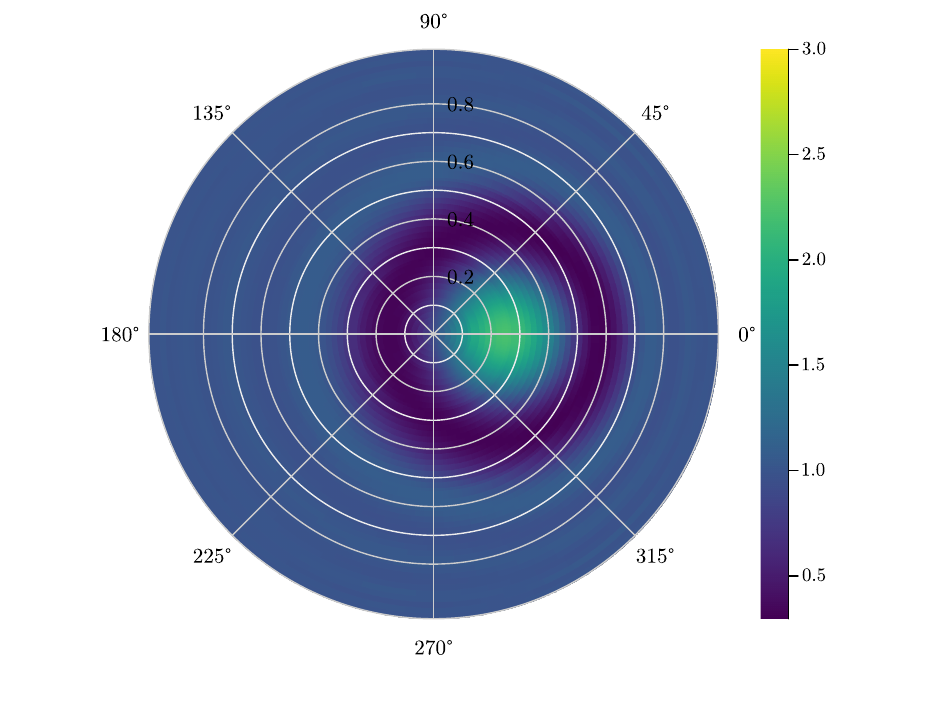}}\hspace{0.04\textwidth}
\subfloat[$\gamma_N$]{\includegraphics[width=0.3\textwidth, trim={1.7cm 0.8cm 1.7cm 0.2cm}, clip]{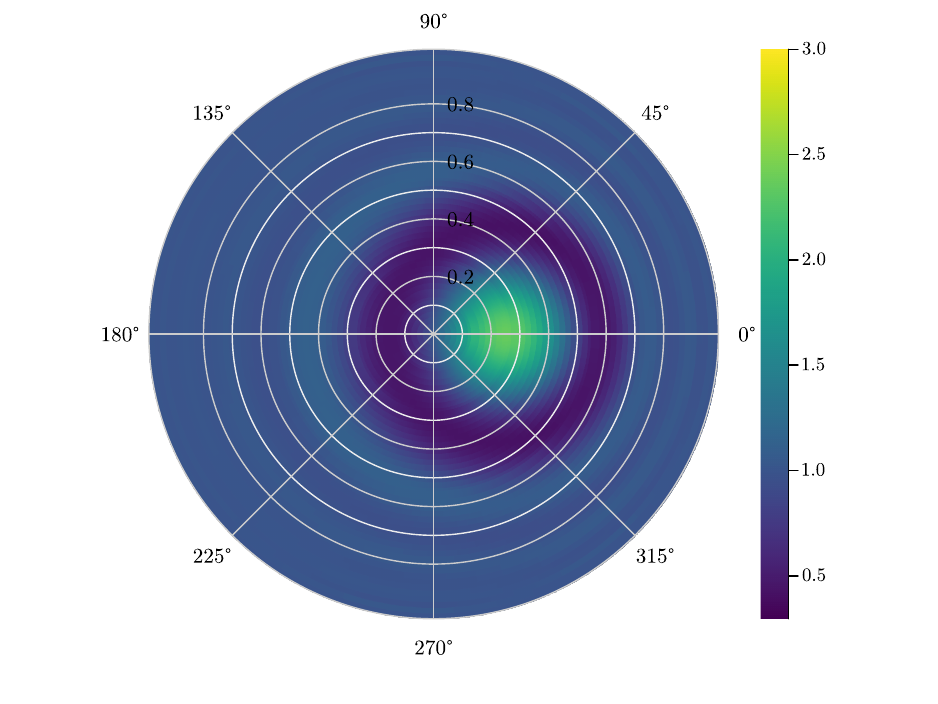}}\hspace{0.04\textwidth}
\vspace{-0.8\baselineskip}
\subfloat[Angular cross section]{\includegraphics[width=0.45\textwidth]{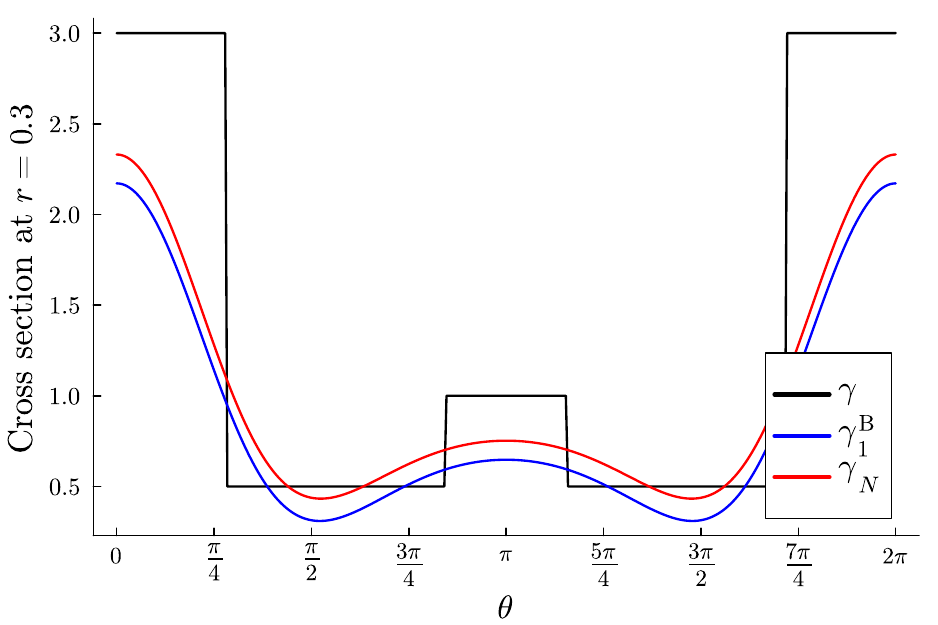}}\hspace{0.05\textwidth}
\subfloat[$\norm{\gamma-\gamma_n}_{L^p(\ID)}$]{
\begin{tabular}[b]{cccc}
\hline
$n$ \textbackslash\, $p$ & 1 & 2 & $\infty$ \\
\hline
0 & 0.43218 & 0.66175 & 2.11403 \\
1 & 0.40207 & 0.60385 & 1.97303 \\
2 & 0.40111 & 0.60776 & 1.98533 \\
3 & 0.40099 & 0.60754 & 1.98497 \\
4 & 0.39807 & 0.60366 & 1.97870 \\
\hline
\vspace{0.5cm}
\end{tabular}
}
\caption{Reconstruction from an explicit DtN map. Iterative scheme run $N=4$ times. $(I,L)=(50,24)$.}
\label{Fig:Explicit2}
\end{figure}

\newpage
\begin{figure}[t]
\centering
\subfloat[$\gamma$]{\includegraphics[width=0.3\textwidth, trim={1.7cm 0.8cm 1.7cm 0.2cm}, clip]{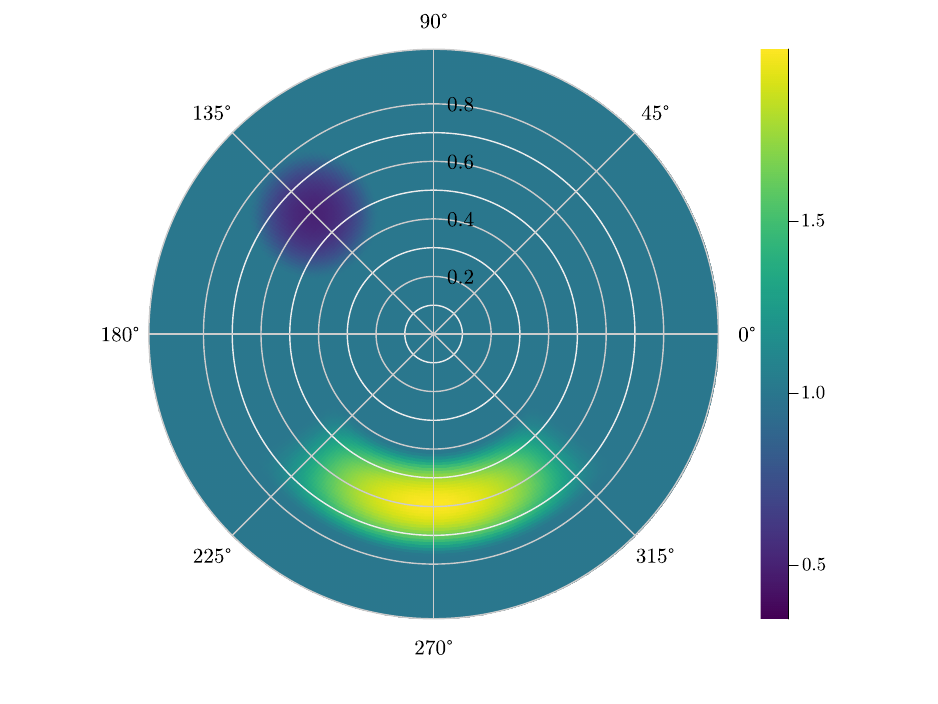}}\hspace{0.1\textwidth}
\subfloat[$\gamma_1^\mathrm{B}$, $\epsilon=0$]{\includegraphics[width=0.3\textwidth, trim={1.7cm 0.8cm 1.7cm 0.2cm}, clip]{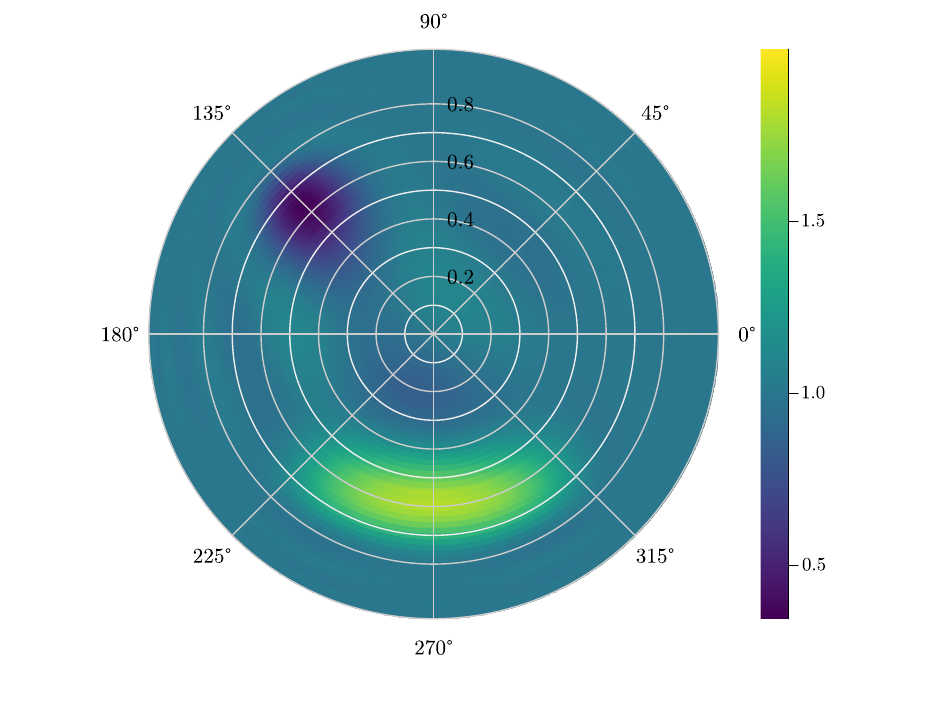}}
\vspace{-0.5\baselineskip}
\subfloat[$\gamma_1^\mathrm{B}$, $\epsilon=10^{-3}$]{\includegraphics[width=0.3\textwidth, trim={1.7cm 0.8cm 1.7cm 0.2cm}, clip]{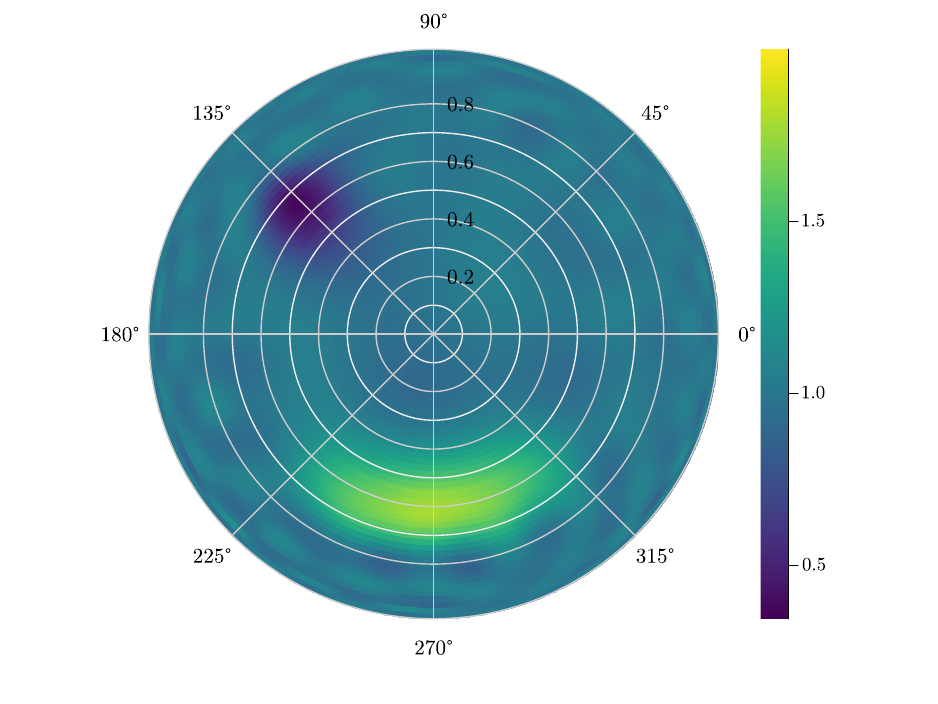}}\hspace{0.1\textwidth}
\subfloat[$\gamma_1^\mathrm{B}$, $\epsilon=10^{-2}$]{\includegraphics[width=0.3\textwidth, trim={1.7cm 0.8cm 1.7cm 0.2cm}, clip]{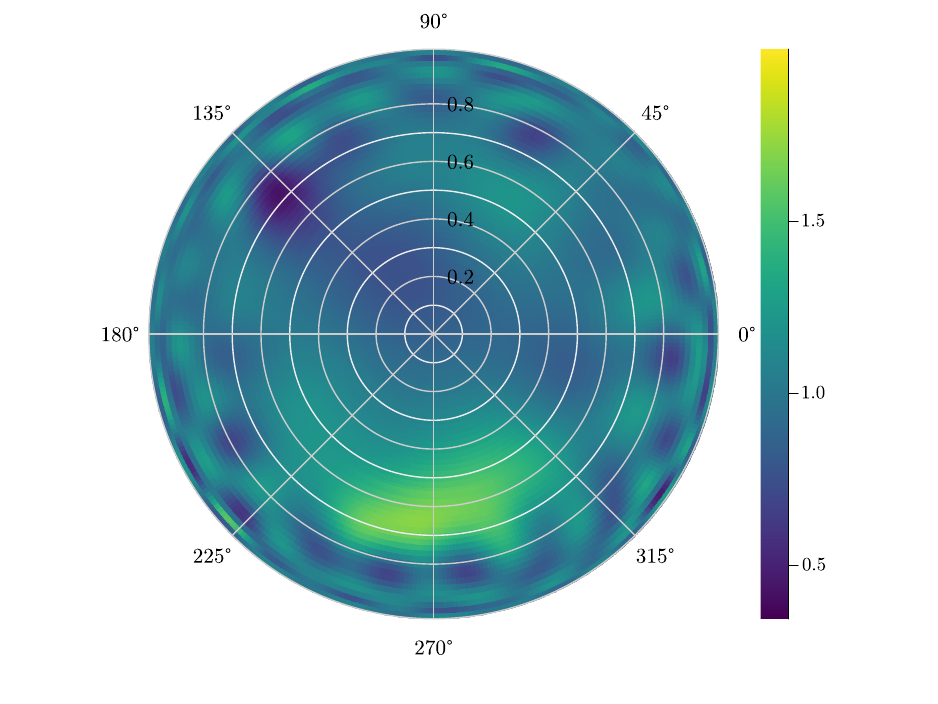}}
\vspace{-0.5\baselineskip}
\subfloat[Angular cross section]{\includegraphics[width=0.45\textwidth]{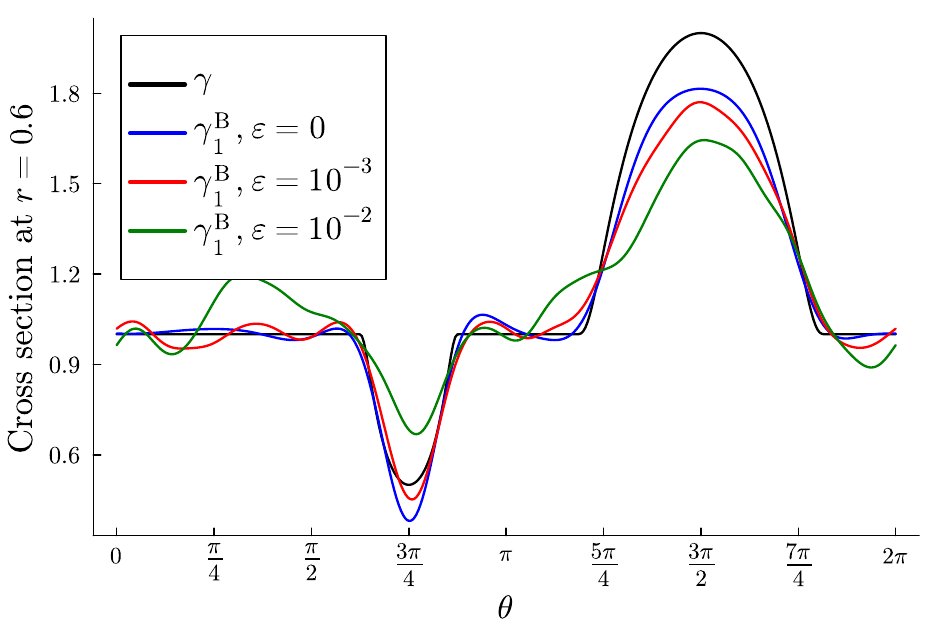}}\hspace{0.05\textwidth}
\subfloat[$\norm{\gamma-\gamma_1^\mathrm{B}}_{L^p(\ID)}$]{
\begin{tabular}[b]{cccc}
\hline
$\epsilon$ \textbackslash\, $p$ & 1 & 2 & $\infty$ \\
\hline
0 & 0.10732 & 0.10094 & 0.24640 \\
$10^{-3}$ & 0.15037 & 0.12096 & 0.31210 \\
$10^{-2}$ & 0.34721 & 0.24992 & 0.55612 \\
\hline
\vspace{1cm}
\end{tabular}
}
\caption{Reconstruction from a DtN map with Gaussian noise. $(N_r,N_\theta)=(50,50)$, $(I,L)=(50,24)$.}
\label{Fig:Noise}
\end{figure}

\clearpage

\bibliographystyle{mynumber}
\bibliography{Refs}

\end{document}